\documentclass[12pt]{amsart}
\usepackage{amsthm}
\usepackage{amsmath}
\usepackage[margin=1.25in]{geometry}
\usepackage{amssymb}
\usepackage{verbatim}
\usepackage{cleveref}
\usepackage{longtable,enumitem,color}
\usepackage{stmaryrd}
\usepackage{graphicx}
\usepackage{amscd}
\usepackage[all]{xy}
\usepackage{enumitem}
	{\providecommand{\noopsort}[1]{} 
\usepackage{chngcntr}\counterwithin{table}{section}

\theoremstyle{plain}
\newtheorem{theorem}{Theorem}[section]
\newtheorem{corollary}[theorem]{Corollary}\newtheorem*{nonumbercorollary}{Corollary}
\newtheorem{lemma}[theorem]{Lemma}
\newtheorem{sublem}[theorem]{Sublemma}

\newtheorem*{addendum}{Addendum}
\newtheorem{proposition}[theorem]{Proposition}

\newtheorem{main}{Theorem}
\newtheorem{maincorollary}[main]{Corollary}
\theoremstyle{definition}
\newtheorem{definition}[theorem]{Definition}

\theoremstyle{remark}

\numberwithin{equation}{section}

\newcommand{\ep}{\epsilon}

\newcommand{\deltad}{x}
\newcommand{\epse}{y}
\newcommand{\Z}{\mathbb{Z}}\newcommand{\Q}{\mathbb{Q}}
\newcommand{\R}{\mathbb{R}}
\newcommand{\s}{\mathbb{S}}

\newcommand{\SO}{\mathrm{SO}}

\newcommand{\SU}{\mathrm{SU}}

\newcommand{\rs}{r}

\newcommand{\Lt}{\mathfrak{t} }

\DeclareMathOperator{\cod}{cod}

\newcommand{\inner}[1]{\left\langle #1 \right\rangle}
\newcommand{\st}{~|~}

\newcommand{\id}{\mathrm{id}}

\renewcommand{\subset}{\subseteq}

\newcommand{\barbeta}{\bar b}

\newcommand{\fQ}{{\mathbb{Q}}} 
                                    
\newcommand{\mult}{{\mu}}                                    

\newcommand{\Sph}{\mathbb{S}}

\newcommand{\gS}{\mathsf{S}}
\newcommand{\gC}{\mathsf{C}}
\newcommand{\gH}{\mathsf{H}}
\newcommand{\gQ}{\mathsf{Q}}

\newcommand{\multtwo}{{\nu}} 

\newcommand{\eps}{\varepsilon}

\newcommand{\RP}{\mathbb{R\mkern1mu P}}
\newcommand{\CP}{\mathbb{C\mkern1mu P}}
\newcommand{\HP}{\mathbb{H\mkern1mu P}}
\newcommand{\CaP}{\mathrm{Ca}\mathbb{\mkern1mu P}^2}

\renewcommand{\R}{{\mathbb{R}}}
\renewcommand{\Z}{{\mathbb{Z}}}

\newcommand{\gE}{\ensuremath{\operatorname{\mathsf{E}}}}
\newcommand{\gF}{\ensuremath{\operatorname{\mathsf{F}}}}
\newcommand{\gA}{{\ensuremath{\operatorname{\mathsf{A}}}}}
\newcommand{\gT}{{\ensuremath{\operatorname{\mathsf{T}}}}}
\newcommand{\gR}{{\ensuremath{\operatorname{\mathsf{R}}}}}
\renewcommand{\H}{\ensuremath{\operatorname{\mathsf{H}}}}

\renewcommand{\SO}{\ensuremath{\operatorname{\mathsf{SO}}}}

\renewcommand{\SU}{\ensuremath{\operatorname{\mathsf{SU}}}}

\renewcommand{\S}{\ensuremath{\operatorname{\mathsf{S}}}}

\newcommand{\gAm}{\gamma}
\newcommand{\bEta}{\beta}
\newcommand{\aLpha}{\alpha}

\def\con#1=#2(#3){#1 \equiv #2 \bmod{#3}}

\newcommand{\ml}{\langle}                     
\newcommand{\mr}{\rangle}                    

\DeclareMathOperator{\pr}{pr}

\DeclareMathOperator{\Ker}{Ker}

\DeclareMathOperator{\spann}{span}

\newcommand{\sectone}{Preliminaries}
\newcommand{\secttwo}{Proof of the  $\S^1$-splitting theorem}
\newcommand{\sectthree}{Proof of a weak version of Theorem \ref{thm:t5}}
\newcommand{\subseceuler}{Positive Euler characteristic}
\newcommand{\sectfour}{Torus actions on $\Sph^h \times \HP^n$}
\newcommand{\sectfourone}{Addendum to Theorem~\ref{thm:t5}}

\newcommand{\subsecbthree}{Vanishing of the third Betti number}
\newcommand{\subsecshxHP}{Topology of fixed-point components in $\Sph^h\times \HP^n$}
\newcommand{\subsecevensphere}{From rational spheres
 to integral spheres in even dimensions}
\newcommand{\sectfive}{Torus actions with a standard two-skeleton}
\newcommand{\subsecmodel}{The graph of a torus action on a manifold of $\CP$ or $\HP$ type}
\newcommand{\subsecstraight}{The combinatorial structure of the graph of the one-skeleton}
\newcommand{\subsecHP}{
Recognizing $\HP^{n/4}$}
\newcommand{\subsecCP}{
Recognizing $\CP^{n/2}$}
\newcommand{\sectfiveCb}{
Proof of Theorem \ref{thm:t8NoCurvature} under the assumptions of b)}
\newcommand{\sectfiveodd}{
Proof of Theorem \ref{thm:t8NoCurvature} in odd dimensions}
\newcommand{\sectsix}{Analyzing the one-skeleton}
\newcommand{\sectseven}{Two results on GKM actions}
\newcommand{\tocsizeone}{\small}
\newcommand{\tocsizeoneone}{\footnotesize}
\newcommand{\tocsizetwo}{\footnotesize}

\title[Splitting of torus representations and applications]{Splitting of torus representations and applications in the Grove symmetry program}
\author{Lee Kennard}
\address{Department of Mathematics, Syracuse University, Syracuse, NY 13244, U.S.A.}\email{ltkennar@syr.edu}
\author{Michael Wiemeler}
\address{Mathematisches Institut, M\"unster University, Einsteinstra{\ss}e 62, 48149 M\"unster}\email{wiemelerm@uni-muenster.de}
\author{Burkhard Wilking}
\address{Mathematisches Institut, M\"unster University, Einsteinstra{\ss}e 62, 48149 M\"unster}\email{wilking@uni-muenster.de}
\date{\today}
\begin{document}


\begin{abstract}
A 1930s conjecture of Hopf states that an even-dimensional compact Riemannian manifold with positive sectional curvature has positive Euler characteristic. We prove this conjecture under the additional assumption that the isometry group has rank at least five. The fundamental new tool used to achieve this is a reduction to, and structural results concerning, a representation theoretic problem involving torus representations all of whose isotropy groups are connected.
\end{abstract}
\maketitle

The classification of positively curved manifolds is a wide open problem that goes back to the origins of Riemannian geometry. On the one hand, very few examples are not diffeomorphic 
 to rank one symmetric spaces
 $\s^n$, $\CP^n$, $\HP^n$, and $\CaP$.
In fact, apart from them, the only other
 compact, simply connected manifolds, 
which are known to admit Riemannian metrics with positive sectional curvature, occur in dimensions 
$6$, $7$, $12$, $13$, and
$24$. On the other hand, few obstructions are known. In fact,
among simply connected, compact manifolds admitting non-negatively curved metrics, there is no known obstruction to admitting positively curved metrics.

In 1992, Karsten Grove suggested one 
should try to study 
the structure of positively (and non-negatively) curved manifolds under symmetry assumptions. He was motivated by a result of Hsiang and Kleiner \cite{HsiangKleiner89} and the case of {\it homogeneous metrics} that was resolved in the 1970s (see \cite{Wallach72,Berard-Bergery76,WilkingZiller18}).
Some of the major results include the 
classification of positively curved manifolds (up to diffeomorphism or homotopy) with large 
symmetry rank \cite{GroveSearle94, Rong02, Wilking03, FangRong05} or low cohomogeneity 
\cite{Wilking06,GroveWilkingZiller08,VerdianiZiller14}. 
 Dessai \cite{Dessai07} showed
the 
vanishing of the indices of certain twisted Dirac operators 
for two-connected positively curved manifolds 
with circle symmetry (see also \cite{Dessai05,Weisskopf17}).

A hope of
the program was also that it should result in motivating new constructions. Successes here include realizations of fundamental groups that answered a 1960s question of Chern (see \cite{Shankar98,Bazaikin99,GroveShankar00,
GroveShankarZiller06}), a new example of a manifold admitting positive curvature that remains the only new example constructed since 1996 (see \cite{Dearricott11,GroveVerdianiZiller11}), and more recently the construction of non-negatively curved metrics on all seven-dimensional homotopy spheres, capping a decades-long search 
 (see \cite{GromollMeyer74,GroveZiller00,GoetteKerinShankar20}).

Our first main result analyzes the fixed-point components of isometric torus actions.

\begin{main}\label{thm:t5}
If $\gT^5$ acts effectively by isometries on a connected, closed, orientable, positively curved Riemannian manifold, then every component of the fixed-point set has the rational cohomology of a sphere or a complex or   quaternionic projective space.
\end{main}

We mention that the six-dimensional Wallach flag $\SU(3)/\gT^2$
is the fixed-point set of an isometric two-torus action 
on the positively curved $24$-dimensional Wallach flag. 
In odd dimensions, the fixed point components are rational spheres, since they are odd dimensional as well.  
In even dimensions, orientability is equivalent to being simply connected and the fixed-point set $F$ of $\gT^5$ is not empty by 
results of Synge and Berger (see Theorem \ref{thm:Berger}). 
 Moreover, the Euler characteristics of $M$ and $F$  coincide.
  Hence, Theorem \ref{thm:t5} implies 
   the following (see Section \ref{sec:PositiveEuler} for further details).

\begin{maincorollary} \label{cor:hopf}
If $M$ is an even-dimensional, connected, closed, positively curved Riemannian manifold whose isometry group has rank at least five, then $\chi(M) > 0$. 

\end{maincorollary} 


A conjecture of Hopf going back to questions asked in 1931 claims that this should hold 
without the additional
assumption of torus symmetry (see \cite{Hopf32}). 
 Previous results required bounds on the torus rank that grew 
linearly or logarithmically in the manifold dimension, and in
   the latter case, one also required additional assumptions on topology of the underlying manifold (see \cite{PuettmannSearle02,RongSu05,SuWang08} and \cite{Kennard13,AmannKennard14}).
%
%

Theorem \ref{thm:t5} provides a rational cohomology computation of each component of the fixed-point set under the assumptions of positive curvature and $\gT^5$ symmetry. We also recover the rational cohomology of $M$ under stronger assumptions. In addition to Theorem \ref{thm:t5}, we need the following tool which does not require curvature 
assumptions.

\begin{main}\label{thm:t8NoCurvature}
Let $M^n$ be a connected, closed, oriented manifold with $H^{i}(M;\Q)=0$ for all odd $i\le \tfrac{n}{2}$. Assume that $M^n$ is endowed with a smooth and effective torus action of
$\gT^d$ with $d\ge 4$ such that at least one of the following holds:
\begin{enumerate}
\item[a)] Every fixed-point component of every codimension three subtorus  is a rational cohomology $\s^m$, $\CP^m$, or $\HP^m$. 
\item[b)]  Every fixed-point component of every codimension two subtorus is a rational cohomology $\s^m$, $\CP^m$, or $\HP^m$, and there is a $n$-dimensional symmetric space $C$ of rank one with $\chi(M^n)=\chi(C)>0$.
\end{enumerate}
 Then $M$ has the  rational cohomology of $\s^n$, $\CP^{\frac n 2}$, $\HP^{\frac n 4}$, or $\CaP$.  
\end{main}

This theorem provides a partial converse to a result of Bredon (Theorem \ref{thm:SinglyGenerated}) stating that fixed-point components of torus actions on manifolds of rational sphere, $\CP$, $\HP$, or $\CaP$ type are of rational sphere, $\CP$, or $\HP$ type.
  It should be clear that for odd $n$ the assumption on the 
  Euler characteristic in b) is never satisfied since it is zero.

 For a  positively curved manifold whose isometry group has rank eight, 
  the condition on the subtori in a) follows by Theorem \ref{thm:t5}. 
  In even dimensions, we can do even better and, as a consequence of Theorem~\ref{thm:z2codim2}, show that the assumptions of b) 
  of Theorem~\ref{thm:t8NoCurvature} are satisfied
  provided one has $\gT^7$-symmetry. More precisely, we have 

\begin{maincorollary}	\label{maincor:t7}	
If $\gT^7$ acts effectively by isometries
on a connected, closed, orientable, positively curved, even-dimensional manifold $M^n$ with 
$H^i(M;\Q)=0$ for all odd $i$, 
then $M^n$ has the rational cohomology of $\s^n$, $\CP^{\frac n 2}$, or $\HP^{\frac n 4}$.

Moreover, 
every involution in $\gT^7$ has at most two fixed-point components, with equality only if 
their dimensions add up to $n-d$, where $d\in \{2,4,n\}$ denotes the degree of the generator of the cohomology of $M^n$.
\end{maincorollary}		


The assumption on the odd Betti numbers is natural in light of the Bott-Grove-Halperin ellipticity conjecture (see \cite{Grove02}). Indeed, if the conjecture is true, then $M$ is rationally elliptic and hence, by Corollary~\ref{cor:hopf}, has vanishing odd Betti numbers. 
In this context, we should also mention that one can combine Theorem~\ref{thm:t5} with a result of Allday \cite{Allday78} to 
see that a rationally elliptic, positively curved, odd-dimensional,  simply connected closed manifold with $\gT^6$-symmetry
satisfies
the identity $-1=\sum_{i=1}^\infty (-1)^i\dim_\Q\bigl(\pi_i(M)\otimes \Q\bigr)$. We also have a variant of Corollary~\ref{maincor:t7}.

\begin{maincorollary}	\label{maincor:t8}	
Let $M^n$ be a closed, odd-dimensional, simply connected and positively curved Riemannian manifold. Suppose
 $H^{i}(M,\Q)=0$ holds for all odd $i\le \tfrac{n-1}{2}$.  
 If $\gT^8$ acts effectively by isometries, then $M^n$ is a rational  sphere.
\end{maincorollary}		
 
 As mentioned earlier, other than spheres, 
  all known examples of odd-dimensional, simply connected, positively curved 
  manifolds occur in dimension 7 and 13. For all of them, 
  the odd Betti numbers vanish in the first half of the cohomology ring. 
  Notice that Corollary~\ref{maincor:t8} immediately follows 
  by combining Theorem~\ref{thm:t8NoCurvature} under the assumptions of a) with Theorem~\ref{thm:t5}.

There are several tools needed to prove these results. Three of the most important previously known tools for us are the Connectedness Lemma and Periodicity Lemma of the third author and the Four-Periodicity Theorem of the first author (see \cite{Wilking03,Kennard13}). Together these results guarantee that, if we have two totally geodesic submanifolds of a positively curved closed manifold $M^n$ intersecting transversely in a manifold $B^b$ with $b\geq \frac n 2$, then the rational cohomology of $M^n$ is four-periodic. This means either $M^n$ is a rational homology sphere or  there is a cohomology class $a\in H^4(M^n;\Q)$ such that multiplication by $a$ induces  isomorphisms in all degrees one could hope for. In order to arrive at a situation where one can apply this theorem, one analyzes the isotropy representation of a torus action at a fixed-point. Here we provide new tool which is a somewhat surprising statement on representations of tori.

\begin{main}[$\S^1$-splitting]\label{thm:S1splitting}
If $\rho:\gT^d \rightarrow \SO(V)$ is a faithful representation with $d \geq 3$, then there exists a one-dimensional subgroup $\gH \subseteq \gT^d$ such that the induced representation $\bar \rho:\gT^d/\H \to \SO(V^{\gH})$ on the fixed-point set of $\gH$ is faithful and isomorphic to a product representation $\bar\rho_1 \oplus \bar\rho_2:\gS^1 \times \gT^{d-2} \to
\SO(V_1) \oplus \SO(V_2)$ 
where $\bar \rho_1$ has the second factor $\gT^{d-2}$ in its kernel and 
$\bar \rho_2$ has the first factor $\gS^1$ in its kernel.
\end{main}

 In particular, the fixed-point sets of $\gS^1$ and $\gT^{d-2}$ intersect transversely in $V^{\gH}$. By iterated applications of Theorem \ref{thm:S1splitting}, we obtain multiple such intersections:

\begin{nonumbercorollary}\label{cor:S1splitting}
If $\gT^{2d+1} \to \SO(V)$ is a faithful representation, then there exists a $d$-dimensional subgroup $\gH \subseteq \gT^{2d+1}$ such that the induced representation $\gT^{2d+1}/\gH \to \SO(V^{\gH})$ is faithful and has exactly $d+1$ non-trivial, pairwise inequivalent, irreducible subrepresentations. In particular, there exist $d+1$ circles in $\gT^{2d+1}/\gH$ whose fixed-point sets in $V^{\gH}$ intersect pairwise transversely.
\end{nonumbercorollary}



The corollary holds, in particular, for $d=0$. In this case, it says that there is a finite subgroup $\gF\subseteq \gS^1$ such that the action of $\gS^1/\gF$ on the fixed-point set of $\gF$ is effective and semi-free. This is easy to see, as  one can take $\gF$ to be any maximal finite isotropy group.
Another outcome of the proof of Theorem \ref{thm:S1splitting} is the following.
 It is not used in the proofs of the main theorems, but it  
 exhibits the {\it dimension-independent rigidity} forced on torus representations 
by the assumption of {\it connected isotropy groups}. An analysis of this assumption is the key point to the proof of Theorem \ref{thm:S1splitting}.

\begin{main}\label{thm:d+1choose2} 
If $\gT^d\rightarrow \SO(V)$ is a faithful representation with only connected isotropy groups, then $\rho$ has at most $\frac{d(d+1)}{2}$ non-trivial, pairwise inequivalent, irreducible subrepresentations. Moreover, equality holds if and only if there is a basis \(e_1,\dots,e_d\) of the dual Lie-algebra of \(\gT^d\) such that the weights of the representation are given by
	\[e_1,\dots,e_d, \quad e_{i}-e_{j}, \quad \text{for} \quad 1\leq i < j \leq d.\]
\end{main}
We remark that the existence of a disconnected isotropy group is equivalent to the existence of a non-trivial finite isotropy group (see Lemma \ref{lem:connectedstab}). The focus on torus representations with connected isotropy groups comes naturally from the inductive machinery used to prove the main theorems. What surprised the authors is the level of rigidity of such representations.\\[1ex] 

\begin{center}{\tocsizeone 
\begin{minipage}{12.8cm}
\begin{itemize}
\item[\ref{sec:Preliminaries}.] \sectone\dotfill {\tocsizeoneone \pageref{sec:Preliminaries}}
\item[
\ref{sec:S1splitting}.] \secttwo  \dotfill {\tocsizeoneone \pageref{sec:S1splitting}}
\item[\ref{sec:t5}.] \sectthree \dotfill {\tocsizeoneone \pageref{sec:t5}}
{\tocsizetwo
\begin{itemize}
\item[\ref{sec:b3}.] \subsecbthree\dotfill \pageref{sec:b3}
\item[\ref{subsec:shxHP}.] \subsecshxHP\dotfill \pageref{subsec:shxHP}
\item[\ref{subsec:evensphere}.] \subsecevensphere\dotfill \pageref{subsec:evensphere}
\item[\ref{sec:PositiveEuler}.] \subseceuler\dotfill \pageref{sec:PositiveEuler}
\end{itemize}
}
\item[\ref{sec:endgame}.] \sectfour \dotfill  {\tocsizeoneone\pageref{sec:endgame}}
{\tocsizetwo
\begin{itemize}
\item[\ref{sec:addendum}.] \sectfourone\dotfill \pageref{sec:addendum}
\end{itemize}}
\item[\ref{sec:recognition}.] \sectfive \dotfill  {\tocsizeoneone \pageref{sec:recognition}}
{\tocsizetwo
\begin{itemize}
\item[\ref{subsec:model}.] \subsecmodel\dotfill \pageref{subsec:model}
\item[\ref{subsec:straight}.] \subsecstraight\dotfill \pageref{subsec:straight}
\item[\ref{subsec:HP}.] \subsecHP\dotfill \pageref{subsec:HP}
\item[\ref{subsec:CP}.] \subsecCP\dotfill \pageref{subsec:CP}
\item[\ref{sec:t8sphere}.] \sectfiveCb\dotfill  \pageref{sec:t8sphere}
\item[\ref{subsec:odd}.] \sectfiveodd\dotfill \pageref{subsec:odd}
\end{itemize}
}
\item[\ref{sec:z2codim2}.] \sectsix \dotfill  {\tocsizeoneone\pageref{sec:z2codim2}}
\item[\ref{sec:GKM}.] \sectseven \dotfill  {\tocsizeoneone\pageref{sec:GKM}}\\
\end{itemize}
\end{minipage}
}
\end{center}
 Section \ref{sec:Preliminaries} reviews the Connectedness Lemma and its consequences, including a description of when transverse intersections result in four-periodic rational cohomology. Section \ref{sec:S1splitting} contains the proof of the $\gS^1$-splitting theorem (Theorem \ref{thm:S1splitting}) and its corollary.
 This section also contains the proof of Theorem \ref{thm:d+1choose2}.

Besides the $\S^1$-splitting theorem, the other new results we need come from computations in equivariant cohomology. First, we show that, under mild symmetry assumptions, an even-dimensional manifold with four-periodic rational cohomology has vanishing odd Betti numbers (see Lemma \ref{lem:b3}). This result, which we call the $b_3$ Lemma, is independent of the $\S^1$-splitting and 
 is the key point in the proof of Proposition~\ref{pro:4periodicN}, 
 which is a weak version of Theorem~\ref{thm:t5} and which already
implies Corollary~\ref{cor:hopf}. 
 The $b_3$ Lemma can also be seen as the main reason
as to why we get results in all even dimensions, not just in those divisible by four (cf. \cite[Theorem A]{Kennard13}).

Proposition~\ref{pro:4periodicN} 
says, in particular, either Theorem \ref{thm:t5} holds 
or we can find a special $\gT^2$-action on a positively curved manifold with the rational cohomology
of $\s^2 \times \HP^m$ or $\s^3 \times \HP^m$. 
The actions are analyzed without any curvature assumptions 
in Section~\ref{sec:endgame}, where we essentially show
 that they have linear models. 
This global approach requires  a bit of work and is motivated in part by work of Hsiang and Su \cite{HsiangSu75}, who examined smooth torus actions on rational cohomology $\HP^m$.
  This in turn allows us to rule out this potential exception
  to Theorem A from Proposition~\ref{pro:4periodicN}.

Once Theorem \ref{thm:t5} is proved, we analyze larger torus actions that satisfy the condition of equivariant formality, that is, the natural 
map from equivariant cohomology to cohomology being surjective. 
Since the equivariant cohomology of $M$ determines the ordinary cohomology of $M$, it suffices to recover the image 
of the injective map from the equivariant cohomology of $M$ 
to the equivariant cohomology of the fixed-point set.
The Chang-Skjelbred Lemma (see \cite{ChangSkjelbred74}) implies that this image is also given by the image of the equivariant 
cohomology of the one-skeleton. We recall that the one-skeleton 
consists of those points whose orbits have dimension $\le 1$. 
To determine this image, we assign in Section~\ref{sec:recognition} a graph to the one-skeleton that, together with its  labels (weights), completely determines the isotropy representation of the torus at the various fixed-point components up to rescaling of the weights. 
%
  We will first see that the combinatorial structure of the unlabeled graph
  is completely determined if we know 
  the number of fixed-point components, the dimension $n$, and 
  the Euler characteristic. In the proof
  of Theorem~\ref{thm:t8NoCurvature} under the assumptions of b), 
  we first recover the labeling of the graph. This will allow us 
  to explicitly determine the  above  image 
of the equivariant cohomology of $M$.
  The whole strategy of the proof is again very much motivated by 
  work of \cite{HsiangSu75} on torus actions on rational cohomology $\HP^n$. 
   A few aspects are also similar to work by Goertsches and the second author \cite{GoertschesWiemeler15} who 
  proved a recognition theorem under much more restrictive conditions.
  
    In Section~\ref{subsec:odd}, we will prove 
  Theorem~\ref{thm:t8NoCurvature} in odd dimensions. 
  The main challenge here is that the assumptions do not allow to
   conclude that the torus action is equivariantly formal, and it requires 
   some work to find a suitable setting in which the 
   Chang-Skjelbred Lemma can be applied.

In Section~\ref{sec:z2codim2}, we prove Corollary~\ref{maincor:t7}.
As in the section before we work without curvature assumptions but
 now we assume  that every (possibly disconnected) codimension two subgroup only has fixed-point components 
whose rational cohomology has $\Sph$, $\CP$ or $\HP$ 
type.  We  show that the $\Z_2$-weights of the isotropy representation obey some very simple rules that help to 
recover the Euler characteristic, see Theorem~\ref{thm:z2codim2}. 

In the last section, we finally prove Theorem~\ref{thm:t8NoCurvature} 
under the assumptions of a). By work done 
in Section~\ref{sec:recognition}, we only need to deal 
with the case that the action of the torus  satisfies the GKM-condition, 
that is, every fixed-point component of a codimension one torus 
having dimension at most two.  This is fairly close to the assumptions of \cite{GoertschesWiemeler15}, and the main idea 
is to adapt their techniques to 
recover the Euler characteristic in this case as well.



\subsection*{Acknowledgements} This material is based upon work supported by the National Science Foundation under Grant DMS-1440140 while the first and third authors were in residence at the MSRI in Berkeley, California, during the Spring 2016 semester. The first author is partially supported by NSF Grant DMS-2005280. The second and third authors  were supported by the Deutsche Forschungsgemeinschaft (DFG, German Research Foundation) – Project-ID 427320536 – SFB 1442, as well as under Germany’s Cluster of  Excellence Strategy EXC 2044 - 390685587, Mathematics Münster: Dynamics–Geometry–Structure.

\section{\sectone}\label{sec:Preliminaries}
Throughout the paper, we take cohomology groups with rational coefficients  whenever the coefficient ring is not specified.
The primary tool we use for deriving topological information from the existence of isometric group actions on positively curved manifolds is the following, proved by the third author \cite[Theorem 2.1]{Wilking03}.

\begin{theorem}[Connectedness Lemma]
Let $M^n$ be a positively curved, connected, closed Riemannian manifold.
	\begin{enumerate}
	\item If $N^{n-k} \subseteq M$ is a closed, totally geodesic, embedded submanifold, then the inclusion is $(n-2k+1)$--connected. 
	\item If $N^{n-k} \subseteq M$ is a fixed-point component of an isometric action on $M$ with principal orbits of dimension $\delta$, then the inclusion is $(n-2k+1+\delta)$--connected.
	\item If $N_i^{n-k_i} \to M$ are closed, totally geodesic, embedded submanifolds with $k_1 \leq k_2$, then the inclusion $N_1\cap N_2 \to N_2$ is $(n-k_1-k_2)$--connected.
	\end{enumerate}
\end{theorem}


Notice that fixed-point components of isometric torus actions on compact, positively curved manifolds satisfy the assumptions of the Connectedness Lemma. In addition, these components are oriented (see Theorem \ref{thm:FPSstructure} below). The third author observed that a highly connected inclusion of closed, orientable manifolds implies a certain type of periodicity in cohomology (see \cite[Lemma 2.2]{Wilking03}).

\begin{lemma}[Periodicity Lemma]\label{lem:PeriodicityLemma}
Assume $N^{n-k} \subseteq M^n$ is a $(n-k-l)$--connected inclusion of closed, orientable manifolds, where $n - k - 2l > 0$. There exists $e \in H^k(M;\Z)$ such that the maps $ H^i(M;\Z) \to H^{i+k}(M;\Z)$, $x \mapsto ex$ are surjective for $l \leq i < n - k - l$ and injective for $l < i \leq n - k - l$.
\end{lemma}

This lemma also holds with integer coefficients replaced by coefficients in any field. When $l = 0$, the inclusion $N \subseteq M$ is $(\dim N)$--connected and hence the integral cohomology of $M$ is periodic in the following sense:

\begin{definition}[Periodic cohomology]
For a ring $R$ and a closed $n$-manifold $M$, we say that $H^*(M;R)$ is $k$-periodic if there exists $e \in H^k(M;R)$ such that the maps $\cup e\colon H^i(M;R) \to H^{i+k}(M;R)$ are surjections for $0 \leq i < n-k$ and injections for $0 < i \leq n-k$.
\end{definition}

Note that when $R$ is a field and $k \leq \frac n 2$, either $H^*(M;R) \cong H^*(\s^n;R)$ or  the map $\cup e\colon H^i(M;R) \to H^{i+k}(M;R)$ 
is an isomorphism for all $0 \leq i \leq n-k$. This condition was studied by the first author, who proved the following \cite[Theorem C]{Kennard13}.

\begin{theorem}[Four-Periodicity Theorem]
Let $M^n$ be a closed, simply connected manifold. If $M$ has $k$--periodic integral cohomology for some $k \leq \frac{n}{3}$, then $M$ has four-periodic rational cohomology. 
\end{theorem}

We work frequently with the condition of four-periodic rational cohomology, so we explain next what exactly it implies.

\begin{proposition}\label{pro:4periodic}
Assume $M^n$ is a connected, closed, and orientable manifold with four-periodic rational cohomology and $b_1(M) = 0$. One of the following occurs:
	\begin{enumerate}
	\item $M$ has the rational cohomology of $\s^n$, $\CP^{\frac n 2}$, or $\HP^{\frac n 4}$.
	\item $M$ has the rational cohomology of $\s^2 \times \HP^{\frac{n-2}{4}}$ or $\s^3 \times \HP^{\frac{n-3}{4}}$.
	\item $H^*(M;\Q)$ is generated as an algebra by some $c \in H^2(M;\Q)$, $a \in H^4(M;\Q)$, and $b_1,\ldots,b_s \in H^3(M;\Q)$ for some even $s > 0$. A complete set of relations is given by $c^2 = 0$, $a^{\frac{n-2}{4} + 1} = 0$, $c b_i = 0$ for all $i$, and $b_i b_j = m_{ij} a c$ for some non-degenerate, skew-symmetric matrix $(m_{ij})$. In particular, $n \equiv 2 \bmod 4$.
	\end{enumerate}
\end{proposition}

\begin{proof}
 We use rational coefficients in the proof. 
Let $a \in H^4(M)$ be an element inducing periodicity. If $a = 0$, then four-periodicity implies that $H^*(M) \cong H^*(\s^n)$. So, assume $a \neq 0$ or equivalently 
$H^4(M)\cong \Q$. If $a$ can be written as a product 
$a=c_1\cdot c_2$ with $c_i\in H^2(M)$, then $\cup c_1$ 
induces periodicity as well.  In fact, this 
follows form an argument similar 
to the proof of Corollary~\ref{cor:transversal} below but also by 
 (\cite[Lemma 1.2]{Kennard13}). But then $M^n$ has two-periodic 
 cohomology and clearly $H^*(M^n)\cong \CP^{n/2}$.  
 Thus we may assume in addition that $a$ is not a product of degree 2 elements.

If $n \equiv 1 \bmod 4$, then $a^{\frac{n-1}{4}} \in H^{n-1}(M) \cong H_1(M) \cong 0$. By four-periodicity, it follows that $a = 0$, which is a contradiction. 

If $n \equiv 0 \bmod 4$, then $H^i(M) = 0$ for all odd $i$ by four-periodicity and Poncar\'e duality. If $H^2(M) = 0$, it follows immediately from four-periodicity that $H^*(M) \cong H^*(\HP^{\frac n 4})$. If instead $H^2(M) \neq 0$, then pick some non-zero $u \in H^2(M)$. Using Poincar\'e duality, choose $v \in H^{n-2}(M)$ such that $uv \neq 0$. Since $v=a^hu_2$ with $u_2\in H^2(M)$, we deduce
$uu_2\neq 0$ and $a$ is a product of degree 2 elements.

If $n \equiv 3 \bmod 4$, then $H^{1+4i}(M) = 0$ and $H^{2+4i}(M) = 0$ for all $i$ by four-periodicity and Poincar\'e duality. Similarly, $H^{4i}(M) \cong H^{4i+3}(M) \cong \Q$. Moreover, four-periodicity immediately implies that $H^*(M) \cong H^*(\s^3 \times \HP^{\frac{n-3}{4}})$.

Finally, assume $n \equiv 2 \bmod 4$. Then $H^2(M)\cong \Q$ by Poincar\'e duality. If $c \in H^2(M)$ be a generator, then $c^2=0$,
since $a$ is not a product of degree 2 elements.  
Thus $H^*(M)$ is generated in even degrees by $a$ and $c$.
 Assume now without loss of generality that $H^3(M) \neq 0$, and pick a basis $b_1,\ldots,b_s \in H^3(M)$. Note that $c b_i = 0$ for all $i$ since $H^5(M) \cong H^1(M) \cong 0$. As for the remaining relations, note that Poincar\'e duality and four-periodicity imply that the composition
	\[H^3(M) \times H^3(M) \longrightarrow H^6(M) \longrightarrow H^n(M)\]
given by the cup product followed by multiplication by $a^{\frac{n-6}{4}}$ is a skew-symmetric, non-degenerate map. Since in addition four-periodicity implies that $H^6(M)$ is generated by $ac$, the claimed relations follow.
\end{proof}

An important consequence is the following theorem refining \cite[Theorem 4.2]{Kennard13}.

\begin{corollary}\label{cor:transversal} 
Let $M^n$ be a connected, closed, positively curved, orientable manifold. 
 Suppose $N_1^{n-k_1}$ and $N_2^{n-k_2}$ are transversely intersecting submanifolds.  Assume that each one is a fixed-point component 
of an isometric circle action.
 Let $k_1 \leq k_2$.
	\begin{enumerate}
	\item If $2k_1 + k_2 \leq n$, then $N_2$ 
	has $k_1$-periodic integral cohomology.
	\item If $3k_1 + k_2 \leq n$, then  $N_2$ and $N_1\cap N_2$ 
	 have four-periodic rational cohomology.
	\item	If $2k_1 + 2k_2 \leq n$ or $\max\{3k_1+k_2,k_1+2k_2\}\le n$,
	 then $M$ and $N_1$ 
	  have four-periodic rational cohomology.
	\end{enumerate}
The conclusions also hold for the universal covers of the manifolds.
\end{corollary}

\begin{proof} The manifold $N_i$ is orientable 
since the circle action induces a complex structure on its
normal bundle for $i=1,2$.
For the first claim, let $\pi:\tilde N_2 \to N_2$ be the universal cover. Then $\pi^{-1}(N_1 \cap N_2)$ is connected  and simply connected by the Connectedness Lemma. In particular, $\tilde N_2$ and $\pi^{-1}(N_1\cap N_2)$ are orientable. By the Connectedness Lemma, $N_1 \cap N_2 \to N_2$ is $(n-k_1-k_2)$-connected. Since coverings induce isomorphisms on homotopy groups above the fundamental group, the inclusion $\pi^{-1}(N_1 \cap N_2) \to \tilde N_2$ is also $(n-k_1-k_2)$-connected. Hence, $\tilde N_2$ is $k_1$-periodic by the Periodicity Lemma. 

Claim (2) follows for the universal covers immediately from the first claim and the Four-Periodicity Theorem. To get them for the manifolds themselves, we only have to consider odd 
dimensions, since even-dimensional, orientable, positively curved manifolds are simply connected. Being an odd-dimensional 
positively curved manifold, 
$L:=N_1\cap N_2$ is orientable. Put  $m=\dim(N_2)$ and $k=k_1$. What 
we know and all we need is that
$L^{m-k}\to N_2^{m}$ is $(m-k)$-connected with $k\le m/3$. 
As before, the inclusion map 
of $\tilde L\to  \tilde N_2$ is then $(m-k)$-connected.  
If $\tilde N_2$ is a rational homology sphere, so is $N_2$ 
and thereby $L$. By Proposition~\ref{pro:4periodic},  
we may assume $m=4l+3$ and
 $H^*(\tilde N_2)\cong H^*(\Sph^3\times \HP^l)$.
The deck transformation group preserves the orientation. 
If the finite deck transformation group acts trivially on the rational cohomology of $\tilde N_2$,
then $N_2$ has the same rational cohomology as $\tilde N_2$. 
Thus we may assume that there is a deck transformation 
$\sigma$ whose induced map on $H^4(\tilde N)$ is $-\id$. One can then assume that the order of 
$\sigma$ is a power of two.
We plan to get a contradiction by looking at the quotient $\bar N= \tilde N/\ml \sigma\mr$
of the cyclic group.
The rational cohomology of $\bar N$ is $8$-periodic, and it is
 generated by  one element in degree $8$ 
 and another element that has either degree $3$ or $7$. 
Using $k\le m/3$ we see  that $\bar N$ is not a rational 
cohomology sphere.
On the other hand,  $\tilde L/\ml \sigma\mr \to \bar N$ 
is $(m-k)$-connected,  and thus  
$\bar N$ has $k$-periodic integral cohomology.
The same holds with $\Z_3$ coefficients and
by the $\Z_3$-version of the Four-Periodicity Theorem in \cite{Kennard13}, 
the $\Z_3$-cohomology is $k'$-periodic where $k'$ is the greatest
common divisor of $4\cdot 3^m$ and $k$. 
Moreover, using
$H^1(\bar N,\Z_3)=0$, one can show as
 in proof of \cite[Theorem C]{Kennard13} 
that the rational cohomology of $\bar N$ is $k'$-periodic, too.
This is a contradiction since $H^{k'}(\bar N)=0$.

For the third claim, note that $2k_1 + 2k_2 \leq n$ implies $\max\{3k_1+k_2,k_1+2k_2\}\le n$.  
\hspace*{0em}From the Connectedness Lemma, we deduce that all three manifolds 
have the same fundamental group. 
We may assume that $M$ is simply connected, by passing to its universal cover.  But we can also run the following argument on $M$ 
itself.
By the second claim $N_2$ has four-periodic rational cohomology.
Since the inclusion $N_2\rightarrow M$ is $k_1+2$ connected, the same holds for the 
cohomology of $M$ up to degree $k_1+2$.
 The Periodicity Lemma shows also that there is an  $e\in H^{k_1}(M)$
 such that $\cup e\colon H^i(M) \rightarrow H^{i+k_1}(M)$ is surjective 
 for $i=k_1-2$, injective for $i=n-2k_1+2$, and bijective in between. 
 We can assume by the former statement that
 $e=  a  u  \neq 0$ with  $a\in H^4(M)$ 
 and $u\in H^{k_1-4}(M)$. 
 Since $n-2k_1+2\ge n/2 +2$, 
  it suffices to prove that
  $$\cup a \colon H^i(M) \rightarrow H^{i+4}(M)$$
   remains an isomorphism for
 $i=k_1-2,\ldots,n-2k_1-3$ and is injective for $i=n-2k_1-2$. 
 The injectivity is known from before in degree $i=k_1-2$ 
 and follows from the injectivity of $\cup e$ in the other degrees.
 To prove surjectivity, one uses the surjectivity of 
 $\cup e\colon H^{i}(M)\rightarrow H^{i+k_1}(M)$
 combined with the injectivity of 
 $\cup u\colon H^{i+4}(M)\rightarrow H^{i+k_1}(M)$. The latter follows from 
 the injectivity of $\cup e$ in degree $i+4$.
\end{proof}

The results above imply four-periodic cohomology in the presence of certain transverse intersections of totally geodesic submanifolds. On one hand, such submanifolds do not exist for generic metrics (see \cite{Ebin68, MurphyWilhelm19}). On the other hand, fixed-point components of isometries are abound in positive curvature in the presence of symmetries. The starting point is always the following result, which is due to Berger \cite{Berger61} in even dimensions and Sugahara \cite{Sugahara82} in odd dimensions (cf. \cite{GroveSearle94}).

\begin{theorem}[Berger]\label{thm:Berger}
If a torus $\gT^d$ acts isometrically on a closed, positively curved Riemannian manifold, then $\gT^d$ has a fixed point in even dimensions and some $\gT^{d-1} \subseteq \gT^d$ has a fixed point in odd dimensions.
\end{theorem}

If the torus $\gT^d$ has large rank, this theorem implies the existence of many fixed-point sets for the subgroups of $\gT^d$.  We mention some basic properties:

\begin{theorem}
\label{thm:FPSstructure}
Let $\gT^d$ act isometrically on a closed, oriented, positively curved Riemannian manifold $M$. 
If $N$ is the fixed-point component of a subgroup $\gH \subseteq \gT^d$, then
	\begin{enumerate}
		\item the action of $\gT^d$ restricts to $N$, and
	\item $N$ is a closed, orientable, totally geodesic submanifold of even codimension.
\end{enumerate}	
\end{theorem}

All statements except for the fact that $N$ is orientable are well known, but we include a proof for completeness.

\begin{proof}
The first statement is immediate since $\gT^d$ is abelian and connected. In the second statement, we simply recall  that fixed-point components of subgroups of the isometry group are closed, totally geodesic submanifolds. Since additionally $\gH$ is a subgroup of a connected subgroup, namely $\gT^d$, of the isometry group, each element of $\gH$ acts by orientation-preserving diffeomorphisms. In particular, the isotropy representation at a point $p \in N$ gives a map $\gH \to \SO(T_p^\perp N)$. The codimension of $N$ is the dimension of this representation, so it follows that $N$ has even codimension.

We can find a \(\gS^1\subset \gT^d\) such that $\gH'=\gH\cap \gS^1$ 
is non-trivial. If we show that the fixed-point component of 
$\gH'$  at $N$ is orientable, then the orientablity of $N$ 
follows by induction on the dimension of $M$. 
Thus we may assume that $\gS^1=\gT^d$.

  If \(\gS^1\) acts trivially on \(N\), then the normal bundle of \(N\) has a complex structure and the claim follows. Otherwise, there is an almost effective \(\gS^1\)-action on \(N\) 
  with the property that any subgroup of $\gS^1$ 
  has only fixed-point components of even codimension in $N$.
Hence, orientability of \(N\) follows from the following lemma.
\end{proof}

\begin{lemma}
  Let \(N\) be a positively curved, connected manifold with an isometric, almost effective \(\gS^1\)-action. If the codimensions of all fixed point components of all subgroups of $\gS^1$ are even, then $N$ is orientable.
\end{lemma}

\begin{proof}
  By Synge's theorem, we can assume that $N$ has even dimension.
  Then \(N/S^1\) is a positively curved Alexandrov space of odd dimension.
  Moreover, it is locally orientable because, for each \(x\in N\), the action of the isotropy group \(\gS^1_x\) on a (orientable) neighborhood of \(x\) in \(N\) is orientation-preserving by the codimension assumption.

  Therefore, it follows from Petrunin's generalization \cite{Petrunin98}  of Synge's theorem to Alexandrov spaces that \(N/\gS^1\) is orientable.
  In particular, the regular stratum of \(N/\gS^1\) is orientable.
  Now the union of principal orbits in \(N\) is a principal \(\gS^1\)-bundle over this regular stratum.
  Hence, it is orientable.
  But this implies that \(N\) is orientable, because all the singular strata of the \(\gS^1\)-action on \(N\) have codimension at least two.
\end{proof}

The following basic observation proves to be very effective in the analysis of the fixed-point sets.

\begin{lemma}\label{lem:connectedstab} 
Suppose a compact abelian Lie group $\gA$ acts effectively on a closed manifold $M$. 
	\begin{enumerate}[label=(\alph*)]
	\item If $\gH$ is  a disconnected isotropy group  of a point 
	$q$, then we can find a non-trivial finite isotropy group
	 near $q$. 
	\item If $N$ is a fixed-point component of $\gA$, then there 
	is a finite group $\gF\subset \gA$  
	such that $\gA/\gF$ is connected and acts effectively 
	on the fixed-point component $M'$ of $\gF$ at $N$ and
	without disconnected 
	isotropy groups near $N$. 
	\end{enumerate}
\end{lemma}
 
 One use of the lemma is to see that the fixed-point component $N$ of a five-dimensional  compact abelian Lie group is also the fixed-point component of a $5$-torus inside 
 a smaller submanifold. We use this in combination with Theorem~\ref{thm:t5} in Section~\ref{sec:z2codim2}.


\begin{proof} Without loss of generality,  $\gH$ has positive dimension. 
We choose an irreducible subrepresentation $U\subset T_qM$ 
on which the identity component of $\gH$ acts non-trivially. 
The isotropy group of $v\in U\setminus \{0\}$ 
is a subgroup $\gR\subset \gH$ 
with at least as many connected components
 as $\gH$ and $\dim(\gR)=\dim(\gH)-1$.
  Since $\gR$ 
is the isotropy group of $\exp(tv)$ for small $t>0$, a) follows by induction over $\dim(\gH)$.
 
 We choose a point $q$ near $N$  whose isotropy 
 group  $\gF$  is finite and maximal among all finite 
 isotropy groups of points in a tubular neighborhood of $N$.
   The induced action of $\gA/\gF$ on the fixed-point component $M'$ of $\gF$ at $N$ has no finite 
 isotropy groups in a neighborhood of $N$.
  The first statement implies 
  all isotropy groups of the action by $\gA/\gF$ on $M'$
  are connected near $N$. 
 Since $\gA/\gF$ itself is the isotropy group of $p\in N$, 
 the group $\gA/\gF$ is connected as well.
\end{proof}

We close the section with some basic remarks on equivariant cohomology. 
If a torus $\gT^d$ acts smoothly on a manifold $M^n$, 
the equivariant cohomology $H^*_{\gT^d}(M)$ is defined as 
the ordinary cohomology of the Borel construction $M\times_{\gT^d}E\gT^d$. 
There is a fibration 
\begin{eqnarray}\label{eq:borelfib}
M\to M\times_{\gT^d}E\gT^d\to B\gT^d
\end{eqnarray}
that induces a canonical identification of the equivariant cohomology $H^*_{\gT^d}(p)$ 
of a fixed point $p$ with $H^*(B\gT^d)\cong \Q[t_1,\ldots,t_d]$. 
Moreover, \eqref{eq:borelfib} turns $H^*_{\gT^d}(M)$ into a module 
over the ring $H^*(B\gT^d)$. The action is called {\it equivariantly formal} if the  
map $H^*_{\gT^d}(M)\to H^*(M)$ is surjective. 
In that case, we can find 
lifts $\alpha_1,\ldots,\alpha_s\in H^*_{\gT^d}(M)$ of a $\Q$-basis $a_1,\ldots,a_s\in H^*(M)$, and they form basis of 
$H^*_{\gT^d}(M)$ as a $H^*(B\gT^d)$-module. 

\begin{lemma}\label{lem:injective}
 Suppose a torus $\gT^d$ acts smoothly on a connected
 manifold $M^n$. Let $b\ge 0$ be minimal such that there are orbits of dimension $b$. 
Let $F$ be the union of all orbits of dimension $b$. 
Then $H^k_{\gT^d}(M,M\setminus F)\to H^k_{\gT^d}(M)$ is injective for all $k$. 
If $d=b+1$, then the map is surjective in all degrees $\ge n$.
\end{lemma}

\begin{proof}  For the injectivity, we plan to show that the composition 
of the map with the pullback map to $H^k(F)$ is injective. 
To see this, one first observes that by excision  
$H^k_{\gT^d}(M,M\setminus F)\cong H^k_{\gT^d}(U,U\setminus F)$ for a tubular neighborhood $U$ of $F$. 
In particular, we may assume that $F$ is connected and $U$ is diffeomorphic to the normal bundle of $F$ in $M$. 
The normal bundle is oriented since the $\gS^1$-action induces a complex structure on it. By the Thom isomorphism, 
$H^k( U,U\setminus F) \cong H^{k-\ell}_{\gT^d}(F)$ where $\ell$ is the codimension of $F$ in $M$. Moreover, the map 
$H^{k-\ell}_{\gT^d}(F)\to H^k_{\gT^d}(F)$ is then given 
by cupping with the equivariant Euler class of the normal bundle.  
To see that this map is injective, we choose a product decomposition 
$\gT^d=\gT^b\times \gT^{d-b}$ such that the second factor 
acts trivially on $F$. Now 
$H^*_{\gT^d}(F)$ is isomorphic to the cohomology of the product
$(F\times_{\gT^b}E\gT^b)\times B\gT^{d-b}$. 
Furthermore, it is easy to see that the Euler class $H^{\ell}_{\gT^d}(F)$ 
pulls back to a nonzero element of $H^\ell(B\gT^{d-b})$. 
Therefore, cupping with the Euler class is an injective map, and the first result follows.  

Suppose now $d=b+1$. Then the action of $\gT^d$ on $M\setminus F$ 
is almost free. Since the homotopy fiber 
of the map $(M\setminus F)\times_{\gT^d} E\gT^d\to (M\setminus F)/\gT^d$ has the rational cohomology of a point, it follows 
that $H^*_{\gT^d}(M\setminus F)\cong H^*((M\setminus F)/\gT^d)$. 
In particular, we find $H^k_{\gT^d}(M\setminus F)\cong 0$ for 
$k>n-d$. The surjectivity now follows from the exactness of the long exact sequence of pair $(M, M\setminus F)$ in equivariant cohomology. 
\end{proof}

One can use the above lemma to reprove the classical

\begin{theorem}[Smith-Floyd]\label{thm:SmithFloyd} Suppose $\gS^1$ acts smoothly on a manifold $M^n$, 
and let $\gF$ denote the fixed-point set. Then the pullback map
 $H_{\gS^1}^k(M)\to H^k(F\times B\gS^1)$ is an isomorphism in degrees $k> n$.
In particular the total odd Betti number of $F$ equals
$\dim_{\Q}(H^i_{\gS^1}(M))$  for all odd $i> n$.
The latter number is in turn bounded above by the total odd Betti number of
$M$. Similar statements hold for the even Betti numbers.
\end{theorem}

\section{\secttwo}\label{sec:S1splitting}

There are three results in this section that will be used in other sections. 
The most important one by far is the following, which implies Theorem \ref{thm:S1splitting}.

\begin{lemma}[Theorem \ref{thm:S1splitting}]\label{lem:S1splitting}
If $\rho:\gT^d \rightarrow \SO(V)$ is a faithful representation with $d \geq 3$, then there exists a one-dimensional subgroup $\gH \subseteq \gT^d$ such that the induced representation $\bar \rho:\gT^d/\H \to \SO(V^{\gH})$ on the fixed-point set of $\gH$ 
 is faithful and isomorphic to a product representation $\bar\rho_1 \oplus \bar\rho_2:\gS^1 \times \gT^{d-2} \to
\SO(V_1) \oplus \SO(V_2)$ 
where $\bar \rho_1$ has the second factor $\gT^{d-2}$ in its kernel and 
$\bar \rho_2$ has the first factor $\gS^1$ in its kernel.
 Moreover, the $\gS^1$-action on $V_1$ defined by $\bar\rho_1$ is semi-free.
\end{lemma}

The one-dimensional subgroup is usually not connected. We recall that a group action is called semi-free if the action is free away from the  fixed-point set.\\[1ex]
{\bf Reduction Step.} It suffices to prove Lemma~\ref{lem:S1splitting} in the case that 
there is no vector in $V$ which has a non-trivial finite isotropy group $\gF$.\\[1ex]
Indeed, choose a vector $v\in V$ whose isotropy group $\gF$ is maximal among all finite (and possibly trivial) isotropy groups. Since $\gT^d/\gF$ acts effectively on the orbit $\gT^{d}\star v$, it acts effectively on the fixed-point set $V^{\gF} \subseteq V$ of $\gF$. By the maximal choice of $\gF$, the induced action $(\gT^d/\gF)$-action on $W$ does not contain any vectors with non-trivial finite isotropy. If we now find a one-dimensional subgroup of $\gT^d/\gF$ for the induced representation 
$\gT^d/\gF\rightarrow \SO(V^{\gF})$ as in the Lemma \ref{lem:S1splitting}, then the lemma
clearly follows for the original representation as well.

Maybe the most surprising thing in this section is that there is something non-trivial one can say about the weight system of a faithful representation of a torus $\gT^d$ without non-trivial finite isotropy groups. 

The second result is fairly straightforward and is also used in the proof of Theorem~\ref{thm:t5}.

\begin{lemma}[Refinement of Theorem~\ref{thm:S1splitting} for $d=3$]\label{lem:refine}
Let $\rho\colon \gT^3\rightarrow \SO(V)$ be a 
 faithful representation without disconnected isotropy groups. 
Then one of the following holds:
\begin{enumerate}
\item[a)] There is a one-dimensional subgroup $\gH\subset \gT^3$ whose fixed-point set
has codimension $k_0$, and $\gT^3/\gH$ acts effectively by a product representation of
$\gS^1\times \gS^1$ on $V^{\gH}$, where the two circles have fixed-point sets of codimensions
 $k_2$ and $k_3$ with 
$k_2+k_3\le k_0$.
\item[b)] The representation of $\gT^3$ itself splits as $\gS^1\times \gT^2$.
\end{enumerate}
\end{lemma}

The third result is Lemma~\ref{lem:repsplitting} below. It is also the main lemma that connects the condition of connected isotropy groups for torus representations with the codimension three property for subsets of vector spaces over $\Z_2$. This lemma is used in Section \ref{sec:z2codim2} to detect the existence of non-trivial finite isotropy groups.

The rest of the section is organized as follows.
The above reduction step will allow us to reduce the proof of Theorem~\ref{thm:S1splitting} to a combinatorial problem in $\Z_2$-vector spaces. The solution to this problem is presented first in the next subsection in order to avoid interrupting the flow of the discussion later. Section~\ref{subsec:weights} contains the core of the proof of Theorem~\ref{thm:S1splitting}. Here a solution is presented in terms of the weights of the representation. In the proof of Theorem~\ref{thm:S1splitting} in Section~\ref{subsec:split}, we verify that  this 
  is indeed exactly what is needed. Thereafter we 
 prove Lemma~\ref{lem:refine}.
The final two subsections are devoted to the proof of Theorem~\ref{thm:d+1choose2}, which first requires a proof of a $\Z_2$-version in Section~\ref{subsec:sparse}. 

\subsection{Sparse subsets in vector spaces over $\Z_2$}\label{sec:cod3property}
We consider finite-dimensional vector spaces $V$ over $\Z_2$, and we consider subsets $S \subseteq V$ which generate $V$ but which are nevertheless sparse in the following sense.

\begin{definition}[Codimension three property]\label{def:codim3}
A subset $S$ of a finite-dimensional $\Z_2$-vector space $V$ has the codimension three property if the following hold:
	\begin{enumerate}
	\item $0 \in S$ and $S$ generates $V$.
	\item For every subspace $U \subseteq V$ of codimension three generated by a subset of $S$, the projection $V \to V/U$ is non-surjective when restricted to $S$.
	\end{enumerate}
\end{definition}

Note in particular that  such a subset $S$ does not contain any three-dimensional subspace of $V$. Two additional basic properties are noted here without proof:

\begin{lemma}\label{lem:cod3closure}
Let $V$ be a finite-dimensional vector space over $\Z_2$, and let $S \subseteq V$ be a subset satisfying the codimension three property.
	\begin{enumerate}
	\item If $s \in S$, then the image $\bar S$ of $S$ under the projection $V \to \bar V = V/\inner{s}$ also satisfies the codimension three property in $\bar V$.
	\item If $V' \subseteq V$ is a subspace generated by a subset of $S$, then $S' = S \cap V'$ satisfies the codimension three property in $V'$.
	\end{enumerate}
\end{lemma}

We now prove the main combinatorial lemma required for the proof of the $\gS^1$-splitting lemma.

\begin{lemma}[\(\mathbb{Z}_2\)-version of the splitting theorem]\label{lem:z2lem} 
Let $V$ be a $d$-dimensional vector space over $\Z_2$ with $d\ge 3$. If $S\subseteq V$ satisfies the codimension three property, then there exist a codimension two subspace $W \subseteq V$ generated by a subset of $S$ and an element $s \in S \setminus W$ such that $s$ is the only element in $S \cap (s + W)$.
\end{lemma}

\begin{proof}
 We prove the lemma by induction on $d$. 
If $d=3$ then $S\neq V$.
Since $S$  generates $V$ it cannot be a subgroup. 
Using $0\in S$ we see that we can find $s_1,s_2\in S\setminus \{0\}$ with $s_1\neq s_2$ 
such that $s_1+s_2\notin S$.
Now the one-dimensional subspace $\ml s_2 \mr$ generated by $s_2$ 
is the solution to our problem since $(s_1+ \ml s_2\mr) \cap S=s_1$.  

Suppose now $d>3$. Choose an element $s_0\in S$ and consider the 
projection $\sigma\colon V\rightarrow Q:=V/\ml s_0\mr$. 
Clearly the induction assumption applies to
$Q$ with the subset $\sigma(S)$. Thus we can find a $(d-3)$-dimensional subspace 
$P\subset Q$ generated by $P\cap \sigma(S)$ and an element $\bar s\in \sigma(S)\setminus P$ 
such that $\bar s$ is the only element of $\sigma(S)$ contained in the affine subspace $\bar s +P$. 
If $\sigma^{-1}(\bar s)$ contains only one element $s$,  we would be done,  
since then $W:= \sigma^{-1}(P)$ together with the affine subspace $s+W$ is the solution to our problem. 
Thus we may assume that $\sigma^{-1}(\bar s)$ is given by two elements 
$s_1$ and $s_2=s_1+s_0$. 
We choose inside $\sigma^{-1}(P)$ a $(d-3)$-dimensional subspace $U$ that is transversal to $\ml s_0\mr$ 
and that is generated by $U\cap S$.

Then $s_2\notin s_1+U$. Thus the affine subspaces $s_1+U$ and $s_2+U$ only contain one element 
of $S$. We consider the projection $\pr\colon V\rightarrow V/U$ to the three-dimensional 
quotient vector space. By assumption, the set $M:=V/U\setminus \pr(S)$ is non-empty and $0\notin M$.
Put $x_i=\pr(s_i)$, $i=1,2$. By construction $x_1\neq x_2$. 

We claim that, for some $m\in M$, we have $x_1+m\notin M$ or $x_2+m\notin M$.
Suppose for a moment this is false and, for all $m\in M$, 
we have $x_1+m\in M$ and $x_2+m\in M$. This clearly implies that $x_1+x_2\notin M$ 
and thus the vectors $x_1,x_2,m$ form a basis 
for $m\in M$. 
Furthermore, we then have $x_1+m$, $x_2+m\in M$ 
and hence $x_1+x_2+m\in M$. Overall $M$ contains four different non-zero elements, 
$m$, $x_1+m$, $x_2+m$, and $x_1+x_2+m$. The remaining three, $x_1$, $x_2$, and $x_1+x_2$, are linearly 
dependent, and we get a contradiction to the fact that $\pr(S)$ generates $V/U$. 

Thus indeed, after possibly switching the roles of $x_1$ and $x_2$, 
we can assume that $x_1+m=x_3 \in \pr(S)$ for some $m\in M$. 
Of course, we can rewrite this as $x_1+x_3=m$. 
This can be expressed slightly differently as follows. 
The one-dimensional space $\ml x_3\mr$ 
has the property that $x_1$ is the only element of $\pr(S)$ in the affine subspace $x_1+ \ml x_3\mr$. 
The preimage $W:=\pr^{-1}(\ml x_3\mr )$ is a $(d-2)$-dimensional space generated by   $W\cap S$, 
and the affine subspace $s_1+ W$ only contains one element of $S$. 
\end{proof}


\subsection{The structure of weight systems of torus representations with connected isotropy groups.}\label{subsec:weights}
Recall that the irreducible non-trivial subrepresentations of a torus correspond to non-zero homomorphisms $\gT^d\rightarrow \gS^1$. Such a homomorphism is in turn determined by the induced homomorphism $h\colon \pi_1(\gT^d)\rightarrow \pi_1(\gS^1)\cong \Z$. 
We may identify $\Z^d$ and $\Gamma:=\pi_1(\gT^d)$, 
 and then $h$ becomes an element in the dual space
  $\Gamma^*\cong \Z^d$. Finally, recall that two real subrepresentations are equivalent if their induced elements in $\Gamma^*$ coincide up to the sign.

\begin{lemma}\label{lem:repsplitting} 
Let $\rho\colon \gT^d\rightarrow \SO(l)$ be a faithful representation of a torus with $d \geq 3$ and the property that there are no non-trivial finite isotropy groups. Consider a maximal collection of pairwise inequivalent, irreducible, non-trivial subrepresentations $\rho_1,\ldots,\rho_l\colon \gT^d\rightarrow \gS^1\cong \SO(2)$ of $\rho$. Let $h_1,\ldots,h_l\in \Gamma^*$ denote the induced elements in the dual space of $\Gamma:=\pi_1(\gT^d)$ as indicated above. 
	\begin{enumerate}[label=(\alph*)]
	\item If $h_{i_1},\ldots,h_{i_d}\in \Gamma^*$ are linearly independent, then they form a $\Z$-basis of $\Gamma^*$.
	\item If $h_1,\ldots,h_d$ are linearly independent, then every other $h_j$ can be expressed in the form $h_j=\sum_{i=1}^{d} z_i h_i$ with $|z_i| \leq 1$.
	\item The projection $\pi\colon \Gamma^*\rightarrow V:=\Gamma^*/(2\cdot \Gamma^*)$ obtained from reduction modulo two is one-to-one on $E = \{h_1,\ldots,h_l\}$ and satisfies $0 \not\in \pi(E)$.
 \item The subset $S = \pi(E) \cup \{0\}$ of $V$ satisfies the codimension three property.

	\item There is a corank one group $L_1\subset \Gamma^*$ and a corank two subgroup $L_2\subset \Gamma^*$ with $L_2\subset L_1$ such that $L_i$ is generated by $L_i\cap E$ for $i \in \{1,2\}$ and $L_1\cap E$ contains only one more element than $L_2\cap E$.
	\end{enumerate}
Furthermore, (e)  remains valid even without the assumption that the representation $\rho$ 
does not have any non-trivial finite isotropy groups.
\end{lemma}

For the proof, we identify  the fundamental group $\Gamma$ with the kernel of the exponential map $\exp\colon \mathfrak{t}\rightarrow \gT^d$.  

To prove (a), we may reorder and consider only the case where $h_1,\ldots,h_d \in \Gamma^*$ are linearly independent. If $h_1,\ldots,h_d$ generates a subgroup whose index is finite and at least two, then there is a non-trivial finite extension $\Gamma\subsetneq \hat\Gamma$ such that the $\Q$-extension of $h_i$ would still map $\hat\Gamma$ to $\Z$ for all $1 \leq i \leq d$. But then the finite group $\exp(\hat\Gamma)$ would be in the kernel $\rho_i$ for all $1 \leq i \leq d$. The intersection of all kernels of $\rho_i$ corresponds to the principal isotropy group of the representation $\rho_1\oplus\cdots \oplus\rho_d$. Since this is by definition a subrepresentation of $\rho$, this is a contradiction to the assumption that there are no finite isotropy groups for the representation $\rho$. 

To prove (b), after reordering it suffices to rule out 
 that $h_j = \sum_{i=1}^{d} z_i h_i$ with $|z_1| \geq 2$
 holds for some $j$. The weights $h_j,h_2,h_3,\ldots,h_d$ are linearly independent and generate a subgroup on index $|z_1| \geq 2$. Indeed, the change of basis matrix from $h_1,\ldots,h_d$ to this set is a lower triangular matrix with determinant $z_1$. This contradicts (a).

To prove (c), recall that $h_i \neq \pm h_j$ whenever $i \neq j$ as $\rho_i$ and $\rho_j$ correspond to inequivalent subrepresentations. By (b), this implies that $h_i$ and $h_j$ are linearly independent. By (a) we can then extend $h_i,h_j$ to a $\Z$-basis of $\Gamma^*$. But a $\Z$-basis is mapped by $\pi$ to a $\Z_2$-basis of $V$. In particular, $\pi(h_i)$ is non-zero and distinct from $\pi(h_j)$. 

To prove (d), let $U \subseteq V$ be a subspace of codimension three that is generated by a subset of $S$. Choose a basis $s_1,\ldots,s_d \in S$ for $V$ such that $s_4,\ldots,s_d$ form a basis for $U$. Choose $h_i \in E \cap \pi^{-1}(s_i)$ for $1 \leq i \leq d$. We argue by contradiction. This means that the seven non-zero elements in $V/U$ come from $S$ and hence ultimately from $E$. Three of these elements are the images of $h_1$, $h_2$, and $h_3$, and these have linearly independent images. Let $v_0,v_1,v_2,v_3\in E$ denote preimages of the other four non-zero elements in $V/U$ under the composition 
	\[E \subseteq \Gamma^* \to V \to V/U.\]
Up to reordering the $v_i$ and replacing $h_i$ by $-h_i$ and $v_i$ by $-v_i$ for $1 \leq i \leq 3$ as needed, we have 
\begin{align*}
	v_0 &= h_1 + h_2 + h_3 + u_0\\
	v_1 &= h_2 + \ep_1 h_3 + u_1,\\
	v_2 &= h_1 + \ep_2 h_3 + u_2,\\
	v_3 &= h_1 + \ep_3 h_2 + u_3
	\end{align*} 
for some $u_0, \ldots, u_3 \in \inner{h_4,\ldots,h_d}$ and some $\ep_i \in \{\pm 1\}$. We claim that all $\ep_i = 1$. Indeed, if $\epsilon_i = -1$, then the subset
	\[\{v_0, v_i, h_i\} \cup \{h_4,\ldots,h_d\}\]
is linearly independent but generates a subgroup of index two. This can be seen since the change of basis matrix from $\{h_1,\ldots,h_d\}$ to this basis has determinant $\pm 2$. This contradicts (a). But now we have that all $\ep_i = 1$, which implies that the subset
	\[\{v_1, v_2, v_3\} \cup \{h_4,\ldots,h_d\}\]
is linearly independent but only generates a subgroup of index two. It is the subspace of elements of even weight. This again contradicts (a), so the proof of (d) is complete.

To prove (e), we apply the combinatorial lemma from the previous subsection. There is a $(d-2)$-dimensional $\Z_2$-vector space $W\subseteq V$ generated by $W\cap S$ and an element $s\in S\setminus W$ such that $s$ is the only element of $S$ contained in the affine subspace $s+W$. Let $\hat{W}$ denote the $(d-1)$-dimensional subspace $W \cup (s+W)$. We define $L_2:=\mbox{span}_{\Z}(\pi^{-1}(W)\cap E)$ and $L_1:=\mbox{span}_{\Z}(\pi^{-1}(\hat{W}) \cap E)$. It is clear that $L_1$ cannot contain  $d$ linearly independent vectors of $E$, since they would form a $\Z$-basis which in turn would be mapped by $\pi$ to a $\Z_2$-basis in $V$. Thus $L_1$ has corank at least one. On the other hand, $\pi_{|L_1}\rightarrow \hat{W}$ is surjective and thus the corank is exactly one. One argues similarly for $L_2$. By construction $E\cap L_1$ contains exactly one more element than $E\cap L_2$, namely the unique preimage of $s$ in $E$. 
 
Finally, we claim that (e) holds without the assumption on the non-existence of finite isotropy groups. Indeed, one simply chooses linearly independent elements $e_1,\ldots,e_d\in E$ such that the index of $\mbox{span}_{\Z}(e_1,\ldots,e_d)=:\Upsilon^*$ in $\Gamma^*$ is maximal. One then replaces $E$ with $E':=E\cap \Upsilon^*$ and $\Gamma^*$ with $\Upsilon^*$. By the choice of $\Upsilon^*$, any linearly independent subset of $d$ vectors in $E'$ forms a $\Z$-basis of $\Upsilon^*$. Thus (a)--(d) can be proven for the pair $(\Upsilon^*,E')$ and (e) follows just as well.

\subsection{Proof of Theorem \ref{thm:S1splitting}}\label{subsec:split}

By the Reduction Step, we can assume that $\rho$ does not have any non-trivial finite isotropy groups. We can then apply part (e) of the previous lemma.
Recall that we viewed $\Gamma=\pi_1(\gT^d)$ as the kernel of the exponential map $\exp\colon \mathfrak{t}\rightarrow \gT^d$. If we  extend the homomorphism $h_i\colon \Gamma\rightarrow \Z$ to
$\hat h_i\colon \mathfrak{t}\rightarrow \R$, then the corresponding representation $\rho_i\colon \gT^d\rightarrow \gS^1$ is given by $\rho_i(\exp(v))\mapsto \exp(2\pi i \hat{h}_i(v))$. We put $\mathfrak{h}=\bigcap_{h\in E\cap L_1} \hat{h}^{-1}(\Z)$ and define $\gH:=\exp(\mathfrak{h})$. Notice that $\gH$ is the intersection of kernels of all elements in the set $\{\rho \mid \rho \mbox{ corresponds to an element in } L_1\cap E\}$. It is one-dimensional, as $E\cap L_1$ generates a corank one group in $\Gamma^*$. 

If $e\in E$ is any element with $e(\mathfrak{h})=\Z$, then it follows that $e$ is contained in the $\R$-span of $L_1$ and by part (b) of the previous lemma it is contained in the $\Z$-span of $L_1\cap E$ and thus it is itself an element of $L_1\cap E$. This implies that the restriction to $\gH$ of a non-trivial subrepresentation of $\gT^d$ is trivial if and only if the representation corresponds to an element in $L_1\cap E$. 

Thus the representation $\bar{\rho}$ of $\gT^d/\gH$ in the fixed-point set of $\gH$ exactly contains the non-trivial irreducible subrepresentions  corresponding to elements in $L_1\cap E$. By part (e) of the previous lemma, we know that $L_1\cap E=\{h\}\cup (L_2\cap E)$. 

We let $\bar{\rho}_1$ denote the sum of all irreducible subrepresentations with weight $h$ and let $\bar{\rho}_2$ denote the subrepresentation of $\bar{\rho}$ containing all others. Clearly $\bar{\rho}_1$ has $(d-2)$-dimensional kernel, and $\bar{\rho}_2$ 
has a one-dimensional kernel as claimed.

\begin{proof}[Proof of Lemma~\ref{lem:refine}]
By Lemma~\ref{lem:repsplitting} the map given by reducing the weights modulo 
2 is injective and maps them to a set $E$ that does not contain all non-zero  elements of $\Z_2^3$. 
We can assume that the element $(1,1,1)$ is avoided. If all other six non-zero elements 
are attained, then 
each of the pairs $\{(1,1,0),(0,0,1)\}$, $\{(1,0,1),(0,1,0)\}$ 
and $\{(0,1,1),(1,0,0)\}$ can be realized 
as non-zero elements of $S$ in a two-dimensional subspace 
of $\Z_2^3$. Thus each of these pairs give a solution to 
Lemma~\ref{lem:z2lem}, and we can find a corresponding 
solution to Lemma~\ref{lem:repsplitting} e). 
If we choose the pair whose 
corresponding multiplicities 
add up to the smallest number, then the conclusion of a) is satisfied. Note in this case that we actually can arrange for $k_2+k_3 \le \frac 1 2 k_0$.

If there are five weights, then we still have two disjoint pairs which can be realized 
as a product representation and this is still enough to get to the conclusion in a). 

If there are at most four weights, then either we are in the situation of 
b) or there are indeed four weights any three of them are linearly independent. But then any two of them give a solution to
Lemma~\ref{lem:repsplitting} e) and we are again in the situation of a) if we choose the pair with 
the smallest multiplicities. 
\end{proof}

\subsection{Some more combinatorics in $\Z_2$-vector spaces.} \label{subsec:sparse}
We prove some more results on subsets satisfying the codimension three
property (see Definition~\ref{def:codim3}).
The following lemma and its corollary are the basis for proving Theorem \ref{thm:d+1choose2}. It
 shows that subsets with the codimension three property are very sparse.

\begin{lemma}\label{lem:cod1property}
Let $V$ be a $d$-dimensional vector space over $\Z_2$ with $d\ge 3$. If $S\subseteq V$ satisfies the codimension three property, then there exists a codimension one subspace $U \subseteq V$ generated by a subset of $S$ such that $S \setminus U$ is linearly independent.
\end{lemma}

\begin{proof}  In order to argue by induction, we need to prove a slightly stronger statement: Let  $W\subsetneq V$ be a
$k$-dimensional subspace generated by a subset $L\subseteq S$. We assume $k\ge 1$ and, 
 if strict inequality holds, we assume that $L\subseteq V$ is given by $k+1$ vectors which are in general linear position inside $W$. We then claim that we can find a codimension one subspace $U\subseteq V$ with $W\subseteq U$ such that $S\setminus U$ is linearly independent.

To prove our claim, we can assume that the subspace $W\subsetneq V$ is maximal with the above property. We argue by induction on $d$. For the induction base $d=3$, there are two subcases $k=1,2$. If $k=2$, then by the above assumptions, $S$ contains all four elements in $W$. Since $S\neq V$, the set $S\setminus W$ contains at most $3$ elements and, for any choice, they are linearly independent. If $k=1$, then $L$ contains only one element $s_0$ and, by the maximality of $W$, 
$s_0+s\not\in S$ for all $s\in S\setminus W$.   Hence,  $S$ contains at most five non-zero elements, 
and we can put $U= W\oplus \ml s\mr$ for any $s\in S\setminus W$.

For the induction step, consider first the subcase $k<d-1$. Let $b_1,\ldots,b_k\in L$ 
denote $k$ pairwise different elements. By assumption, they are linearly independent
 and, in case of $k>1$, the remaining element is given by $\sum_{i=1}^k b_i$. 
 We claim that the projection $\pr\colon V\rightarrow V/\ml b_1\mr=Q $ becomes injective when
  restricted to $S':=S\setminus W$. In fact, otherwise we could find an element $s\in  S'$ such that 
  $b_1+s\in S'$, but then the elements $s, b_1+s,b_2,\ldots,b_k,\sum_{i=1}^kb_i$ are in general
   linear position in the vector space $W\oplus \ml s\mr$. Thus indeed the $\pr_{S'}$ is injective.
    The image of $\pr(L\setminus b_1)$ is  in general linear position in the vector space $\pr(W)$. By the 
 induction assumption, we can find a vector space $\bar U\subseteq Q$ of codimension one with 
 $\pr(W)\subseteq \bar U$, and the set $\bar S':= \pr(S)\setminus \bar U$ is linearly independent. 
  We just proved that every element in $S'$ has a unique preimage in $S$, and thus $\pr^{-1}(\bar S')\cap S$
   is linearly independent as well. Therefore, $U:=\pr^{-1}(\bar U)$ is the solution to our problem. 

It remains to treat the case of $k=d-1$. Assume on the contrary that $S':=S\setminus W$ is not linearly independent. Choose a minimal subset $C\subseteq S'$ such that $C$ is linearly dependent. Then $\sum_{c\in C}c=0$. Since $C$ is contained in an affine subspace, it follows that the order of $C$ is even.
 Let $c_1,\ldots, c_{2h}\in C$ denote the elements in $C$. Their span $\ml C\mr $ is then a $2h-1$-dimensional subspace.

If $2h-1\neq d$, then $W\not \subset \ml C\mr $, and we can find an element $b_1 \in L $ with 
$b_1\not\in \ml C\mr$.  If $2h-1=d$, we choose an arbitrary element $b_1 \in L$.  In either case, we consider 
the projection $\pr\colon V\rightarrow V/\ml b_1\mr$.
 The image of $L\setminus b_1$ is a set in general linear position inside the codimension one subspace $\pr(W)$.  By the induction assumption, this implies that $\pr(S')\supseteq \pr(C)$ is linearly independent. 
 In the case $2h-1\neq d$, we get a contradiction since $0=\sum_{c\in C} \pr(c)$ and since
 $\pr_{|C}$ is injective  by
  our choice of $b_1$.  
 
 Thus we may assume $2h-1=d\ge 4$.  
Then 
 $\pr(C)$ must have  $d-1$ elements since their span is $(d-1)$-dimensional.  
 Thus we see that $C$ has exactly $2$ elements more than $\pr(C)$. 
 We can find a subset  $K\in \pr(C)$ consisting of two elements with $\pr^{-1}(K)\subset C$.
 But now it is easy to see that $\pr^{-1}(K)$ is linearly dependent and by the minimal choice of $C$ 
 we get $2h=4<d+1$ -- a contradiction.
\end{proof}

As corollary we get the following upper bound on the size of subsets satisfying the codimension three property. This corollary will in turn imply Theorem \ref{thm:d+1choose2}.

\begin{corollary}\label{cor:d+1choose2forS}
If $V$ is a $d$-dimensional vector space over $\Z_2$ and $S \subseteq V$ has the codimension three property, then $S$ has at most $\frac{d(d+1)}{2}$ non-zero elements. Moreover, equality holds if and only if there exists a basis $b_1,\ldots,b_d \in S$ such that
	\[S = \{0\} \cup \{b_i\}_{1 \leq i \leq d} \cup \{b_i+b_j\}_{1 \leq i < j \leq d}.\]
\end{corollary}

\begin{proof}
We first study the set described in the conclusion. 
If $S_0$ is given by the elements of Hamming weight $\le 2$ with respect to the 
basis $b_1,\ldots,b_d$, then the same holds with respect to the basis 
$b_1,b_1+b_2,\ldots,b_1+b_d$. This shows that the group of linear transformations of $V$
leaving $S_0$ invariant acts transitively on $S_0\setminus \{0\}$.
Now for any element $s\in S_0\setminus \{0\}$, 
the projection $\pr \colon V\rightarrow V/\inner{s}$ maps $S_0$ onto 
a set which is again given by the elements of weight $ \le 2$ with respect to a suitable 
basis of $V/\inner{s}$. In fact, by the previous statement it suffices to prove this for $s=b_d$, 
in which case the result is obvious. 
We also deduce that $S_0$ has the codimension three property.

Next we establish the claimed bound for $S$.
For $d \leq 2$, the bound is trivial. For $d \geq 3$, we apply Lemma \ref{lem:cod1property}. There exists a codimension one subspace \(U\subset V\) generated by a subset of \(S\) such that \(S\setminus U\) is linearly independent. In particular, $S\setminus U$ has at most $d$ elements. On the other hand, $S \cap U$ has the codimension three property in the $(d-1)$-dimensional space $U$ by Lemma \ref{lem:cod3closure}, so it has at most $\frac{d(d-1)}{2}$ non-zero elements by the induction hypothesis. Putting these estimates together, we have that $S$ has at most $d + \frac{d(d-1)}{2} = \frac{d(d+1)}{2}$ non-zero elements as claimed.

If equality holds, then 
$S\setminus U$ is given by a basis $B=\{b_1,\ldots,b_d\}$ and $S\cap U$ has $\tfrac{1}{2}d(d-1)$ non-zero elements. 
By induction on the dimension and the property established in the first paragraph the following holds
for any non-zero  element $s\in U\cap S$: 
The projection $\pr\colon V\rightarrow V/\inner{s}$
 maps $S\cap U$ onto a set with $\tfrac{1}{2}(d-1)(d-2)$ non-zero elements. 
 It is clear that the set $\pr(B)$ and $\pr(S\cap U)$ are disjoint. 
 Using once more induction on the dimension we see that $\pr_{|B}$ cannot be injective. 
 Therefore, we find $i\neq j$ with $\pr(b_i)=\pr(b_j)$. 
 But this proves $s=b_i+b_j$. Since this holds for any non-zero element $s\in S\cap U$, 
 it follows that $S$ is given by the elements of weight $\le 2$ with respect to the basis $B$.
\end{proof}

\subsection{Proof of Theorem \ref{thm:d+1choose2}}
%
%
We have a faithful representation 
$ \gT^d\rightarrow \SO(V)$ without non-trivial finite isotropy 
groups, and we let $E$ denote the weights corresponding to maximal collection of non-trivial, pairwise inequivalent
subrepresentations. We have to bound the order 
of $E$ by $\tfrac{1}{2}d(d+1)$ and determine when equality can hold.  

By Lemma \ref{lem:repsplitting},  $E \subseteq \Gamma^*$ is mapped injectively into $\Gamma^*/2\Gamma^*\cong \Z_2^d$ and $\{0\}$ is not in the image of $E$. 
Note that if $d \leq 2$, this can be shown directly and the corollary follows. Letting $S$ be the union of $\{0\}$
 with the image of $E$, Lemma \ref{lem:repsplitting} implies that $S$ satisfies the codimension three property
  defined in \ref{def:codim3}. By Corollary \ref{cor:d+1choose2forS}, we conclude that $E$ has at most $\frac{d(d+1)}{2}$ elements.

Moreover, if equality holds, there exists a basis $s_1,\ldots,s_d \in \Z_2^d$ such that
	\[S\setminus \{0\} = \{s_i\}_{1 \leq i \leq d} \cup \{s_i + s_j\}_{1 \leq i < j \leq d}.\]
Let $h_i \in E$ and $h_{ij} \in E$ denote the unique preimages of $s_i$ for $1 \leq i \leq d$ and $s_i + s_j$ for $1 \leq i < j \leq d$, respectively. Note that, after replacing $h_{1j}$ by $-h_{1j}$ as necessary for $2 \leq j \leq d$, we may assume that $h_{1j} = h_1 + \epsilon_{1j} h_j$ for some $\epsilon_{1j} \in \{\pm 1\}$. Moreover, by replacing $h_j$ by $-h_j$ as needed for $2 \leq j \leq d$, we may furthermore assume that $\epsilon_{1j} = -1$ and hence that
	\[h_{1j} = h_1 - h_j\]
for all $2 \leq j \leq d$. Now replacing $h_{ij}$ by $-h_{ij}$ as needed for $2 \leq i < j \leq d$, we may assume
	\[h_{ij} = h_i + \epsilon_{ij} h_j\]
for some $\epsilon_{ij} \in \{\pm 1\}$. Suppose for a moment that $\epsilon_{ij} = 1$ for some $2 \leq i < j \leq d$. After reordering the $h_k$ with $2 \leq k \leq d$, we may assume that $i = 2$ and $j = 3$. Intersecting the kernels of the subrepresentations corresponding to the weights in
	\[\{h_1 - h_2, h_1 - h_3, h_2 + h_3\} \cup \{h_4,\ldots,h_d\}\]
yields a finite isotropy group of order two, a contradiction. Hence, all weights are of the form $h_i$ or $h_i - h_j$ with $i < j$, as claimed.

Finally, to see that the bound is sharp, one just has to check that any $d$ linearly independent vectors in
	\[
	E = \{e_i \st 1 \leq i \leq d\} \cup \{e_i - e_j \st 1 \leq i < j \leq d\}
	\]
form a $\Z$-basis. The second subset of $E$ contains no $d$ linearly 
independent vectors. Thus up to a permutation every linearly independent subset of order $d$
contains the element $e_d$, and it is easy to proceed by induction.

\section{\sectthree}\label{sec:t5}
\label{sec:Obtaining4periodicity}\label{sec:weakt5}

In order to establish a weak version of Theorem A
the first step is to use the $\S^1$-splitting theorem, together with the Connectedness Lemma and Four-Periodicity Theorem, to produce a submanifold with four-periodic rational cohomology. This is an immediate consequence in the presence of 
$\gT^7$ symmetry, as we explain now. Indeed, the $\S^1$-splitting theorem would imply that some three-dimensional subgroup $\gH \subseteq \gT^7$ has a fixed-point component $N$ containing $F$ with the property that $N$ contains four mutually transversely intersecting, totally geodesic submanifolds $N_i$. Since the codimensions of the $N_i \subseteq N$ sum to at most $\dim N$, there exist $N_i$ and $N_j$ such that $\dim(N_i \cap N_j) \geq \frac{1}{2} \dim N$. By Corollary \ref{cor:transversal}, it follows that $N$ has four-periodic rational cohomology. As it turns out, we can accomplish a similar result using only $\gT^5$ symmetry.

\begin{proposition}\label{pro:4periodicN}
Let $\gT^5$ be a torus acting isometrically and effectively 
on a  closed, connected, oriented, positively curved Riemannian manifold,
and let $F$ be a fixed-point component. There is a two-dimensional subgroup $\gH\subseteq \gT^5$
 such that the following holds for the induced action of
 $\gT^3=\gT^5/\gH$ on the fixed-point component $N^m$ of $\gH$ at $F$. 
	\begin{itemize}
	\item The action of $\gT^3$ is effective and splits in a neighborhood of $p\in F$ as a product action $\gS^1\times \gS^1\times \gS^1$ of 
	three circles acting locally semi-freely 
	whose fixed-point components in $N^m$ intersect pairwise  perpendicularly at $F$. 
	\item $N$ and its universal cover have four-periodic rational cohomology, and either
\begin{enumerate}
\item[a)] Some circle factor has a fixed-point component at $F$ in $N$ 
with  the rational cohomology  of a rank one symmetric space, or
\item[b)]  $N$, the fixed-point component at $F$ of each circle factor, and all intersections thereof  
have the rational cohomology of $\Sph^h\times \HP^{n_i}$ with $h\in \{2,3\}$.	
\end{enumerate}
	\end{itemize}
\end{proposition}

To get from here to Theorem A, we will show in the next section that always alternative a) holds.  
It should be understood that in alternative a) the remaining two circle factors still act effectively on the fixed-point component and that by Theorem~\ref{thm:SinglyGenerated} below, $F$ has the cohomology 
of a rank one symmetric space, too.
Notice that the proposition already implies the corollary on the positivity of the Euler characteristic (see Section \ref{sec:PositiveEuler} for details).
To prove the dichotomy  in the proposition, we will need the following two lemmas. 
They rely on calculations in equivariant  cohomology. However unlike in the next section, where we need tori, 
here we can get by with the equivariant cohomology of circle actions. 

\begin{lemma}[$b_3$ Lemma]\label{lem:b3} Suppose a circle acts effectively  on 
a connected, closed, orientable manifold $M^{4n+2}$ with four-periodic rational cohomology
 and $b_1(M) = 0$. If every fixed-point component has $b_1(F_i) = 0$ and
the inclusion map $F_0\to M$ is seven-connected for  some fixed-point component $F_0$, then $b_3(M) = 0$.
\end{lemma}

\begin{lemma}\label{lem:shxHP}
Let $M$ be a closed, oriented manifold with $H^*(M) \cong H^*(\s^h \times \HP^n)$ with $h=2,3$. Assume $M$ admits an effective action by a 
torus $\gT$. Every fixed-point component $F_i$ is a rational cohomology $\CP^{n_i}$, $\HP^{n_i}$, $\s^{h_i} \times \CP^{n_i}$, or $\s^{h_i} \times \HP^{n_i}$ for some $n_i \leq n$ and $1 \leq h_i \leq h$ with $h_i \equiv h \bmod 2$.
\end{lemma}

The proofs of these lemmas use similar arguments to a classical result of Bredon, which itself will be used frequently in this paper. 

\begin{theorem}[Bredon]\label{thm:SinglyGenerated}
If a torus acts on an oriented, closed manifold $M$ with $H^*(M) \cong \Q[x]/(x^{n+1})$, then each fixed-point component $F_i$ has $H^*(F_i) \cong \Q[x_i]/(x^{n_i + 1}_i)$ where $|x_i| \equiv |x|\bmod{2}$ and $|x_i| \leq |x|$, with equality only $x$ restricts non-trivially to $H^*(F)$. 
Moreover the sum of $n_i$ over all components $F_i$ satisfies $\sum (n_i + 1)=n+1$.
\end{theorem}

Note in particular that a rational $\CP^n$ has fixed-point components that are also rational $\CP^{n_i}$, and that for a rational $\HP^n$, both $\CP^{n_i}$ and $\HP^{n_i}$ can arise. The statement for rational spheres goes back to Smith \cite{Smith38}.  For the case of rational Cayley planes, one can strengthen the conclusion somewhat since 
there 
 is no closed manifold with dimension $12$ and total Betti number three by the Hirzebruch signature theorem (see, for example, \cite{Su14}).

We now first prove the proposition and then the two lemmas. At the end, we have included two subsections whose results are not needed further on. In Section~\ref{subsec:evensphere}, we show that, if the fixed-point set of a $\gT^5$-action is an even-dimensional rational sphere, then it is an integral homology sphere. 
In Section~\ref{sec:PositiveEuler}, we have explained in more detail how Corollary~\ref{cor:hopf} follows. 
We also rule out the existence of positively curved orientable manifolds with Euler characteristic three and $\gT^5$-symmetry.


\begin{proof}[Proof of Proposition~\ref{pro:4periodicN}]
We apply Theorem~\ref{thm:S1splitting} to the isotropy representation at $p$ in the normal space $\nu_p(F)$, and we choose a one-dimensional subgroup $\gH_1$ as in the statement. Let $G^g$ be the fixed-point component of $\gH_1$ containing $p$. Then $\gT/\gH_1$ acts effectively $G^g$, and $\gT/\gH_1$ splits as $\gS^1\times \gT^3$ in a neighborhood of $p$. 
 The circle factor $\gS^1$ acts semi-freely in a neighborhood of $p$, and its fixed-point component $G_1^{g-k_1}$ intersects any submanifold of $G$ that is fixed by a subgroup of $\gT^3$ transversely.
After  passing to the fixed set of a finite isotropy group we may assume that there are no non-trivial finite isotropy groups in a neighborhood of $p$ in $G$. 

We now apply Theorem~\ref{thm:S1splitting} again and choose a 
subgroup $\gH_2 \subseteq \gT^3$ as in the theorem, and let $N^m \subseteq G^g$ 
be the fixed-point component of $\gH_2$ that contains $p$. 
The induced action of $\gT^3/\gH_2$ splits as a product of semi-free circle actions
 whose fixed-point components, $N_2^{m-k_2}$ and $N_3^{m-k_3}$, intersect transversely in $N^m$.
 The group  $\gH_2$ is not unique and we choose $\gH_2$ in such a way 
 that the codimension $k_0=g-m$ of $N^m$ in $G$ is maximal.
 
Altogether $N^m$ admits an action $\gS^1\times \gS^1\times \gS^1$ where the three circles fix the three submanifolds $N^{m-k_i}_i$ that intersect pairwise perpendicularly
 in $N^m$. Notice also that $N_1^{m-k_1}$
 is the transverse intersection of $G_1^{g-k_1}$ and $N^m$ 
 in $G$. 
The  major step will be the proof of the following
 \\[1ex]
 {\bf Claim.}  $N^m$ and its universal cover have four-periodic cohomology.\\[-1.5ex]
%

%
%
First, if $k_1 \geq \frac{m}{2}-f$, then the result follows 
by applying Corollary \ref{cor:transversal}
 to the transverse intersection of $N_2$ and $N_3$. 


Second, assume that $k_1 < \frac m 2-f$ and $k_1 \leq k_0=g-m$. 
Then $N_1^{m-k_1}\to N^m$ is $(m-k_1)$-connected
since  $N_1$ is the transverse intersection of $G_1^{g-k_1}$ and $N^m$ in $G^g$ and $k_1 \leq k_0$.  
If $k_1\le m/3$ we can apply Corollary \ref{cor:transversal} to get 
the result. If $k_1 >m/3$, then we still know 
that $\tilde N^m$ has $k_1$-periodic integral cohomology. 
 Moreover,  we may assume 
$k_2<\tfrac{m}{3}$. 
Hence, the Periodicity Lemma applied to the transverse intersection of $N_1^{m-k_1}$ and $N_2^{m-k_2}$ in $N^m$ implies that $\tilde N_1$ has $k_2$-periodic integral cohomology. Using once more that the inclusion $N_1^{m-k_1} \to N^m$ is $(m-k_1)$-connected, we have that the inclusion $\tilde N_1^{m-k_1} \to \tilde N^m$ is $(m-k_1)$-connected as well. Since $m-k_1\ge k_1+1$, the element $e \in H^{k_1}(\tilde N^m;\Z)$ inducing $k_1$-periodicity in $\tilde N^m$ must factor as a product of two elements where one has degree $k_2$. It follows that $\tilde N$ has $\gcd(k_1,k_2)$-periodic integral cohomology (cf. \cite[Lemma 3.2]{Kennard13}). Since $k_2 < k_1$, $\tilde N$ has $k$-periodic integral cohomology for some $k \leq \frac{1}{2} k_1 < \frac{m}{4}$. By the Four-Periodicity Theorem, it follows that $\tilde N$ has four-periodic rational cohomology. To see that $N$ itself has four-periodic cohomology as well one argues exactly as in the proof of Corollary~\ref{cor:transversal}.

Finally, suppose $ k_1 < \frac{m}{2}-f$ and $k_0 <  k_1$. 
Here we come back to the fact that we have chosen $\gH_2$ in such a way that it maximizes 
the codimension $k_0$.  Using the slight refinement of Theorem~\ref{thm:S1splitting} established
in Lemma~\ref{lem:refine} we either have $k_0\ge k_2+k_3$ 
or the product action 
of $\gS^1\times \gT^3$ at $p$ in $G$ actually splits 
as $\gS^1 \times \gS^1\times \gT^2$ product action.
In the former case  $k_1= m-k_2-k_3-f\ge m-k_0-f\ge m-f-k_1$ and we get a contradiction to our 
assumptions of the case.  In the latter case we can assume that remaining 
two torus $\gT^2$ on $G$ does not split as a product as well, because otherwise we are done as
explained in the paragraph before 
Proposition~\ref{pro:4periodicN}. Thus  $\gT^2$ has three non-trivial inequivalent representations, and they occur 
with various multiplicities. By our choice of $\gH_2$, 
$k_3/2$ is the smallest multiplicity while $k_0/2$ 
is the sum of the other two multiplicities.
In particular, $k_0\ge 2k_3$. 
In addition, we can switch the roles of $k_1$ and $k_2$ in all previous arguments. So 
we also have $k_0\le k_2$. 
After a possible permutation we get $k_1\ge k_2\ge  2k_3$.
Thus $\max\{k_2+3k_3,2k_2+k_3\}\le m$ and
 the result follows from Corollary~\ref{cor:transversal}.

We have proved the above claim.  Notice that we have proved in all cases that, at least for one $i$, the inclusion map $N_i\rightarrow N$ is $(m/2)$-connected. Furthermore, 
we have $m\ge 12$ unless $k_i=2$ for some $i$ and $\tilde N^m$ is fixed-point homogeneous and 
diffeomorphic to rank one symmetric space by a result of Grove and Searle.  
Thus we can apply
Lemma~\ref{lem:b3} to the $\gS^1$-action of the $i$-th factor 
to see that $b_3(N^m)=0$ in the case of $m\equiv 2$ mod $4$. 
By Proposition~\ref{pro:4periodic}, we see that $N^m$ has the cohomology of 
$\Sph^m$, $\CP^{m/2}$, $\HP^{m/4}$, or  $\Sph^h\times \HP^{(m-h)/4}$ with $h\in \{2,3\}$. 
Using Lemma~\ref{lem:shxHP} and the fact that $N_i$ has finite fundamental group, we 
see that either the conclusion of a) holds or  $N_i$ 
has the rational cohomology of $\Sph^h\times \HP^{n_i}$ 
 or of $\Sph^h\times \CP^{n_i}$ with $n_i>0$ for $i=1,2,3$. 
The latter can be ruled 
out by the Connectedness Lemma. Indeed, for $k_i \leq \frac m 2 - 1$, it implies that the inclusion map $N_i \rightarrow M$ is $3$-connected, and otherwise it implies that the inclusion map $N_i \cap N_j \rightarrow N_i$ is $\dim(N_i \cap N_j)$-connected for $j \neq i$.    But $\Sph^h\times \CP^{n_i}$ contains no submanifold of dimension at least three whose 
inclusion map is maximally connected, so we have shown that either Conclusion a) holds or    
$N_1$, $N_2$  and $N_3$ have the cohomology of $\Sph^h\times \HP^{n_i}$. 
That intersections thereof have the same cohomology type  is an immediate consequence of the Connectedness Lemma. 
\end{proof}

\subsection{\subsecbthree}\label{sec:b3}
We prove Lemma~\ref{lem:b3}. 
All cohomology groups and Betti numbers are taken with respect to $\Q$.
We argue by contradiction. By Proposition~\ref{pro:4periodic}, we may assume 
 that $H^*(M)$ is generated as an algebra 
by some $c \in H^2(M^{4n+2})$, $a \in H^4(M^{4n+2})$, and $b_1,\ldots,b_s \in H^3(M^{4n+2})$ for some even $s > 0$. A complete set of relations is given by $c^2 = 0$, $a^{n + 1} = 0$, $c b_i = 0$ for all $i$, and $b_i b_j = m_{ij} a c$ for some non-degenerate, skew-symmetric matrix $(m_{ij})$.

We want to work with equivariant cohomology $H^*_{\gS^1}(M)$, that is, the rational 
cohomology of the Borel construction $M\times_{\gS^1}E\gS^1$. There is a fibration 
 $$M\stackrel{j}{\rightarrow} M\times_{\gS^1} E\gS^1\rightarrow B\gS^1.$$
We first claim that the map   
$j^*\colon H^*(M\times_{\gS^1}E\gS^1)\rightarrow H^*(M)$ 
is surjective. By assumption, the inclusion map $F_0\rightarrow M$ of some fixed-point component $F_0$ is $7$-connected.
 In particular, $H^i(M)\rightarrow H^i(F_0)$ is an isomorphism in degrees $i\le 6$, and the same follows for $H^i(M\times_{\gS^1} E\gS^1)\rightarrow H^i(F_0\times B\gS^1)$. The map $H^*(F_0\times B\gS^1)\rightarrow H^*(F_0)$ is surjective by the K\"unneth formula. Thus we see that $j^*$  is surjective in degrees $i\le 6$. Since the cohomology of $M$ is generated in degrees $\le 4$, it is surjective for all $i$.

This implies that, as a $H^*(B\gS^1)$-module, $H^*_{\gS^1}(M)$ is isomorphic to the 
cohomology of the product $M\times B\gS^1$. Moreover, by Theorem \ref{thm:SmithFloyd}, the total Betti number of the fixed-point set $F$
of the $\gS^1$-action coincides with the total Betti number of $M$.
In fact, it follows that the map $H^i_{\gS^1}(M)\to H^i(F\times B\gS^1)$ 
is an isomorphism for $i>4n+2$. Since the total odd Betti number of $F_0$ is smaller than the one of $M$, it follows that there must be a second component $F_1$ with non-trivial odd Betti number.

Before we come back to $F_1$, we use the seven-connected inclusion map $F_0\rightarrow M$ to normalize a generating system of the equivariant cohomology of $M$. 
We let $t\in H^2(B\gS^1)$ denote a generator and identify it with its pullback to $H^2_{\gS^1}(M)$. The cohomology ring $H^2_{\gS^1}(M)$ of $M\times_{\gS^1} E\gS^1$ coincides up to degree $6$ with the cohomology of the product $F_0\times B\gS^1$, which in turn coincides up to degree $6$ with the cohomology of the product $M\times B\gS^1$.
Therefore, we can find
lifts $\gAm\in H^2_{\gS^1}(M)$ of $c$,
$\aLpha\in H^4_{\gS^1}(M)$ of $a$,
and $\bEta_1,\ldots,\bEta_s\in H^3_{\gS^1}(M)$ 
of $b_1,\ldots,b_s$ such that  their pullbacks to  $H^*(F_0\times B\gS^1)$  correspond to  elements in $H^*(F_0)$ in the K\"unneth decomposition. 
 This way we see that $\gAm^2=0$ and $\bEta_k\bEta_l=m_{kl}\aLpha \gAm$ where $m_{kl}\in \Q$ is as above.

The elements $\gAm,\bEta_1,\ldots,\bEta_s,\aLpha$, and $t$ generate  the cohomology ring $H^*_{\gS^1}(M)$. 
We let $\iota\colon F_1\times B\gS^1\rightarrow M\times_{\gS^1}E\gS^1$ denote 
the inclusion map. Then we can decompose  the pullbacks 
of our generating system to $H^*( F_1\times B\gS^1)$
  according to the  K\"unneth formula.
$$\begin{array}{rcll}
\iota^*(\gAm)&=&c_2 & \mbox{ with $c_2\in H^2(F_1)\subset H^2(F_1\times B\gS^1)$ }\\
\iota^*(\bEta_j)&=&\barbeta_j&\mbox{ with $\barbeta_j\in H^3(F_1)\subset H^3(F_1\times B\gS^1)$ }\\
\iota^*(\aLpha)&=& a_4+a_2t+a_0t^2
&\mbox{ with $a_4\in H^4(F_1)$, $a_2\in H^2(F_1)$ and $a_0\in \Q$. }
\end{array}
$$
In the first equation, we used that $\iota^*(\gAm)^2=\iota^*(\gAm^2)=0$, 
and in the second equation, we used our assumption $H^1(F_1)=0$.

We claim that $a_0 = 0$. We  know that $\iota^*$ is surjective in degrees larger than $4n+2$. Since the image of $\iota^*$ is generated by the above elements and $t$, 
it follows that the cohomology of $F_i$ is generated by $c_2$, $\barbeta_j$ ($j=1,\ldots, s$), $a_2$, and $a_4$.  
In particular, we get 
$\barbeta_k\neq 0$
for some $k$ because otherwise all odd Betti numbers of $F_1$  would vanish. 
Using Poincar\'e  duality of $F_1$  we can find $l$ such that 
$\barbeta_l\cdot \barbeta_k\neq 0$. But then
\begin{eqnarray*}\label{prod_alpha}
0\neq \barbeta_k\barbeta_l&=&m_{kl} \iota^*(\gAm \aLpha)=m_{kl}(c_2 a_4+c_2 a_2 t+a_0c_2 t^2)
\end{eqnarray*}
As we can view both sides as polynomial in $t$, we get $c_2 a_2=0$ and $a_0c_2=0$.
Moreover, $c_2\neq 0$ and thus $a_0=0$. 

We choose a point $p_i\in F_i$ for $i=0,1$.  We look at the image of
 $$H^{2j}_{\gS^1}(M)\rightarrow H^{2j}(\{p_0,p_1\}\times B\gS^1)$$
 for  $2j>4n+2$. 
We claim that our previous calculation implies that the image is one-dimensional. 
In fact, we have just seen that elements $\bEta_i, \gAm$, and $\aLpha$ are in the kernel
$H^*_{\gS^1}(M)\to H^*(p_i\times B\gS^1)$. This in turn implies that 
the image of the above map is given by the image of $\Q t^j$ and thus one-dimensional.

The contradiction arises since, on the other hand, we know that the map is surjective. In fact, as mentioned earlier,  $H^{2j}_{\gS^1}(M)\to H^{2j}(F\times B\gS^1)$ is surjective, and since $p_1$ and $p_0$ are in different components, the map $ H^{2j}(F\times B\gS^1) \to H^{2j}(\{p_1,p_0\}\times B\gS^1)$ is surjective as well.

\subsection{\subsecshxHP}\label{subsec:shxHP}
 We prove Lemma~\ref{lem:shxHP}.  We may assume that the torus is a circle $\gS^1$, 
 since every fixed-point component of a torus is also the fixed-point component of a circle. 
 We can also assume that the fixed-point set is not empty.  
 We again work with the equivariant cohomology $H^*_{\gS^1}(M)$.  
 We first claim the map to $H^*(M)$ is surjective. This is clear for $h=2$ since all 
 odd Betti numbers vanish. For $h=3$, we can look at the fibration
 $$
\gS^1\rightarrow M\cong M\times E\gS^1\stackrel{j}{\rightarrow} M\times_{\gS^1}E\gS^1
$$
 and the associated Gysin sequence 
$$
\cdots\to H^k_{\gS^1}(M)\stackrel{\cup e}{\rightarrow} H^{k+2}_{\gS^1}(M)\stackrel{j^*}{\rightarrow}H^{k+2}(M)\stackrel{Gy}{\longrightarrow} 
H^{k+1}_{\gS^1}(M)\to \cdots
$$
 The Euler class $e\in H^2_{\gS^1}(M)$ has the property that $e^k\neq 0$ for all 
 $k$ since its restriction to $p\times  B\gS^1$ is non-zero for any fixed point $p$. 
 With this in mind, we see that $H^3_{\gS^1}(M)\to H^{3}(M)$ is surjective.  It suffices to prove $D:=j^*\circ Gy=0$. 
 Indeed, since $j^*$ is injective in degree three,  $Gy$ must then be 
  zero in degree four and hence $j^*$ is surjective 
 in degree four. Since the cohomology of $M$ is generated 
 in degrees $\le 4$, it follows that $j^*$ is surjective in all degrees.   To show $D=0$
it is convenient to work with de Rham cohomology. 
 If $\omega$ is an  $\gS^1$-invariant closed differential form 
 on $M$, then $D([\omega])=[X\lnot\omega]$, 
 where $X$ is the Killing field induced by the $\gS^1$-action on $M$. Therefore, 
 $D([\omega_4^k])=k [D(\omega_4)\wedge \omega_4^{k-1}]$
 holds for an $\gS^1$-invariant closed differential form $\omega_4$ 
 representing a non-zero element in $H^4_{dR}(M)$. 
 If we choose $k=n+1$, then $0=D[\omega^{n+1}]=
 (n+1) [D(\omega_4)\wedge \omega_4^{n}]$,  
 but this proves $D([\omega_4])=0$, since otherwise 
 the right-hand side would be a non-zero multiple of the fundamental class.

 We have shown that $j^*$ is surjective. We
fix   a generator $t \in H^2(B\gS^1)$ and choose elements \(\aLpha \in H^4_{\gT}(M)\) and $\gAm \in H^h_{\gT}(M)$ which pull back to generators 
  $a\in H^4(M)$ and $c\in H^h(M)$, respectively.  Let  $F$ denote a fixed-point component of $\gS^1$ and $\iota\colon F\times B\gS^1\to E\gS^1 \times_{\gS^1}M$ 
  the inclusion. We can decompose  the pullbacks of $\aLpha$ and $\gAm$ 
  according to the  K\"unneth formula.
If $h=3$,   we get
	\begin{eqnarray*}
	\iota^*\aLpha&=&  a_4+ t  a_2+ a_0 t^2,\\
	\iota^*\gAm		&=&  c_3 + t   c_1,
	\end{eqnarray*}
for some  $ c_i \in H^i(F)$ for $i=1,3$, 
and
$ a_i\in H^i(F)$ for $i=0,2,4$.
 Note moreover that we may assume $ a_0 = 0$ by replacing $\aLpha$ by $\aLpha -  a_0 t^2$.

We claim that if $ a_2\neq 0$, then  $ a_4$ is a multiple of 
$ a_2^2$. 
We know $H^*_{\gS^1}(M)\rightarrow H^*_{\gS^1}(F)$ is surjective in degrees $\ge \dim M+1 = 4n+4$.
Since all even cohomology of $H^*_{\gS^1}(M)$ can be expressed as 
linear combination of $\aLpha^j\cdot t^l$, we get an equation of the form
$$
t^{2n+1} a_2=\sum_{i=0}^{n+1}q_i t^{2i}(a_4 + t  a_2)^{n+1-i}.
$$
with rational coefficients $q_i$.
Comparing the coefficients of $t^{2n+1}$  gives $q_{n}=1$, and then comparing those of $t^{2n}$ 
gives $a_4 + q_{n-1} a_2^2=0$, our desired relation. 
This implies that cohomology groups in even degrees are generated by powers of 
$ a_2$ or, if $ a_2=0$, by powers of $ a_4$.  
The structure of
 the odd cohomology groups follows from Poincar\'e  duality.
Thus $F$ has the cohomology of $\Sph^{h_f}\times \CP^{n_f}$ or $\Sph^{h_f}\times \HP^{n_f}$. 
Since $\iota^*\gAm \neq 0$, either the first or third Betti number is not zero. Hence, $h_f\in \{1,3\}$.

If $h=2$, we get after possibly adding elements of $H^4(B\gS^1 )$ 
and $H^2(B\gS^1)$ that
 \begin{eqnarray*}
	\iota^*\aLpha	&=&  a_4 + t  a_2,\\
	\iota^*\gAm	&=&  c_2.
	\end{eqnarray*}
If $ c_2=0$, one can argue as above to see that the cohomology of $F$ is generated by one element in degree 
$2$ or $4$. If $ c_2 \neq 0$, one still has $c_2^2=0$ 
since $ c_2$ is given by restricting $c$ to $F$. 
 We now claim  that, if $ a_2$ and $ c_2$ 
are linearly independent, then $a_4$ is a linear combination of $a_2^2$ and $a_2\cdot  c_2$. 
We have an equation of the form
\[t^{2n+1} a_2 = \sum_{i=0}^{n+1} q_i t^{2i} (a_4 + t a_2)^{n+1-i} + \sum_{i=1}^{n+1} r_i t^{2i-1} c_2 (a_4+ta_2)^{n+1-i}.\]
We get $q_{n+1} = 0$ by comparing the coefficients of $t^{2n+2}$ and  
$a_2 = q_{n} a_2 + r_{n+1} c_2$ by comparing the $t^{2n+1}$ coefficients. 
Since $c_2$ and $ a_2$ are assumed to be linearly independent, it follows that $q_n = 1$ and $r_{n+1} = 0$.    
Now the equation $0 = a_4 + q_{n-1} a_2^2 + r_{n}  c_2 a_2$ from the coefficients of $t^{2n}$ gives the desired relation.

In summary, if $ c_2\neq 0$, then the cohomology of $F$ is either generated 
by $c_2$ and $ a_4$ or by $c_2$ and $ a_2$. 
Together with $c_2^2=0$ and Poincar\'e  duality, we deduce that $F$ either has 
 the cohomology of $\Sph^2\times \HP^{n_i}$ 
or the cohomology of $\Sph^2\times \CP^{n_i}$.

 \subsection{\subsecevensphere}\label{subsec:evensphere}
 In this subsection we want to prove 
\begin{theorem}\label{thm:evenspheres} Suppose $\gT^5$ acts effectively on a positively curved, even-dimensional, simply connected manifold by isometries.
If  $F$ is a fixed-point component of positive dimension with 
$\chi(F)\le 2$, then $F$ is homeomorphic to a sphere.
\end{theorem}

\begin{proof} Notice that $F$ is orientable and thus simply connected by Synge's theorem.
If $\dim(F)\le 4$, the lemma follows. Thus we may assume $\dim(F)\ge 6$.

Choose submanifolds $F \subseteq N_i  \subseteq N^m$ 
as in Proposition~\ref{pro:4periodicN}, and assume $k_1 \leq k_2 \leq k_3$. The argument above
 shows that $N$ is a rational $\s^m$, $\CP^m$, $\HP^m$, or $\s^2 \times \HP^m$.
Since $F$ is a fixed-point component of $\gT^3$ in $N$, $\dim(F)\ge 6$, and $\chi(F)\le 2$,   we deduce  from Theorem~\ref{thm:SinglyGenerated} and Lemma~\ref{lem:shxHP} that $F$ and $N$ are rational spheres. 

Recall that the circle acting on $N$ and fixing $N_1$ acts semi-freely in a neighborhood of $N_1$. In particular, $N_1$ is a fixed-point component of the $\Z_p$ inside this circle. Since in addition $N_1$ is a rational sphere by the Connectedness Lemma, we have that $\chi(N) = \chi(N_1)$. 
In particular, Berger's theorem implies that the fixed-point set of the $\Z_p$-action on $N$ is connected. By \cite[Theorem 5.1]{Wilking03}, we have $H^{k_1}(N;\Z_p) = 0$. 

By the Connectedness Lemma, $H^{k_1}(N_2;\Z_p) = 0$, and by the Periodicity Lemma,  $H^*(N_2;\Z_p)$ is $k_1$-periodic. Hence, $N_2$ is a $\Z_p$-homology sphere. It follows by Smith's theorem that $N_1 \cap N_2$ and $F = N_1 \cap N_2 \cap N_3$ are also $\Z_p$-homology spheres, as they arise as fixed-point components of $\Z_p$-actions on $\Z_p$-homology spheres. Since this holds for all primes, $F$ is an integral homology sphere. Since in addition $F$ is simply connected, $F$ is homeomorphic to a sphere by the resolution of the Poincar\'e conjecture.
\end{proof}

\subsection{\subseceuler}\label{sec:PositiveEuler}
Let $M$ be an even-dimensional, closed, oriented, positively curved Riemannian manifold 
with an effective, isometric $\gT^5$-action. Then the
Euler characteristic $\chi(M) = \sum \chi(F)$ where the sum is over components $F \subseteq M^{\gT^5}$. By a result of Berger, the fixed-point set $M^{\gT^5}$ is non-empty (see Theorem \ref{thm:Berger}). 
Thus Proposition~\ref{pro:4periodicN} implies that the Euler characteristic is at least two.

This proves the first claim of the following.

\begin{theorem}\label{thm:PositiveEuler-Rigidity}
Let $M^{2n}$ be a simply connected, closed, positively curved Riemannian manifold. If $\gT^5$ acts isometrically and effectively on $M$, then $\chi(M) \geq 2$, with equality only if the fixed-point set is homeomorphic to a sphere. Moreover, $\chi(M) \neq 3$.
\end{theorem}

 It remains to show that $\chi(M) \neq 3$. We split this into two steps.

\begin{lemma}\label{lem:t3connected}
Let $M$ be a closed, simply connected, positively curved manifold with $\gT^3$ symmetry. If  $\gT^3$ has a connected fixed-point set, then $\chi(M)$ is even.
\end{lemma}

Before proving this lemma, we remark on some examples. First, there are $\gT^3$-actions on spheres with a connected fixed-point set, but this is not a contradiction since spheres have Euler characteristic zero or two. Second, some circle actions on positively curved Riemannian manifolds can have connected fixed-point sets. For example, there is a circle acting on $\HP^d$ that fixes a $\CP^d$. However, any $\gT^2$-action on $\HP^d$ already has  a disconnected fixed-point set. Third, while the existence of totally geodesic submanifolds $\CaP \supseteq \HP^2 \supseteq \CP^2 \supseteq \RP^2$ might suggest a counterexample to this lemma, this chain of submanifolds cannot arise as fixed-point sets of tori, since  $\RP^2$ is not orientable (cf. Theorem \ref{thm:FPSstructure}).

\begin{proof}[Proof of Lemma \ref{lem:t3connected}]
%
Let $F^f = M^{\gT^3}$. We apply the $\gS^1$-splitting theorem to the isotropy representation of $\gT^3$ at $F$. Combined with the Connectedness Lemma, it follows that for some two-dimensional subgroup $\gH\subseteq \gT^3$, the component $G \subseteq M^{\gH}$ containing $F^f$ has the following properties: $F \hookrightarrow G$ is $f$-connected, and the induced action by the circle $\gS^1 = \gT^3/\gH$ on $G$ is semi-free in a neighborhood of $F$.

Since the circle acts semi-freely in a neighborhood of $F$, 
we see that $F$ is a fixed-point component of
$\Z_2 \subset \gS^1$ in $G$. Moreover, $\Z_2$ has no other 
fixed-point components in $G$ because any
 other component would contain 
also a fixed point of $\gT^3$.

Assume now that $\chi(M)=\chi(F^f)$ is odd. A result of Hirzebruch and Conner-Floyd (see \cite[Corollary 6.16]{AtiyahSinger68-partIII}) implies that $F^f = G^{\Z_2}$ satisfies $f \geq \frac 1 2 \dim G$. Since the inclusion $F^f \subseteq G$ is $f$-connected, the Periodicity Lemma implies that $H^*(F;\Z)$ is $k$-periodic where $k = \cod(F\subseteq G)$. Finally, since $F$ is the connected fixed-point set of the above $\Z_2$-action on $G$ that is $f$-connected, a modification of the proof of \cite[Theorem 5.1]{Wilking03} implies that $H^k(G;\Z_2) = 0$. It follows that $G$ is a $\Z_2$-homology sphere and $F$ is as well.
In particular, $\chi(M) = \chi(F) = 2$, a contradiction.
\end{proof}

\begin{proof}[Proof that $\chi(M) \neq 3$]
We assume $\chi(M) = 3$ and proceed by contradiction. 
We claim that there is a three-dimensional subgroup $\gT^3$
with a fixed-point component $K$ that contains all fixed points of $\gT^5$ in $M$. 
Since any other fixed-point component of $\gT^3$ would also contain a fixed point of $\gT^5$,
the contradiction arises from Lemma~\ref{lem:t3connected}. 

The claim follows trivially if the fixed-point set of $\gT^5$ is connected.
Since any fixed-point component of $\gT^5$ of positive dimension has Euler characteristic $\ge 2$, there is at least
one isolated fixed point  and we may choose 
$\gT^4 \subseteq \gT^5$ such that the component $L \subseteq M^{\gT^4}$ 
containing the isolated fixed point of $\gT^5$ has positive dimension.
Moreover,  we may assume that $L$ and the manifold $N$ from Proposition~\ref{pro:4periodicN} 
have an intersection of positive dimension and thus there is a second fixed point of $\gT^5$ in 
$L$. 

If $L$ contains all fixed points of $\gT^5$, we are done. 
Otherwise, there is an isolated fixed point  $p$ of $\gT^5$ outside $L$. 
We can repeat the construction and pass to a codimension one group $\gT^3$ inside $\gT^4$ 
such that the fixed-point component of $\gT^3$ at $p$ also contains at least two fixed points 
of $\gT^5$.   Now the 
assumption that $\chi(M)=3$ implies that all fixed points 
of $\gT^5$ are contained in one fixed-point component of $\gT^3$, as claimed.
\end{proof}

\section{\sectfour}\label{sec:endgame}
The goal of the section is to rule out alternative b) in Proposition~\ref{pro:4periodicN} and thereby prove
 Theorem~\ref{thm:t5}.
 To avoid confusion, we should mention that we change notation: our new manifold $N$ 
 will take the role of $N_3$ from Proposition~\ref{pro:4periodicN},
 the fixed-point component of the third circle factor, endowed 
 with the remaining $\gT^2$-action. The new submanifolds $N_1$ and $N_2$ will then correspond to $N_1\cap N_3$ and 
 $N_2\cap N_3$. Everything takes place under the assumption that we are in the alternative scenario b).

\begin{theorem}\label{thm:endgame} Suppose that $N^{4n+h}$ has the rational cohomology of $\Sph^h\times \HP^n$ with $h\in\{2,3\}$.
We assume $N$ is endowed with an effective $\gT^2$-action with a fixed-point component 
$F_0$, such that $\gT^2$ splits in a neighborhood of $p\in F_0$ as a product action $\gS^1\times \gS^1$. 
We let $N_i$ denote the fixed-point component of the $i$-th factor and assume that $N_i$ 
has the cohomology of $\Sph^{h}\times \HP^{n-k_i}$ for $i=1,2$.
We also assume that $F_0=N_1\cap N_2$ has the cohomology of $\Sph^{h}\times \HP^{n-k_1-k_2}$
and that the inclusion $F_0\to N_2$ is rationally 
two-connected.
Then one of the following holds
\begin{enumerate}
\item[a)] There is a second fixed-point component $F_1$ with $H^*(F_1)\cong H^*(\Sph^u\times \CP^l)$ 
with $l>0$ and $u\in \{1,2,3\}$.
\item[b)] There is a circle $\gS^1\subset \gT^2$ fixing a component $B$ of dimension at least three such that the 
induced circle action $\gT^2/\gS^1$ has two fixed-point components of codimension two in $B$.
\end{enumerate}
\end{theorem}

The theorem remains true without assuming $F_0\cong \Sph^{h}\times \HP^{n-k_1-k_2}$
and $F_0\to N_2$ being two-connected, but the proof is easier in this case. 
In the presence of positive curvature these extra assumptions follow immediately from the Connectedness Lemma. 
 By the proof of Lemma~\ref{lem:shxHP}, 
 the maps $N_i\to M$ are rationally $4$-connected. 
 Thus $F_0\to N_2$ is rationally $2$-connected if and only if the same 
 holds for $F_0\to N_1$ or $F_0\to M$.
Theorem~\ref{thm:endgame} will follow from Lemma~\ref{lem:endgame} e) below.

\begin{proof}[{\bf Proof of Theorem A with Theorem~\ref{thm:endgame}}] As explained before the theorem, 
if the alternative scenario of 
Proposition~\ref{pro:4periodicN} b) holds, 
then there is a manifold $N$
 satisfying the assumptions of Theorem~\ref{thm:endgame} with the additional property 
 that it is a fixed-point component of a three-dimensional subgroup $\gH\subset \gT^5$ 
 where $\gT^5$ acts effectively and isometrically on a closed, positively curved and orientable manifold $M$. 
 The $\gT^2$-action on $N$ arises as the induced $(\gT^5/\gH)$-action.
 
If the conclusion of Theorem~\ref{thm:endgame} a) holds, then we have found another fixed-point component of $\gT^5$ 
 in $M$ whose topology was ruled out by Proposition~\ref{pro:4periodicN}. 
 But the conclusion of Theorem~\ref{thm:endgame} b) cannot hold either, since $B$ would violate the classification 
 of fixed-point homogeneous positively curved manifolds by Grove and Searle. 
 Therefore, the scenario b) of Proposition~\ref{pro:4periodicN} never occurs and Theorem~\ref{thm:t5} 
 follows. 
\end{proof}

The following lemma is known to experts but we include its proof for convenience.
\begin{lemma}\label{lem:alphaineqalphaj} Suppose $p_1$ and $p_2$ are fixed points 
of an action of $\gT^d$ on a manifold $M^n$.  
Let $\rs_i\colon H_{\gT^d}^*(M)\rightarrow H^*_{\gT^d}(p_i)\cong H^*(B\gT^d)$ be
the pullback homomorphism for $i=1,2$. 
\begin{enumerate}
\item[a)]  Then $\rs_1-\rs_2=0$ holds  if and only if $p_1$ and $p_2$ are contained in the same fixed-point component. 
\item[b)] Suppose $\alpha_0,\ldots,\alpha_k\in H^*_{\gT^d}(M)$  generate together with a basis of $H^2(B\gT^d)$
 the cohomology ring.
If $\rs_1(\alpha_i)=\rs_2(\alpha_i)$ for $i=1,\ldots,k$ but $\rs_1(\alpha_0)\neq \rs_2(\alpha_0)$, then 
$\rs_1(\alpha_0)-\rs_2(\alpha_0)\in H^*(B\gT^d)\cong \Q[t_1,\ldots,t_d]$ decomposes into linear factors.
\end{enumerate}
\end{lemma}

Part a) of the lemma came up  already under slightly more special circumstances in the proof 
of Lemma~\ref{lem:b3}.

\begin{proof}
{\em a).} If $p_1$ and $p_2$ are in the same fixed-point component $F_1$, we clearly have $\rs_1=\rs_2$ by the K\"unneth decomposition of $F_1 \times B\gT^d$.
If $p_1$ and $p_2$ are in different components, we can choose an $\gS^1\subset \gT^d$ such that 
this remains the case for the $\gS^1$-action. Hence, we may assume $d=1$. 
If $F$ denotes the fixed-point set of $\gS^1$, then in large degrees the map $H^j_{\gS^1}(M)\rightarrow H^j_{\gS^1}(F)$ 
is surjective. Since $p_1$ and $p_2$ are in different components of $F$, the map 
$H^j_{\gS^1}(F)\rightarrow H^{j}_{\gS^1}(\{p_1,p_2\})$ is surjective as well. Clearly, $\rs_1\neq \rs_2$ follows.

{\em b).} The assumptions imply that the image of $\rs_1-\rs_2$ is the principal ideal generated by 
$\rs_1(\alpha_0)-\rs_2(\alpha_0)$. 
We look at the Thom class of the  normal bundle of $F_1 $ which by excision can be seen 
as an element in $H^k_{\gT^d}(M, M\setminus F_1)$, where $k$ is the codimension of $F_1$ in $M$. Its image  $th\in H^k_{\gT^d}(M)$ 
pulls back to an element $\rs_1(th)\in H_{\gT^d}^k(p_1)$ given by the  product of the non-zero weights of the isotropy 
representation of $\gT^d$ at $F_1$. 
On the other hand, the pullback to $F_2$ is zero since it factors through $H^k_{\gT^d}(F_2, F_2)$. 
In summary, $\rs_2(th)=0$ and $\rs_1(th)$ is a non-zero element decomposing into linear factors. 
By our initial sentence, $\rs_1(\alpha_0)-\rs_2(\alpha_0)$ divides $\rs_1(th)$ and hence decomposes into linear factors as well.
\end{proof}

\begin{lemma}\label{lem:b2} Suppose a torus $\gT^d$ acts effectively on a closed manifold $M$.
If $M$ has the rational cohomology of $\Sph^2\times \HP^n$, 
then the following dichotomy holds.
\begin{itemize}
\item[(i)]  The map $H^2(M)\to H^2(F)$ is zero for all fixed-point components $F$ of $\gT^d$  or
\item[(ii)] the map $H^2(M)\to H^2(F)$ is injective 
 for all fixed-point components $F$ of $\gT^d$.
\end{itemize}
\end{lemma}

\begin{proof} 
We may choose a circle $\gS^1\subset \gT^d$ with the same fixed-point set as $\gT^d$. Therefore, without loss of generality $d=1$.
Let $F_0,\ldots,F_k$ denote the fixed-point components of 
$\gS^1$ in $M$. We may assume that $H^2(M)\rightarrow H^2(F_0)$ is injective 
and then have to show that the same holds  for all $i$. 
We let $\beta\in H^2_{\gS^1}(M)$ and $\alpha\in  H^4_{\gS^1}(M)$ denote lifts of 
generators $b \in H^2(M)$ and $a \in  H^4(M)$. 
We can normalize $\beta$ by adding a multiple of $t\in  H^2(B\gS^1)$ in order to 
arrange that, under the restriction $s\colon H^2_{\gS^1}(M)\rightarrow H^2(F_0\times B\gS^1)$, it is mapped to an element of $H^2(F_0)$ with respect to the K\"unneth decomposition of the product. By assumption, we have $s(\beta)\neq 0$, since $s(\beta)$ is also the restriction of $b$ to $F_0$. 
By the same reasoning $s(\beta)^2=0$. 
We claim that $\beta^2=0$. 
Since $t^2$, $\beta t$, $\alpha\in H^4_{\gS^1}(M)$ 
is a basis, we have 
$\beta^2 = q_1 t^2 + q_2 \beta t + q_3 \alpha$ where $q_i \in \Q$. Using that $b$ is the pullback of $\beta$ and $a$ that of $\alpha$, we see $q_3=0$.
Now if $(q_1, q_2) \neq 0$, we would get $s(\beta^2)\neq 0$ -- a contradiction. 
Thus indeed $\beta^2=0$. 
It now follows that the image of $\cup \beta\colon H^*_{\gS^1}(M)\to H^*_{\gS^1}(M)$, $x\mapsto x\beta$, coincides with its kernel.

  We know that the dimension of $H^{2j}_{\gS^1}(M)$ is equal to the total Betti number $b(M)$ of $M$, which is equal to the total Betti number of $F=\bigcup F_i$, for $j>\dim(M)$. Therefore, dimension-wise exactly half of $H^{2j}_{\gS^1}(M)$  is in the image 
  of $\cup \beta$. Since $H^{2j}_{\gS^1}(M)\to H^{2j}_{\gS^1}(F)$ 
  is an isomorphism, the same must hold for the pullback of $\beta$ to $F$. 
  Using that the pullback again squares to zero, this can only be achieved if, for 
  each connected component, dimension-wise half of $H^{2j}_{\gS^1}(F_i)$ is in the image of 
  $\cup s_i(\beta)$, where $s_i\colon  H^2_{\gS^1}(M)\rightarrow H^2(F_0\times B\gS^1)$ 
  is the restriction. In particular,  $s_i(\beta)\neq 0$. 
  Since $s_i(\beta)^2=0$, we see that $s_i(\beta)$ represents 
  a non-zero element in $H^2(F_i)$ as claimed.
\end{proof}

\begin{lemma}\label{lem:endgame} Suppose we are in the situation of Theorem~\ref{thm:endgame}. Let $F_1,\ldots,F_k$ 
denote the other fixed-point components of $\gT^2$. Choose $p_i\in F_i$ 
and let $\rs_i\colon H^*_{\gT^2}(N)\rightarrow H^*_{\gT^2}(p_i)\cong H^*(B\gT^2)$
denote the restriction for $i=0,\ldots,k$.
\begin{enumerate}
\item[a)]  After possibly reordering, we have $F_1,\ldots, F_\ell\subset N_1$ and 
$F_{\ell+1},\ldots,F_k\subset N_2$ for some $0<\ell<k$. 
\item[b)]  There are lifts $\beta \in H^h_{\gT^2}(N)$ and 
$\alpha \in H^4_{\gT^2}(N)$ of generators in $H^h(N)$ and $H^4(N)$, respectively.
We can arrange for $\alpha$ and $\beta$ to be in the kernel of $\rs_0$. We then have $\rs_i(\beta)=0$ 
and put $\alpha_i=\rs_i(\alpha)$ for all $i>0$.
\item[c)]  After possibly adjusting the scaling of $\alpha$, we get $\alpha_i=c_i^2 t_1^2$ for $i\le \ell$ and $\alpha_i=c_i^2t_2^2$ for $i>\ell$, where $t_1,t_2\in H^2(B\gT^2)$ is a basis and $c_i\in \fQ$.
\item[d)]  For each $i>0$, $F_i$ is a rational cohomology $\Sph^{h_{1}}\times \CP^{n_i}$ for some $h_1\in \{1,2,3\}$ and $n_i\ge 0$.
\item[e)] If $n_i=0$ for all $i>0$, or equivalently if the conclusion of 
\ref{thm:endgame} a) does not hold, then the rational weights of the isotropy representation 
at $F_1$ are given by $t_1$ 
and  $(c_1t_1\pm c_it_2)$ for $i=\ell+1,\ldots,k$. The latter weights 
are pairwise linearly independent and have multiplicity one, and the conclusion of \ref{thm:endgame} b) holds. 
\end{enumerate}
\end{lemma}

 \begin{proof}

{\em a)}  As was shown in the beginning of the proof of Lemma~\ref{lem:shxHP}, 
our assumption of a non-empty fixed-point set implies that  
the circle action on $N_u$ has a fixed-point set with total Betti number  equal to the total 
Betti number $ b(N_u)$ of  $N_u$ for $u=1,2$.  Using $b(N_u)=\tfrac{4n-k_u}{2}+2$ 
and the fact that only the total Betti number of $F_0=N_1\cap N_2$ is counted twice, we
 see that the total Betti number of components that can either be found in $N_1$ or in $N_2$ add up to
$ 2n+2$, which is the total Betti number of $N^{4n+h}$. By the Smith-Floyd inequalities there cannot be any other 
components. We also deduce from the equality case that $H^*_{\gT^2}(N)\to H^*(N)$ is surjective.
 
  {\em b)} Thus there are lifts $\alpha$ and $\beta$. We can add  an element of $H^4(B\gT^2)$ 
  to $\alpha$ in order to arrange for $\rs_0(\alpha)=0$.  If $h=3$, then
  $\rs_i(\beta)=0$ follows for degree reasons and we are done.
  If $h=2$,  
  then $\rs_0(\beta)=0$ follows after modifying $\beta$ 
  by a summand in $H^2(B\gT^2)$. It remains to show $\rs_i(\beta)=0$ for all $i$. This clearly  follows 
  if we can prove $\beta^2=0$.
    On the other hand, we know from the assumptions 
  of Theorem~\ref{thm:endgame} that the restriction of $\beta$ to
   $H^2(F_0)$ is non-zero.
  Now $\beta^2=0$ follows nearly directly 
  from the proof of Lemma~\ref{lem:b2} where we showed it
  under the same 
  normalization. The difference is that in that proof only a circle action was present.  But if $\beta^2\neq 0$, it is easy
   to find a circle $\gS^1\subset \gT^2$ such that the image of $\beta$ in $H^2_{\gS^1}(N)$ squares
    to non-zero element and we would get a contradiction to the result established in 
    the proof of Lemma~\ref{lem:b2}.

{\em c)}  We can specify our normalization of $\alpha$ 
slightly more  since we can still add elements of the form $t\beta$ to $\alpha$ with
$t\in H^2(B\gT^2)$. This way, we can ensure that under the restriction  
$j\colon H^*_{\gT^2}(N)\to H^*(F_0\times B\gT^2)$, the element $\alpha$ is mapped 
to  an element of $H^4(F_0)$ with respect to the K\"unneth decomposition of the product. 
After this normalization, $\alpha$ is determined up to a rational factor.

Next recall that  $N_u$ is fixed by a circle $\gS^1_u\subset \gT^2$. 
Let $\gQ_u:= \gT^2/\gS^1_u$ denote the quotient group. There is a natural 
 map $H^{*}_{\gQ_u}(N_u)\to H^*_{\gT^2}(N_u)$ induced by the
   fibration 
 \[
B\gS^1_u\to E\gT^2\times_{\gT^2} N \simeq (E\gQ_u\times E\gT^2\times N_u)/\Delta \gT^2\to 
E\gQ_u\times_{\gQ_u}N_u
\]
for $u=1,2$. We claim that the restriction $\bar \alpha_u\in H^4_{\gT^2}(N_u)$ of  $\alpha$ is in the image of $H^4_{\gQ_u}(N_u)$. 
To see this, we just observe that in $H^4_{\gQ_u}(N_u)$, we can also find a non-zero element 
whose restriction to $F_0\times B\gQ_u$ is given by an element in $H^4(F_0)$. 
Clearly the image of this element must coincide with $\bar \alpha_u$ up to a factor.
Now if $F_i\subset N_u$, then $\alpha_i$ must be in
 the one-dimensional 
image of $H^4(B\gQ_i)$ in $H^4(B\gT^2)$. 
Therefore, we see that $\alpha_i=d_i t_1^2$ for $i\le \ell$ and $\alpha_i=d_it_2^2$ 
for $i>\ell$. From Lemma~\ref{lem:alphaineqalphaj} a), we know $d_1,\ldots,d_\ell\in \fQ\setminus \{0\}$ 
are pairwise different and the same holds for $d_{\ell+1},\ldots,d_k$. 
After scaling  $\alpha$, we may assume that $d_1=c_1^2$ is a square.

\hspace*{0em}From Lemma~\ref{lem:alphaineqalphaj} b), we know that $(c_1^2t_1^2-d_it_2^2)$ 
decomposes into linear factors, and clearly this implies that $d_i=c_i^2 $ must be a square as 
well for $i=\ell+1,\ldots,k$. By now replacing  the index $1$ with $\ell+1$ we see that 
$d_i=c_i^2$ is also a square for $i=2,\ldots,\ell$.

  {\em d)}  We first claim that each fixed-point component $F_i$ has the cohomology of 
  $\Sph^{h_i}\times \HP^{n_i}$ or $\Sph^{h_i}\times \CP^{n_i}$ with $h_i\in\{1,2,3\}$. 
  In the case of  $h=3$, this is a direct consequence of  Lemma~\ref{lem:shxHP}.
  In the case of $h=2$ we can use Lemma~\ref{lem:b2}  
  to see that for any fixed-point component $F_i$ 
  the element $\beta$ restricts to a non-zero class in $H^2(F_i)$, since this is true for $i=0$. 
  By the proof of  Lemma~\ref{lem:shxHP}, this implies our claim.

    After reordering, 
  $h_1=\min \{h_1,\ldots,h_k\}$. We will prove 
  $F_i\cong \Sph^{h_1}\times \CP^{n_i}$ for all $i>\ell$.
   The same then follows for $0<i\le \ell$, by switching the roles 
   of $F_1$ and $F_k$.  
  We can completely determine the isotropy representation of 
  $\gT^2$ at $F_1$. Notice that the weight $t_1$ must occur with multiplicity 
  $\tfrac{\dim(N_1)-\dim(F_1)}{2}$, and all other non-trivial  irreducible subrepresentations of $\gT^2$ 
  are in the normal space of $N_1$. 
  To see that the weights $(c_1 t_1 \pm c_it_2)$ must occur 
  for all $i>\ell$, we
   let $\gS^1_{i\pm}$ be the circle  contained in the kernel of the weight.
   Since each of these
    weights divides $\alpha_1-\alpha_i$,
   the elements $\alpha_1$ and $\alpha_i$ 
   pullback to the same element of $H^*(B\gS^1_{i\pm})$. 
   Using that  the pullbacks of $\beta$ and $\alpha$ together with $t\in H^2(B\gS^1_u)$ generate the cohomology ring of $H^*_{\gS^1_{i\pm}}(M)$, and using the fact that $\beta$ is in the kernel of $\rs_i$, we infer from Lemma~\ref{lem:alphaineqalphaj}
    that the fixed-point component $B_{i\pm}$ of $\gS^1_{i\pm}$ at $F_1$ 
   contains $F_i$ as well.  
   
   We next claim that the multiplicity of the weights $c_1 t_1 \pm c_it_2$ 
   is at least $n_i+1$, which is half of the total Betti number of $F_i$, and equality can only  hold if $F_i$ is of the form 
   $\Sph^{h_1}\times \CP^{n_i}$.
   The inequality is a straightforward consequence of the fact that 
   the total Betti number 
  $b(B_{i\pm})$ of $B_{i\pm}$ 
  is at least $b(F_1)+b(F_i)$.  Using  
    $b(F_i)=2+2n_i$, we see that the dimension of  $B_{i\pm}$ is at least 
    $2(n_i+1)+\dim(F_1)$. Furthermore, by Lemma~\ref{lem:shxHP}, 
    equality can only hold if 
    $B_{i\pm}\cong \Sph^{h_1}\times \CP^{n_1+n_i+1}$. 
    But this is impossible if
     $F_i\cong \Sph^{h_i}\times \HP^{n_i}$ or if 
     $F_i\cong \Sph^{3}\times \CP^{n_i}$ while $h_1=1$. 
     In fact, in the former case $H^*(B_{i\pm})$ cannot be generated in degrees $<4$ and in the latter case not in degrees $< 3$.

   Thus indeed each of the two weights $c_1 t_1 \pm c_it_2$ has multiplicity at least  $(n_i+1)$.  
   These weights are pairwise  linearly independent, and the corresponding
    subrepresentations 
   are contained in the $k_1$-dimensional normal space of $N_1$.
   Since $k_1$ is also the codimension of $F_0$ in $N_2$, 
   $$k_1/2=b(N_2)-b(F_0)=\sum_{i=\ell+1}^kb(F_i)=\sum_{i=\ell+1}^k 2(n_i+1).$$
   Hence, equality must hold everywhere
  and $F_i\cong \Sph^{h_1}\times \CP^{n_i}$.

  {\em e)}  is now a consequence of the above calculation. 
  If $n_i=0$ for all $i$, then the weight $(c_1 t_1\pm c_it_2)$ of the isotropy representation 
  of $\gT^2$ at $F_1$ occurs with multiplicity $1$.  
  Thus we get that $B_{i\pm}\cong \Sph^{h_1}\times  \CP^1$, and there is a second fixed-point component 
  $\Sph^{h_1}\cong F_i$ of codimension two for the induced circle action on $B_{i\pm}$.
 \end{proof}

 \subsection{\sectfourone}\label{sec:addendum}
 We actually proved a stronger statement 
 
 \begin{addendum} 
If $\gT^5$ acts effectively by isometries on a connected, closed, orientable, positively curved Riemannian manifold, then every fixed-point component $F$ is contained in a fixed-point
component $L$ of a three-dimensional subgroup $\gH$ satisfying the following:
  The effective action of $\gT^5/\gH$ on $L$ splits as product action
in a neighborhood of a point $p\in F$, and the inclusion map $L_i\to L$
of the fixed-point component $L_i$ of the $i$-th circle  factor is $\dim(L_i)$-connected for $i=1,2$.
Finally, $L$ and its universal cover have the rational cohomology 
of $\Sph^{\ell}$, $\CP^{\ell/2}$, or $\HP^{\ell/4}$.
 \end{addendum}
 By the Connectedness Lemma, it also follows that $F\rightarrow L$ is $\dim(F)$-connected, and thus the universal covers of $F$, $L_1$, and $L_2$ are rationally equivalent to a rank one symmetric spaces, too.
 \begin{proof} We consider the manifold $N$ from Proposition~\ref{pro:4periodicN}. We let $N_i$ 
 denote the fixed-point component of the $i$-th circle factor for $i=1,2,3$, and we assume the codimensions satisfy $k_1\le k_2\le k_3$.  
 Since the conclusion of Proposition~\ref{pro:4periodicN} a) holds
 and the inclusion map of $N_i\to N$  is $3$-connected for $i\le 2$, 
 it follows that $L:=N_3$ has the rational cohomology of a rank one symmetric 
 space. Moreover, the inclusion maps $L_i:=N_3\cap N_i\to L$ 
 are $\dim(L_i)$-connected by the Connectedness Lemma for $i=1,2$.  
 To show  the universal cover $\tilde L$ has the same rational cohomology as $L$ we argue by contradiction.
  By Synge's theorem, $L$ is
  odd-dimensional and hence
  $k_1<\dim(L)/2$ holds. 
  By Corollary~\ref{cor:transversal},  $\tilde L$ has $k_1$-periodic 
  rational cohomology and thus $H^{k_1}(\tilde L,\Q)\cong \Q$. 
  Since the deck transformation group 
 acts trivially on $H^{2k_1}(\tilde L,\Q)\cong \Q$, the quotient 
 $L$ cannot possibly be a rational homology sphere -- a contradiction.
 \end{proof}


\section{\sectfive}\label{sec:recognition}
The main aim of this section is to prove  
Theorem \ref{thm:t8NoCurvature} under the assumptions of b) and also in odd dimensions.
Suppose a torus $\gT^d$ acts effectively on a closed, orientable manifold $M^n$ with non-trivial fixed-point set $F$. 
We recall that the $k$-skeleton of $\gT^d$ consists of those points 
whose $\gT^d$-orbit has dimension $\le k$. It is therefore the finite union 
of fixed-point components of  subgroups of codimension $k$.

\begin{definition}(Assigning a graph to the one-skeleton)\label{def:graph}\begin{enumerate}
\item[a)]
If $N^m$ has the rational cohomology of an
 even-dimensional rank one symmetric space, we define $\mult(N)\in \{1,2,4,m/2\}$ 
by the property that the cohomology of $N$ is generated in degree $2\mult(N)$.
\item[b)]
Suppose an action of $\gT^d$ on a closed, orientable orbifold has the property that any codimension one torus has only  fixed-point components  with the rational cohomology of $\CP^k$, $\HP^k$, or $\Sph^{2k}$.  
Consider a connected component $C$ of the one-skeleton. 
 We let $F_1,\ldots,F_{k_0}$ denote the different components of $F\cap C$. 

We define a graph with $k_0$ vertices as follows. For each codimension one torus $\gR\subset \gT^d$, 
we choose a weight $r\in \Lt^*$ whose kernel is given by 
the Lie subalgebra of $\gR$. 
We put $\mult(N)$ edges between the $i$-th and the $j$-th point with label $r$ if 
 $F_i$ and $F_j$ are contained in the same fixed-point component $N$ of $\gR$.

If the fixed-point component at $F_i$ has positive dimension, we put $\mult(F_i)$ edges from the $i$-th vertex to itself 
with label $0$. Moreover, if $F_i$ has positive dimension and $N\supseteq F_i$ is the fixed-point component of a codimension one 
torus $\gR$ we put $(\mult(N)-\mult(F_i))$ edges 
with label $r$ from $i$ to itself.
\end{enumerate}
\end{definition}

 Note by compactness that there are only finitely many isotropy groups in $\gT^d$. Therefore, the graph has only finitely many edges. 

A torus action  is called GKM if every connected component of the one-skeleton that intersects $F$ 
consists of a finite union of $2$-spheres. The graph assigned 
to this connected component then just corresponds to the orbit space of the one-skeleton.

\begin{lemma}\label{lem:graphofmodel} Suppose $\gT^2$ acts effectively on a manifold $M^n$.
\begin{enumerate}
\item[a)] If $M^n$ has the rational cohomology of $\CP^{n/2}$, then 
the  one-skeleton is connected, and the assigned graph is given by the one-skeleton of 
a simplex with $k_0$-vertices with the additional property that,  
for each fixed-point component $F_i$ of positive dimension, there is one edge from 
$i$ to itself.
\item[b)] If $M^n$ has the rational cohomology of $\HP^{n/4}$, then 
the  one-skeleton is connected, and the assigned graph looks as in a) except with exactly  two edges in place of one. 
\end{enumerate}
\end{lemma}

%

With this in mind, the proof of the following crucial lemma is fairly straightforward. It generalizes 
a similar statement on GKM graphs in Goertsches and Wiemeler, where they assumed in addition that every three 
weights of the isotropy representation at any fixed point are linearly independent.

\begin{lemma}\label{lem:straight} Suppose an action of $\gT^d$ on a closed manifold $M^n$ has the property that any codimension two torus only has fixed-point components of rational type $\CP^k$, $\HP^k$, or $\Sph^k$. 
Choose $C$ and $F_1,\ldots, F_{k_0}$ as in Definition \ref{def:graph}. 
\begin{enumerate}
\item[a)] The graph and the dimensions of the fixed-point components completely determine the isotropy representation of $\gT^d$  at $F_i$ up to rational equivalence. Here we call two irreducible representations of $\gT^d$ rationally equivalent if their kernels have the same identity component. 
\item[b)]
There exists $m\geq 1$ such that there are precisely $m$ edges between any two vertices, and from the $i$-th vertex to itself whenever $\dim(F_i) > 0$. In addition, we have $\chi(C) = \sum_{i=1}^{k_0} \chi(F_i) = \tfrac{n}{2m} + 1$.
\item[c)] If some fixed-point component $F_i$ has $\dim(F_i) \geq 4$, then $m$ is $1$, $2$, or $\frac n 2$.
\end{enumerate}
\end{lemma}

 If $m \in \{1,2,\frac{n}{2}\}$ or $(m,n) = (4,16)$, and if the odd Betti numbers vanish, then we can recover the topology of the underlying manifold. Indeed, for $m = \frac{n}{2}$, the Euler characteristic is two by the lemma and hence $M$ is a rational homology sphere. Similarly, if $m = 4$ and $n = 16$, then the total Betti number is three, so $M$ is a rational cohomology Cayley plane by Poincar\'e duality. For the cases $m \in \{1,2\}$, the result is non-trivial and is covered by the following two lemmas. 


\begin{lemma} \label{lem:HP} Suppose  $\gT^d$ with 
$d\ge 4$ acts effectively on a connected, closed, orientable manifold $M^{n}$ with $b_{odd}(M)=0$ such that 
all codimension two tori have fixed-point components of $\CP$, $\HP$, or $\Sph$ type. 
If $\chi(M^{n})=\tfrac{n}{4}+1$, then $H^*(M^{n})\cong H^*(\HP^{n/4})$.
\end{lemma}

\begin{lemma} \label{lem:CP}  Suppose  $\gT^d$ with 
$d\ge 3$ acts effectively on a connected, closed, orientable manifold $M^{n}$ with $b_{odd}(M)=0$ such that 
all codimension two tori have fixed-point components of $\CP$, $\HP$, or $\Sph$ type. 
If $\chi(M^{n})=\tfrac{n}{2}+1$, then  $H^*(M^{n})\cong H^*(\CP^{n/2})$.
\end{lemma}

\begin{corollary}\label{cor:codim2CP} Suppose $\gT^3$ acts on a connected, orientable, closed manifold $M^{n}$ with $b_{odd}(M)=0$ such that all circles only have fixed-point components whose rational cohomology is of  $\CP$ type,
 then $M^{n}$ has the rational cohomology of $\CP^{n/2}$.
\end{corollary}

In addition to the tools from equivariant cohomology used to 
prove Theorem \ref{thm:t5}, these results  require a fundamental result of Chang and Skjelbred \cite{ChangSkjelbred74}. It provides a way to lift information about the equivariant cohomology of the fixed-point set $M^{\gT}$ and the one-skeleton $\{x \in M \st \dim(T\cdot x) \leq 1\}$ to the ambient manifold $M$.

\begin{theorem}[Chang-Skjelbred Lemma]\label{lem:ChangSkjelbred}
Let $M$ be a connected, orientable $\gT$-orbifold of finite type, 
that is there is a compact $\gT$-invariant subset whose 
inclusion map is a homotopy equivalence. 
Suppose that the map $H^*_{\gT}(M)\to H^*(M)$ is surjective. Then $M^{\gT}$ is contained in one component $C$ of the one-skeleton. In addition, the
restriction map $H^*_{\gT}(M) \to H^*_{\gT}(M^{\gT})$ is injective, and its image is the same as the image of the restriction map $H^*_{\gT}(C) \to H^*_{\gT}(M^{\gT})$.
\end{theorem}
The surjectivity of the map
$H^*_{\gT}(M)\to H^*(M)$ is by definition equivalent 
to equivariant formality. It is of course automatic if the odd Betti numbers of $M$ vanish. It is also equivalent to saying that the fixed-point set has the same total Betti number as $M$.
We should mention that the first and last conclusion of the theorem hold slightly more generally. Instead of requiring equivariant formality 
one can demand that $H^*_\gT(M)$ becomes a free $H^*(B\gT)$-module after one divides out the $H^*(B\gT)$-torsion elements.
Moreover, all conclusions of the theorem hold if and only if \(H^*_{\gT}(M)\) is reflexive
(see \cite[Corollary 1.3]{AlldayFranzPuppe14}). 
However, we will use the theorem only in the above 
form. The reason that we state it for orbifolds rather than closed 
manifolds is that in the odd-dimensional case we will divide 
out an almost free subtorus action on a subset of $M$ and verify 
the assumptions of the Chang-Skjelbred Lemma for the induced 
torus action on the quotient.

Using this result, Corollary~\ref{cor:codim2CP} 
 is an immediate consequence of Lemma~\ref{lem:CP} and Lemma~\ref{lem:straight}.
In fact, by the  Chang-Skjelbred Lemma  the one-skeleton is connected, and we can 
then use Lemma~\ref{lem:straight} to see that the Euler characteristic is $\frac n 2+1$, since, in the case of $m\ge 2$,  
we could find a codimension two torus whose fixed-point component is not of  $\CP$ type.

We now prove the lemmas in order of occurrence. In Section~\ref{sec:t8sphere}, we explain how Theorem~\ref{thm:t8NoCurvature}
under the assumptions of b) follows in even dimensions.
The proof of Theorem~\ref{thm:t8NoCurvature} in odd dimensions is quite different and is carried out in Section~\ref{subsec:odd}.
Throughout this section all cohomology will be with  rational coefficients.

\subsection{\subsecmodel} \label{subsec:model} We prove Lemma~\ref{lem:graphofmodel}. 
Let $M^n$ denote an orientable manifold whose rational cohomology is generated by one element in degree $w\in \{2,4\}$.
We will  frequently use  a result of Bredon (see Theorem~\ref{thm:SinglyGenerated})
implying that any fixed-point component of a circle action on $M$ 
has cohomology generated by an element in degree $2$ or $w$.
The rational cohomology $H^*(B\gT^d)$ is a polynomial ring of $d$ variables of degree 2. Any homomorphism 
$h\colon \gT^d\rightarrow \gS^1$ induces naturally a homomorphism 
$H^*(B\gS^1)\rightarrow H^*(B\gT^d)$, and we can think of the 
image of a generator $t\in H^2(B\gS^1)$ as the rational weight of this 
representation. Thus any non-zero element in $H^2(B\gT^d)$ 
determines uniquely a codimension one torus, the identity component of the kernel
of the corresponding homomorphism. 
In the situation of the lemma, we have a $\gT^2$-action.

We let $\alpha\in H^w_{\gT^2}(M)$ denote a lift of the generator 
of $H^w(M)$. Then $\alpha$ together with a basis of $H^2(B\gT^2)$ 
generate  the equivariant cohomology of $M$.  
In each of the fixed-point components $F_1,\ldots, F_{k_0}$, 
we choose a point $p_i\in F_i$ and 
let $\alpha_i\in H^w(B\gT^2)\cong \Q[t_1,t_2]$
 denote the pullback of $\alpha$  to $H^w(p_i\times B\gT^2)\cong H^w(B\gT^2)$.
 We can employ Lemma~\ref{lem:alphaineqalphaj} to see 
 that $\alpha_i\neq \alpha_j$ and that $\alpha_i-\alpha_j$ decomposes 
 into linear factors for $i \neq j$. \\[1ex]
{\bf Claim 1.} Suppose $w=2$.   
The graph assigned to the one-skeleton has exactly one edge from $i$ to $j$, and the label is a multiple of $\alpha_i-\alpha_j$ for $i\neq j$. 
\\[1ex]
Let $\gH$ be the codimension one torus determined by
  $\alpha_i -\alpha_j \in \Lt^*$. 
Then $\alpha_i$ and $\alpha_j$ pull back to the same element of $H^*(B\gH)$. By Lemma~\ref{lem:alphaineqalphaj}, $F_i$ and $F_j$ must be contained in the same 
fixed-point component of $\gH$.

For any other codimension one torus $\hat \gH$, the pullback of $\alpha_i -\alpha_j$ to 
$H^*(B\hat \gH)$ remains non-zero, and hence $p_i$ and $p_j$ are not in the same fixed-point component of 
$\hat \gH$.\\[1ex]
{\bf Claim 2.} Suppose $w=4$, and let $r_{ij}^\pm$ denote the (possibly linearly dependent) linear factors of $\alpha_i - \alpha_j$ for $i \neq j$. The graph assigned to the one-skeleton has exactly two edges from $i$ to $j$, and the labels are multiples of $r_{ij}^\pm\in H^2(B\gT^2)$.\\[1ex]
Let $\gH^\pm$ be the codimension one tori whose Lie algebras are given by the kernels of $r_{ij}^\pm$.
Then $\alpha_i$ and $\alpha_j$ pull back to the same element in $H^*(B\gH^\pm)$.  As before, this implies  that $F_i$ and $F_j$ are in the same fixed-point component $N^\pm$ of $\gH^\pm$.
We can finish the proof by showing that  $N^\pm$ is of $\HP$ type if and only if $\ml r_{12}^\pm \mr$ is one-dimensional. 

Suppose first that $N^+$ is of $\HP$ type. 
Notice that $\alpha_i-\alpha_j$ remains unchanged 
if we add an element of $H^4(B\gT^2)$ to $\alpha$ before 
defining $\alpha_i$ by the pullback. On the other hand, 
we can use this flexibility to arrange that 
the pullback $\hat{\alpha}$ of $\alpha$ to 
 $H^4_{\gT^2}(N^+)\cong H^4(E\gS^1\times_{\gT^2/\gH^+} N^+ \times B\gH^+)$ lies 
 in $H^4_{\gT^2/\gH^+}(N^+)$. Recall that, as in the proof 
 of Lemma~\ref{lem:endgame}, there is a natural 
 injection of $H^4_{\gT^2/\gH^+}(N^+)$ into $H^*_{\gT^2}(N^+)$
induced by the fibration $E\gT^2 \times_{\gT^2}N^+\to
  E\gT^2/\gH^+ \times_{\gT^2/\gH^+}N^+$.
 Hence, $\alpha_1-\alpha_2\in H^4(B\gT^2/\gH^+)$ is given up to a rational factor by the square of a generator
  of $H^2(B\gT^2/\gH^+)\cong \Q$.

It remains to see that, if $N^+$ is of $\CP$ type, then $r_{ij}^\pm$ are linearly independent. 
We look again at the image $\hat \alpha$ of $\alpha$ in the equivariant cohomology 
$H^4_{\gT^2}(N^+)$.
Here we can use the decomposition arising from the fibration 
 to say $\hat \alpha= u\cdot t+q u^2$,
where $u\in H^2_{\gT^2/\gH^+}(N^+)$ is a generator, $t\in H^2(B\gT^2)$, and $q \in \Q$.
It is well known that the image of $u\cdot t$ in $H^4_{\gH^+}(N^+)$ 
does not vanish and in particular $t\not\in H^2(B(\gT^2/\gH^+))$.
 If we let $\beta_i\in H^2_{\gT^2}(p_i)$ denote 
the image of $u$, then $\beta_i \in H^2(B\gT^2/\gH)$ and hence
$\beta_i$ and $\beta_j$ are multiples of each other with $\beta_i\neq \beta_j$. 
Since $\alpha_u=\beta_ut+q \beta_u^2$ for $u=i,j$, 
this shows that the difference $\alpha_i -\alpha_j$ has a second linearly independent factor.

Finally, we have to determine the number of edges from $i$ to itself. 
If $F_1$ is a point, there are clearly no edges from $1$ to itself. 
If $F_1$ has non-trivial cohomology generated in degree $w$, then it is also clear that 
there cannot be any edge from $i$ to itself with a non-zero weight because 
the corresponding fixed-point component would violate the result of Bredon. It remains to check that, if $F_1\neq pt$ is of $\CP$ type and $w=4$, 
then there is precisely one edge from $1$ to itself with a non-zero weight.
We look again at the image $\hat \alpha$ of $\alpha$ in the  cohomology 
$H^4(F_1\times B\gT^2)$. After possibly choosing a different generator $ \alpha$ we may assume $\alpha_1=0$, and
we can use the K\"unneth decomposition of the product to say $\hat \alpha= u\cdot t+q u^2$,
where $u\in H^2(F_1)$ is a generator, $t\in H^2(B\gT^2)$, and $q \in \Q$.
It is well known that $u\cdot t\neq 0$. For any codimension one torus, the fixed 
point component at $F_1$ remains of $\CP$ type unless we take the one determined by $t$. 
So $t$ is the only non-zero label that occurs as a label of an edge from $1$ to itself.

\subsection{\subsecstraight}\label{subsec:straight} We prove Lemma~\ref{lem:straight}.
 We are given an action by a torus $\gT^d$ on a manifold $M^n$ with the property that every codimension two torus has fixed-point components of rational type $\CP^k$, $\HP^k$, or $\s^k$. We fix a component $C$ of the one-skeleton, and we let $F_1,\ldots,F_{k_0} \subseteq M^{\gT^d}$ denote those fixed-point components lying in $C$.
We will often pass to a fixed-point component $N$  at $F_i$ of 
a subtorus $\gR$ of codimension two. 
The graph assigned 
to the connected component $C'\supseteq F_i$ of the one-skeleton 
 of the induced $(\gT^d/\gR)$-action on $N$ is 
 a subgraph of the graph of $C$. We  will frequently use that 
  this  {\em reduced 
 graph } is subject to the conclusion of Lemma~\ref{lem:graphofmodel}.

{\em a).} In order
to see that the isotropy representation is determined, we can 
actually work with the weaker assumptions of 
Definition~\ref{def:graph}. Consider an 
irreducible subrepresentation of the isotropy representation of $\gT^d$ at $F_i$. 
We can consider the fixed-point component $N$ at $F_i$ of the identity component of the kernel of this 
subrepresentation.  If $N$ contains another component $F_j$, then the weight of this subrepresentation 
occurs as a label on an edge between $i$ and $j$. If it does not, then it must occur as a label of an edge  from 
$i$ to itself.  It remains to see with which multiplicity each of these subrepresentation occurs. This is equivalent to 
determining the dimension of $N$, because the multiplicity is
given by $\tfrac{\dim(N)-\dim(F_i)}{2}$. 
Notice that we can read off the graph which fixed-point components are contained in $N$. 
In particular, we know the Euler characteristic of $N$. Since we also can determine  the number $\mult(N)$ by the graph
 and $\chi(N)=\tfrac{\dim(N)}{2\mult(N)}+1$,
 this allows us 
 to read off the dimension of $N$.
 
{\em b).} To see that the number of edges  between  any two vertices is the same, we argue as follows. 
  We prove this for  $1$, $2$, and $3$, and we first consider the subcase that
 there is a fixed-point component $N$ of a  codimension one torus $\gT^{d-1}$ containing
 $F_1$, $F_2$, and $F_3$. If $N$ is of $\HP$ type, then there cannot be any other edges 
 from $i$ to $j$ for all distinct $i,j \in\{1,2,3\}$, because otherwise 
 we could find a codimension two torus for which the reduced graph has more 
 than two connections between $1$, $2$, and $3$. 
 If $N$ is of $\CP$ type, then any codimension two torus $\gH\subset \gT^{d-1}$ 
 whose fixed-point component contains one additional edge between $1$ and $2$
 also contains exactly one additional edge between $2$ and $3$ and between $1$ and $3$.
 This gives the needed correspondence.

 It remains to consider the subcase that $F_1$, $F_2$, and $F_3$ 
 are not all contained in the fixed-point component of a codimension one group. 
 We fix an edge with label $r$ between $1$ and $3$, and want to prove that the number of edges between $1$ and $2$ and between $2$ and $3$ is the same. Given $r$, there is a natural division of labels between $1$ and $2$ 
 into two disjoint subsets. The first subset consists of those labels $s_1$ such that we can find another label $s_2$ with the property that $s_1$, $s_2$, and $r$ are linearly dependent. Then $s_2$ is uniquely determined by $s_1$. The fixed-point component of the codimension two torus determined by the kernels of $r$ and $s_1$ also has exactly two edges labeled $u_1$ and $u_2$ between $2$ and $3$, and of course $u_1$, $u_2$, and $r$ are linearly dependent as well. This shows that $r$ induces a canonical mapping of pairs of edges between $1$ and $2$ in the first category to corresponding pairs between $2$ and $3$. 
 
The remaining labels $s$ have the property that $s$, $r$, and $s_2$ are linearly independent for any choice $s_2\neq s$. But then the fixed-point component of the codimension two torus determined by $s$ and $r$ has a fixed-point component of $\CP$ type and we again get a one-to-one mapping to corresponding edges between $2$ and $3$. 
 
The argument for edges from $i$ to itself is similar.

Clearly $\chi(\bigcup_{i=1}^{k_0} F_i)=2$ holds if and only if the $m=\tfrac{n}{2}$ and the graph has either one or two vertices. 
If $\chi(\bigcup_{i=1}^{k_0} F_i)\ge 3$, then all fixed-point components 
of codimension one tori are of $\CP$ or $\HP$ type.
To compute the Euler characteristic of $\bigcup_{i=1}^{k_0} F_i$, assume first that $F_1$ consists of a single point. 
We look at all codimension one tori whose fixed-point component $N$ at $F_1$ 
is bigger.  We use the number $\mult(N)\in \{1,2\}$ from Definition~\ref{def:graph} for the multiplicity of $N$ and the formula $\chi(N)=\tfrac{\dim(N)}{2\mult(N)}+1$. We let $N_1,\ldots, N_h$ denote the fixed-point components arising in this way, listed with multiplicity. Since different $N_i$ intersect orthogonally, their dimensions add up to $n$. Due to our choice of listing components of $\HP$ type twice, we have $n=\sum_{i=1}^h\tfrac{\dim(N_i)}{\mult(N_i)}$. Therefore
	$$\sum_{i=1}^h \bigl(\chi(N_i)-\chi(F_1)\bigr) = \sum_{i=1}^{h}\tfrac{\dim(N_i)}{2\mult(N_i)}=\tfrac{n}{2}.$$
 On the other hand, we can compute the sum differently. 
 The number $\chi(N_i)-\chi(F_1)$ is exactly the Euler characteristic of
 $N_i\cap \bigcup_{j=2}^h F_j$.  
 It is now easy to see that for each $F_j$ there are exactly $m$ numbers 
 $\{i_1,\ldots,i_m\}$ such that $F_j\subset N_{i_u}$. These correspond to the labels 
 of edges between $1$ and $j$. Some label might occur twice but then the corresponding component is of $\HP$ type and it was listed twice. 
 This shows 
$$\tfrac{n}{2}=\sum_{i=1}^{h}\bigl( \chi(N_i)-\chi(F_1)\bigr)=m\sum_{j=2}^{k_0}\chi(F_j),$$
and the result follows. 
 
 If there is no component consisting of a single point, the argument is similar.
 If $m\ge 3$ or $m=1$, one can use the fact that all the manifolds $N_i$
 that contain $F_1$ and another component must be 
 of $\CP$ type, since otherwise one could find a  fixed-point component of a codimension two torus which is not 
 of $\CP$, $\HP$, or $\Sph$ type for $i=1,\ldots,h$. 
 Moreover, for any non-zero label of an edge from $1$ to itself, the corresponding fixed-point component is of $\HP$ type,
 and the subrepresentation is perpendicular to the tangent space of $N_i$ for $i=1,\ldots,h$. Therefore, the dimensions of the components satisfy $\sum_{i=1}^h(\dim(N_i)-\dim(F_1))=  n-m\dim(F_1)$. With these modifications one easily arrives 
 at the same result. 
 If $m=2$, then there is a codimension one torus $\gT^{d-1}$ whose fixed-point component 
 is of $\HP$ type.  If one works with all codimension one tori inside $\gT^{d-1}$ instead,
 one can use that  all these components  at $F_1$ are of $\HP$ type.
 
{\em c).} Finally, we want to prove $m\le 2$ if
 $\dim(F_1)\ge 4$ and $\chi(\bigcup F_i)\ge 3$.  Assume on the contrary that $m\ge 3$. 
 If $F_1\cong \HP^k$ for some $k\ge 1$, we consider two linearly independent weights $r_1$ and $r_2$ 
 of edges between $1$ and $2$.  Each of these labels 
 occurs with multiplicity two, since the fixed-point components 
 of the codimension one tori are of $\HP$ type. 
 But then we get a contradiction by considering the fixed-point component of the codimension two torus determined by $r_1$ and $r_2$. In the remaining case, $F\cong \CP^k$ and the assumption of the dimension implies 
 $k\ge 2$. We then consider two non-zero weights $r_1$ and $r_2$ of edges from 1 to itself. 
 The fixed-point component of the corresponding codimension two torus cannot be of $\CP$, $\HP$, or $\Sph$ type.

\subsection{\subsecHP}\label{subsec:HP} We want to prove Lemma~\ref{lem:HP}.
By the Chang-Skjelbred Lemma, we know that the one-skeleton is connected. 
In view of Lemma~\ref{lem:straight}, the knowledge of the Euler characteristic then shows that the unlabeled graph assigned to the one-skeleton looks like the unlabeled graph of a torus action on a rational $\HP^{n/4}$. 
The first major step is to recover the labels in this case. 
This can be seen as reproving Hsiang and Su's result \cite{HsiangSu75} before we know that we are on a rational $\HP^{n/4}$. 
As before, we assume that $F_0,\ldots,F_{k_0}$ are the fixed-point components. 
We let $r_{ij}^\pm\in \Lt^*$ denote the (rational) weights of the two edges between $i$ and $j$ for $i\neq j$. 
We may assume that $r_{ij}^\pm$ are rational multiples of the actual weights and hence that $r_{ij}^\pm$ maps the lattice in $\Lt$ given by the kernel of the exponential map to a subset of $\Q$. Then $r_{ij}^\pm$ is determined up to a non-zero rational factor. We will use the notation $\ml r_{ij}^\pm\mr:=\spann_{\Q}\{r_{ij}^+,r_{ij}^-\}$. The labels $r_{ii}^\pm$ are only defined if $F_i\neq pt$.

\begin{sublem}\label{lem:linearmodel} We can find $\lambda_{ij}^\pm\in \Q$ 
and $u_0,\ldots,u_{k_0}\in \Lt^*$ such that 
$r_{ij}^\pm=\lambda_{ij}^\pm(u_i\pm u_j)$ holds
for all $i\le j$, after possibly switching the roles of the two weights $r_{ij}^\pm$.  
\end{sublem}

\begin{proof}
We remark that the statement is equivalent to showing that the set  $\{\Q r_{ij}^+,\Q r_{ij}^-\}$ is equal to $\{\Q (u_i-u_j),\Q (u_i+u_j)\}$. Notice also that the latter set does not change if we switch $i$ and $j$.\\[1ex]
%
%
%
{\bf Claim 1.}  If $\ml r_{ij}^\pm\mr\neq \ml r_{ik}^\pm \mr $ are both one-dimensional, 
then $\ml r_{jk}^\pm \mr$ is two-dimensional.\\[1ex]
%
Consider the codimension two torus $\gH$ whose Lie algebra is given by the intersections of the kernels of
 $r_{ij}^+$ and $r_{ik}^+$. By assumption, it has a fixed-point component $N$ of 
$\HP$ type containing $F_i$, $F_j$, and $F_k$. Choose a point  $p_u\in F_u$, 
and let $\alpha_u\in H^4(p_u\times B\gT^d) \cong H^4(B\gT^d)$ denote the restriction 
of  a generator $\alpha\in H^4_{\gT^d}(N)$ for $u=i,j,k$.

By Lemma~\ref{lem:alphaineqalphaj}, $\alpha_i$, $\alpha_j$, and $\alpha_k$ are pairwise different. Furthermore, 
$r_{ij}^+ $ divides $\alpha_i-\alpha_j$, and since it is the only linear factor in the factorization, we see $\alpha_i-\alpha_j= c_1 (r_{ij}^+)^2$. Similarly, $\alpha_i-\alpha_k=c_2 (r_{ik}^+)^2$. We know that $r_{jk}^\pm$ are the only linear factors of $\alpha_k -\alpha_j = c_1 (r_{ij}^+)^2 - c_2 (r_{ik}^+)^2$, so they are linearly independent.\\[1ex]
%
%
%
%
\hspace*{0em}From Claim 1, it follows that there is at most one $i$ such that $\ml r_{ij}^\pm \mr$ 
is one-dimensional for all $j$. After reordering, we may assume that $r^-_{01}$ and $r^+_{01}$ are linearly independent. \\[1ex]
{\bf Claim 2.}  There is a non-zero element $u_0\in \Lt^*$ with 
$u_0\in \ml r^-_{0i}, r^+_{0i}\mr$ for all $i$. \\[1ex]
{\em Case 1.} If there is an $i\ge 0$ such that $r_{0i}^ \pm$ are linearly dependent, we argue as follows. Here $i=0$ is only 
allowed if $F_0$ is not a point. We choose $u_0$ with $r_{0i}^\pm\in \Q u_0$ and consider the fixed-point set of the codimension one torus $\gT^{d-1}$ determined by $u_0$. It is then easy to see that, for each $j$, the isotropy representations corresponding to $r_{0j}^\pm$ must be rationally equivalent when restricted to $\gT^{d-1}$. In fact, otherwise there would be a fixed-point component of a codimension two torus whose reduced graph only has one connection from $0$ to $j$ but two connections from 
$0$ to $i$, contradicting Lemma~\ref{lem:graphofmodel}.
Therefore, it follows that $u_0$ is contained in $\ml r_{0j}^\pm \mr$ whenever the two vectors are linearly independent. Hence, it suffices to rule out that there a $k$ such that $\ml r_{0k}^\pm\mr$ is one-dimensional but not equal to $\Q u_0$. But otherwise,
 $\ml u_0,r_{0k}^+\mr$ would be equal to 
  $\ml r_{0j}^\pm\mr$ whenever the latter space is two-dimensional.  Therefore, all weights at $F_0$ would lie in 
  the two-dimensional subspace $\ml r_{01}^\pm \mr$, contradicting the assumption $d\ge 4$. \\[-1ex]
 
%
%
%
{\em Case 2.} If $r_{0i}^ \pm$ are linearly independent for all 
$i$, then
after reordering we can assume that $\ml r^\pm_{02} \mr\neq \ml r^\pm_{01}\mr$. Since we may
 switch the roles of the two elements $r_{12}^\pm$, we may assume $r_{12}^-$ is in $\ml r_{01}^-, r_{02}^-\mr$ and $\ml r_{01}^+, r_{02}^+\mr$. 
 In particular, $r^-_{01}$, $r^+_{01}$, $r^-_{02}$, and $r^+_{02}$ are linearly dependent, and we find a non-zero $u_0\in \ml r^\pm_{01}\mr\cap \ml r^\pm_{02}\mr$. 
Since $d\ge 4$, we can next assume that $\ml  r^\pm_{03}\mr\not\subset \ml r^-_{01}, r^+_{01}, r^-_{02}, r^+_{02}\mr$. 
 By the same reasoning, there are non-zero elements 
  $v_0\in  \ml r^\pm_{01}\mr\cap  \ml r^\pm_{03} \mr$ and 
  $w_0\in  \ml r^\pm_{02}\mr\cap  \ml r^\pm_{03}\mr$. 
  By our choice, the intersection $\ml r_{03}^\pm \mr\cap \ml  r^-_{01}, r^+_{01}, r^-_{02}, r^+_{02}\mr$ is one-dimensional as well and thus $u_0$, $v_0$, and $w_0$ are multiples of each other. The claim now follows since, for any $i$, we can put $\ml r^\pm_{0i}\mr$  in place of $\ml r^\pm_{0j}\mr$ for some $j\in\{1,2,3\}$.\\[-2ex]

 We have proved Claim 2. 
If  $r^\pm_{0i}$ are linearly independent, then $r^\pm_{0i}\not\in  \Q u_0$. 
In fact, otherwise, we could choose some $r_{0k}^+$ such that $r_{0i}^\pm$ 
and $r_{0k}^+$ are linearly independent and pass to the reduced graph determined by the weights $u_0$ and $r_{0k}^+$.  From the above claim, it follows that the two edges between $0$ and $k$ remain part of the reduced graph while there is only one edge between $0$ and $i$.

Thus after a rescaling of $r^+_{0i}$ and $r^-_{0i}$, we can assume $2u_0 = r^+_{0i} + r^-_{0i}$. We now choose $u_1,\ldots, u_{k_0}\in \Lt^*$ such that $r^\pm_{0i}=u_0\pm u_i$. Once we fix $u_0$, the element $u_i$ is uniquely determined for all 
 $i$ for which $\ml r_{0i}^+\mr$ is two-dimensional. If  $r_{0i}^\pm\in \Q u_0$, then we can still adjust the scaling of $u_i$. The sublemma follows from the following\\[1ex]
{\bf Claim 3.}  After rescaling and possibly switching the roles of $r^+_{ij}$ and $r^-_{ij}$, 
we must have $r^+_{ij}=u_i+u_j$ and $r^-_{ij}=u_i-u_j$ for all $i<j$. \\[1ex]
For the rescaling, we have included the possibility of adjusting the scaling of those elements $u_i$ with $u_i\in \Q u_0$ and $i>0$. The combination of Claim 2 and Claim 3 implies that $r_{jj}^\pm =u_j\pm u_j$ holds after rescaling as well when $F_j$ is not a point.\\[-1ex]
 
{\em  Case 1.} If $u_i$, $u_j$, and $u_0$ are linearly independent, then
  $\{\Q r_{ij}^\pm\}\cap \{\Q r_{0j}^\pm\}=\varnothing$ 
 because otherwise we would get a contradiction to  
 $\{\Q r_{0j}^\pm\}\cap \{\Q r_{0i}^\pm\}=\varnothing$ 
 by looking at the fixed-point set of a codimension one torus.
 After possibly switching the roles of $r^\pm_{ij}$ and adjusting the scaling, we can find $\lambda_\pm \neq 0$ with
	\[r_{ij}^\pm = u_0 + u_i + \lambda_\pm(u_0 \mp u_j).\]
On the other hand, we can also find 
$\mu_\pm$ and $\nu_\pm $ with
	\[r_{ij}^\pm = \mu_\pm (u_0 - u_i) + \nu_\pm(u_0 \pm u_j).\]
 This gives
 \begin{eqnarray*}
 0&=&(1+\lambda_+-\mu_+   -\nu_+)u_0+(1+\mu_+)u_i-(\lambda_+ + \nu_+)u_j\\
 0&=&(1+\lambda_- -\mu_-    -\nu_-)u_0+(1+\mu_-)u_i+ (\lambda_-  + \nu_-)u_j
\end{eqnarray*}
 Since $u_0$, $u_i$, and $u_j$  are linearly independent, we see $\mu_\pm=-1=\lambda_\pm=-\nu_\pm$.\\[-1ex]

{\em Case 2.}  If $u_0$, $u_i$, and $u_j$ are pairwise linearly independent but not linearly independent, then we choose 
some $k$ such that $u_i$, $u_j$, and $u_k$ are linearly independent.
By the assumption in this case, we have that $u_0$, $u_i$, and $u_k$ are linearly independent and likewise for $u_0$, $u_j$, and $u_k$. By the previous case we know after rescaling that $r_{ik}^\pm =u_i\pm u_k$ and $r_{jk}^\pm=u_j\pm u_k$. 
 Arguing now in the triangle $(k,i,j)$ with $0$ replaced by $k$, we again can apply the previous calculation to get the claimed result. \\[-1ex]
%
 %
 %

 {\em Case 3.} Suppose that $u_i$ is a multiple of $u_0$.
  Recall that we are free to adjust the scaling of $u_i$ once in this case. We use this flexibility to get the claimed equation for one $j$. If, for all $j$, we get that the equation holds with $u_i=0$, then we are done. 
 Otherwise, we choose $j$ with $u_j\notin \Q u_0$ such that the equation $r_{ij}^\pm=u_i\pm u_j$ holds with $u_i\neq 0$.
 We then get the equation for all $k$ with $u_k\notin  \ml u_i, u_j \mr$ by looking at the triangle $(i,j,k)$. Of course this requires no further modification of $u_i$. Then in turn we can switch the roles of $j$ and $k$ and get the equation for all 
 $l$ with  $u_l\notin  \ml u_i, u_k\mr $. Thus we have $r_{ij}^\pm=u_i\pm u_j$ for all $j$ with $u_j\not\in \Q \cdot u_0$. In the remaining case where $u_j,u_i\in \Q u_0$, we just need 
 to rule out $u_i=\pm u_j$, which will be done below.\\[-1ex] 

{\em  Case 4.} If $u_i$ is a multiple of $u_j$ and $u_i\not \in\Q \cdot u_0$, then we show that $r_{ij}^+$ and $r_{ij}^-$ are multiples of $u_j$. The claim follows unless $u_i=\pm u_j$, a possibility which is ruled out later.
 We choose  $k$ such that  $u_0$, $u_j$, and $u_k$ are linearly independent.
 By Case 1, $r_{ik}^\pm =u_i \pm u_k$ and $r_{jk}^\pm =u_j\pm u_k$ holds after rescaling. 
 Now the claim follows from 
 $$r_{ij}^\pm\in \ml r_{jk}^\pm,r_{ik}^\pm\mr\cap  \ml r_{0i}^\pm, r_{0j}^\pm \mr=\ml u_j,u_k\mr\cap \ml u_0,u_j\mr=\Q u_j.$$

  To finish the proof of Claim 3, we must also show $u_i\neq \pm u_j$ for $i\neq j$.  
 This does not follow from the above proof since, in the case $u_i \in \Q u_j$, we have only shown that the weights  $r_{ij}^\pm$ are multiples of $u_j$. First, note that $u_i=u_j=0$ was ruled out by Claim 1. 
 Assume now on the contrary that $u_i = \pm u_j \neq 0$ for some $i\neq j$. 
 We choose $u_k$ linearly independent to $u_i$ and consider the codimension two torus $\gC$ determined by $u_i$ and $u_k$. 
 It follows that $F_i$, $F_j$, and $F_k$ are contained in the same fixed-point component $N=N_{ik}$ of $\gC$ and that $N$ is of $\HP$ type. 
 We choose $\alpha\in H^{4}_{\gT^d}(N)$ projecting to a generator of $H^4(N)$. 
 Let $p_u\in F_u$, and let $\alpha_u\in H^4(BT^d)$ be the image of $\alpha$ in $H^*_{\gT^d}(p_u)\cong H^*(B\gT^d)$ for $u=i,j,k$. By Lemma~\ref{lem:alphaineqalphaj}, $\alpha_i$, $\alpha_j$, and $\alpha_k$ are pairwise 
 distinct. Furthermore, it follows that $\alpha_i-\alpha_k$ and $\alpha_j-\alpha_k$ 
 must be divisible by $(u_i^2-u_k^2)$. Similarly, $\alpha_i-\alpha_j$  must be divisible by $u_i$. This is clearly not possible. 
 \end{proof}

 \begin{sublem} We choose points $p_i\in F_i$ for $i=0,\ldots,k_0$ and view $u_0,\ldots,u_{k_0}\in \Lt^*$ from Sublemma~\ref{lem:linearmodel} as elements in $H^2(B\gT^d)$.
 There exists $\alpha\in H^4_{\gT^d}(M)$ whose image in $H^4_{\gT^d}(\{p_0,\ldots,p_{k_0}\})\cong H^4(B\gT^d)^{k_0+1}$ 
 is given by $(u_0^2,\ldots,u_{k_0}^2)$ and has the property that $\alpha^0,\ldots, \alpha^{n/4}$ is a basis
 the $H^*(B\gT^d)$-module $H^*_{\gT^d}(M)$.
 \end{sublem}

 Lemma~\ref{lem:HP} follows from the sublemma
 since $\alpha^i$ cannot be in the kernel of the natural
  map $H^*_{\gT^d}(M)\rightarrow H^*(M)$ for 
  $i=0,\ldots,n/4$.

 \begin{proof} \hspace*{1em}\\
 {\bf Step 1.}
 Let $N$ be a fixed-point component of a subtorus 
  codimension at most two. Suppose 
  $N\cap \{p_0,\ldots,p_{k_0}\}=\{p_{i_0},\ldots,p_{i_b}\}$ 
  with $0\le i_0<\cdots<i_b$ and $b\ge 1$. 
  \begin{enumerate}
  \item[(i)] There exists $\beta\in H^{4}_{\gT^d}(N)$ whose pullback to 
  $H_{\gT^d}^4(p_{i_j}) $
  is given by $u_{i_j}^2$ for all $j$.
 \item[(ii)] If $N$ is of $\HP$ type or $b\ge 2$, then  $\beta$ is unique.  
 \item[(iii)] If $N$ is of $\CP$ type, then one can arrange 
 for $\beta=\gamma^2 $ with $\gamma\in H^2_{\gT^d}(N)$ 
  and $\gamma$ is unique up to sign. 
 \item[(iv)] If $F_{i_j}\neq\{p_{i_j}\}$ with $j\le b$, there is a unique $\beta$ whose restriction to $F_{i_j}\times B\gT^d$ equals $u_{i_j}^2+u_{i_j}\deltad+\epse$ with respect to the K\"unneth formula, where $\deltad\in H^2(F_{i_j})$ and $\epse\in H^4(F_{i_j})$. If $N$ is of $\CP$ type, then this is the solution from (iii). 
 \end{enumerate}
After a permutation we may assume $i_j=j$ 
for $j=0,\ldots,b$. 
 
{\em Case 1.}  If $N$ is of $\HP$ type, then
the uniqueness is obvious since, by
Lemma~\ref{lem:alphaineqalphaj} a), the pullback map  $H^4_{\gT^d}(N)\to H^4_{\gT^d}(\{p_{0},p_{1}\})$ is injective.
For the existence, we may assume that the action of $\gT^d$ on $N$
has a kernel of codimension two.

After an additional permutation, 
$u_{0}$ and $u_{1}$ are linearly independent, and
 it is easy to see that there is a unique element $\beta\in H^4_{\gT^d}(N)$ 
whose restriction $\beta_i$  to $p_{i}
\times B\gT^d$ is given by $u_{i}^2$ for $i=0,1$. 
We have to check that the equation remains valid automatically for $i=2,\ldots,b$.
We prove this for $i=2$. Since $\beta_2-\beta_0$ is divisible by $u_{0}\pm u_{2}$ 
and has no other linear factors, we see that 
$\beta_2=u_{0}^2+ c(u_{0}^2-u_{2}^2)$.
Similarly, we deduce $\beta_2=u_{1}^2+d(u_{1}^2-u_{2}^2)$. 
If $u_0$, $u_1$, and $u_2$ are pairwise linearly independent, then the quadratic polynomials $u_0^2$, $u_1^2$, and $u_2^2$ are linearly independent. By comparison it follows that $\beta_2=u_2^2$. In the remaining case, we may assume that $u_2$ is a multiple of $u_0$. But then we still get $d=-1$ and thus the result.

It remains to check that this solution also satisfies (iv).
If $F_j$ is of $\HP$ type,  then 
$H^2(F_{j})=0$ and (iv) is satisfied. If $F_j\cong  \CP^{n_j}$ with 
$n_j>0$, 
we let $u_j^2+t\deltad+\epse$ denote the restriction 
of $\beta$ to $H^4_{\gT^d}(F_j)$. 
We need to prove that $t$ is a multiple of $u_j$. 
 Notice that $t\deltad +\epse$ only changes by a factor if we 
 replace $\beta$ by  another lift of a generator of $H^4(N)$. 
Let $L\subset N$ denote the fixed-point component of $\HP$ type
of the codimension one torus $\gR$  determined by
 the  non-zero weight $r_{jj}^+=2u_j$ of the edge from $j$ to itself.
 As before, $H^4_{^{\gT^d/\gR}}(L)$ injects into $H^4_{\gT^d}(L)$, 
and we can find $\varphi\in H^4(B\gT^d)$ such that 
 the restriction of $\beta+ \varphi$  to $H^4_{\gT^d}(L)$ 
 lies in the image of this map. This in turn shows 
 that the pullback of $\beta+\varphi$ to $H^4_{\gT^d}(F)$ 
 lies in $H^4_{^{\gT^d/\gR}}(F_j)$ and thus $t\in \Q u_j$ as claimed.
\\[-1.5ex]

{\em Case 2.} If $N$ is of $\CP$ type,  we first look at the 
reduced graph of $N$. For any $0\le i<j\le b$ there 
is precisely one edge between $i$ and $j$ in the reduced graph, 
and it has the label $u_i-u_j$ or $u_i+u_j$.
Since it cannot contain both, the elements $u_0,\ldots,u_b$ 
must be pairwise linearly independent. The elements 
$u_0,\ldots,u_{k_0}$ from Sublemma~\ref{lem:linearmodel}
are only determined up to sign, so we may assume that the 
$(u_0-u_i)$-edge is part of the reduced graph for all $i=1,\ldots,b$. 
But then the reduced graph must also contain the $(u_i-u_j)$-edges 
for $i<j\le b$.

We next prove the various uniqueness statements. 
Suppose we have found a solution $\beta=\gamma^2$ as 
in (iii). We let $\gamma_i$ denote the pullback to 
$H^2_{\gT^d}(p_i)$. Then $\gamma_j=\pm u_j$ and 
$\gamma_1-\gamma_0$ must be a multiple of $u_1-u_0$. 
Thus we see that $(\gamma_0,\gamma_1)=\pm (u_0,u_1)$. 
 Since the kernel $H^{2}_{\gT^d}(N)\to H^2_{\gT^d}(\{p_0,p_1\})$
is trivial, this shows that $\gamma$ is unique up to sign. 
If we are in the situation of (iv), we can permute indices and assume $j=0$. The restriction of $\beta$ to $F_0\times B\gT^2$ is then given by $(\gamma_0+t)^2$ with $\gamma_0\in H^2(F_0)$ and $t\in H^2(B\gT^2)$. Moreover, we must have 
$t^2=u_0^2$, which implies $t=\pm u_0$, and hence the pullback of $\beta$ has indeed a decomposition as in (iv).
To see that this composition in turn characterizes $\beta$
itself we observe that  the kernel of
$H^{4}_{\gT^d}(N)\to H^4_{\gT^d}(\{p_0,p_1\})$ 
is the one-dimensional space generated by 
$\eta:=(\gamma-u_0)(\gamma -u_1)$. Since
$u_0$ and $u_1$ are linearly independent,  it is easy to see 
that the K\"unneth decomposition of the pullback of $\beta+c\eta$ 
has no longer the form of (iv) for any non-zero $c\in \Q$. 
Finally, if $b\ge 2$, then $\beta$ is unique since
 it is straightforward to show that  the kernel of
$H^{4}_{\gT^d}(N)\to H^4_{\gT^d}(\{p_0,p_1,p_2\})$ 
is trivial.
 
 To establish existence, we may assume that the induced action of $\gT^d$ on $N$ has a kernel of codimension two. Indeed, if the kernel has codimension one, we can find a
larger fixed-point component that is still of $\CP$ type, and
the solution in the smaller component is then obtained from 
the solution in the larger one by restriction.
Let $\gamma \in H^2_{\gT^2}(N)$ 
be a lift of a generator of $H^2(N)$. 
We continue to denote by $\gamma_i $ the restriction of $\gamma$ 
to $p_{i}\times B\gT^d$.  From the knowledge of the reduced 
graph, we have
$\gamma_i-\gamma_j=q_{ij}(u_i-u_j)$ for $0\le i<j\le b$ 
 with $q_{ij}\in \Q\setminus \{0\}$. 
After rescaling and adding an element of $H^2(B\gT^d)$ to $\gamma$, 
we can assume that $\gamma_0=u_0$ and $\gamma_1=u_1$ holds. 
We claim that $\gamma_i=u_i$ holds automatically for $i=2,\ldots,b$. 
 In fact,
 \begin{eqnarray*}
 \gamma_i=u_0+q_{i0}(u_i-u_0)&=&u_1+q_{i1}(u_i-u_1) \quad \mbox{ gives}\\
 (q_{i0}-q_{21}) r_{0i}^-&=&(1-q_{i1})r_{01}^-.
 \end{eqnarray*} 
 If $r_{0i}^-$ and $r_{01}^-$ are linearly independent, then 
$q_{i1}=1=q_{i0}$ follows. We may assume that  $r_{01}^-$ 
and $r_{02}^-$ are indeed linearly independent. 
To prove the statement for $j\le b$ with $r_{0j}^-\in  \Q r_{01}^-$, 
one then just repeats the above argument with $r_{01}^-$ replaced by $r_{02}^-$.\\[-1ex]
 
 We have established Step 1.
 Let $N_{ij}^\pm$ be the fixed-point component at $F_i$ of 
the torus whose Lie algebra is given by the kernel of $r_{ij}^\pm$.  
The cohomological image of the one-skeleton is clearly contained 
in the image of the map $h_{ij}^\pm \colon H^*_{\gT^d}(F \cup N_{ij}^\pm)\rightarrow H^*_{\gT^d}(F)$.
Using the Mayer-Vietoris sequence, one checks that the image of the cohomology of the 
 one-skeleton is given by the intersection of the images of all these  maps. By the Chang-Skjelbred Lemma, the first statement of the sublemma therefore follows from\\[1ex]
{\bf Step 2.}  
Let $j_v\neq i$ and $\sigma_v\in \{\pm\}$, 
and choose $\beta_v \in H^4_{\gT^d}(N_{ij_v}^{\sigma_v})$ 
 as in Step 1 (i). In addition, we demand that $\beta_v$ is a square if $N_{ij_v}^{\sigma_v}$ is of $\CP$ type for $v=1,2$. Then $\beta_1$ and $\beta_2$ 
 pull back to the same element 
  $\zeta_i\in H^4_{\gT^d}(F_i)$. Moreover, if $F_i\cong \CP^{n_i}$, with $n_i>0$, then $\zeta_i$ is 
 also in the image of $H^4_{\gT^d}(N_{ii}^+)$. \\[1ex]
If $F_i$ is a point, there is nothing to 
prove. 
We consider a fixed-point component $N$ of a codimension 
two torus that contains the two submanifolds $N_v:=N_{ij_v}^{\sigma_v}$ 
$(v=1,2)$. 
By Step 1, there is a unique element $\beta\in H^4_{\gT^d}(N)$
that satisfies the conclusion of (iv). 
In particular, the restriction of $\beta$ to 
$F_i\times B\gT^d$ is of the form
$u_i^2 +\deltad_i u_i+ \epse_i\in H^4(F_i\times B\gT^d)$ 
with $\deltad_i\in H^2(F_i)$ and $\epse_i\in H^4(F_i)$. 
The pullback $\hat\beta_v$ of $\beta$ to $H^4(N_v)$ is  the unique solution satisfying the conclusion
of Step 1 (iv)  for the submanifold $N_v$ for $v=1,2$. 
But this implies $\hat \beta_v=\beta_v$. 
Clearly the claim follows. 
If we choose $j_1=j_2\neq i$, $\sigma_1=-$, and $\sigma_2=+$, 
then $N$ contains  $N_{ii}^+$, and therefore 
$\zeta_i$ is also in the image of $H^4_{\gT^d}(N_{ii}^+)$.\\[-1ex]

Since the map $H^*_{\gT^d}(M)\to H^*_{\gT^d}(F)$ 
is injective, the sublemma follows if we show\\[1ex]
{\bf Step 3.} By Step 2 and the remark before that, 
we can choose $\alpha\in H^4_{\gT^d}(M^n)$ whose cohomological image 
in $\{p_0,\ldots,p_{k_0}\}\times BT^d$ is given by $(u_0^2,\ldots,u_{k_0}^2)$.
We let $\hat\alpha $ denote the image of $\alpha$ in $H^4(F\times BT^d)$. 
Then the image of $H^*_{\gT^d}(M)$ in $H^*(F\times B\gT^d)$ is the free $H^*(B\gT^d)$-module with the basis $\hat\alpha^0,\ldots,\hat{\alpha}^{n/4}$  \\[1ex]
To prove Step 3, first notice that it suffices to prove that
 $\hat\alpha^0,\ldots,\hat{\alpha}^{n/4}$  are linearly independent 
 over $H^*(B\gT^d)$. In fact, once this is established,  the image has the same dimension as the module generated 
 by $\hat\alpha^0,\ldots,\hat{\alpha}^{n/4}$. If the image would be strictly bigger, then 
 it would be generated  as $H^*(B\gT^d)$-module by elements of degree $<n$,
 but this is impossible. 
 To prove linear independence we first treat a special case.\\[-1ex]
 
{\em Case 1.} The fixed-point set is discrete. Then $k_0=n/4$, and
 the result follows from the knowledge of the determinant of the following 
 Vandermonde matrix 
 \begin{eqnarray*}
 \det\left(\begin{array}{ccc}
 1 &\cdots&1\\
 u_0^2&\cdots&u_{n/4}^2\\
 \vdots&\vdots&\vdots\\
 u_0^{\frac{n}{2}}&\cdots & u_{n/4}^{\frac{n}{2}}  \end{array}\right)=\prod_{0\le i<j\le n/4}(u_i^2-u_j^2)\neq 0
 \end{eqnarray*}
 
{\em Case 2.} We now consider the general case.  We start the other way around and first 
 check that the relevant matrix has  a non-zero determinant, before we explain why exactly this matrix is the correct one.
 
 Let $\beta_0,\ldots,\beta_{n/4}\in H^*(B\gT^d)$, and consider again the Vandermonde determinant
 \begin{eqnarray}\label{Vdet} 
\det(A):=\det\left(\begin{array}{ccc}
 1 &\cdots&1\\
 \beta_0&\cdots&\beta_{n/4}\\
 \vdots&\vdots&\vdots\\
 \beta_0^{\frac{n}{4}}&\cdots & \beta_{n/4}^{\frac{n}{4}}  \end{array}\right) = \prod_{0 \leq p < q \leq \frac n 4}(\beta_p - \beta_q).
\end{eqnarray}
  We eventually evaluate at $\beta_0=\ldots=\beta_{n_0}=u_0^2$, $\cdots$,
 $\beta_{n/4-n_{k_0}}=\cdots=\beta_{n/4}=u_{k_0}^2$. In other words, the value $u_i^2$ 
 is attained exactly $n_i+1$ times, the total Betti number of $F_i$, where $i=0,\ldots,k_0$. 
 We will have $\sum_{i=0}^{k_0} (n_i+1)=n/4+1$. 
But before we evaluate at these values, we differentiate both sides 
of \eqref{Vdet} $h_q$ times with respect to the variable $\beta_q$, 
where $h_q$ is the the order of the set $\{p\mid p<q,\, \beta_p=\beta_q\}$.
The result equals
\[
 \frac{d}{d\beta_1}\!\cdots\!\frac{d^{n_0}}{d\beta_{n_0}^{n_0}}\cdots
 \frac{d}{d\beta_{n/4-n_{k_0}+1}}\!\cdots\!
 \frac{d^{n_{k_0} } }{d\beta_{n/4 }^{n_{k_0}}}\det(A)=C\cdot\prod_{0\le i<j\le k_0}(u_i^2-u_j^2)^{(n_i+1)(n_j+1)}
 \]
where $C=\prod_{i=0}^{k_0} \prod_{q=1}^{n_i}((-1)^q q!)$. Hence, the result is non-zero. 

On the other hand, using multilinearity of the determinant, we can differentiate each of the columns separately before we take the determinant. This way, we end up essentially with the relevant matrix. To get exactly to the relevant matrix, one has  to multiply a column that was differentiated $k$ times with $\tfrac{1}{k!}$.
 
 We let $\alpha_i\in H^4(F_i\times B\gT^d)$ be the image of $\alpha$. 
 We put $\gamma_i=\alpha_i-u_i^2=  \deltad_i l_i+ \epse_i$ 
 where $\deltad_i\in H^2(F_i)$, $l_i\in H^2(B\gT^d)$, and $\epse_i\in H^4(F_i)$.
 Since the image of $\alpha$ always factors through the image of some fixed-point component 
 of $\HP$ type, it follows that  $\deltad_i\cdot l_i\neq 0$ if
  $F_i\cong \CP^{n_i}$ and that $\epse_i\neq 0$
  if $F_i\cong \HP^{n_i}$ with $n_i>0$.
 In either case, $\gamma_i^0,\ldots,\gamma_i^{n_i}$ are linearly independent over 
 $H^*(B\gT^d)$. 
 Clearly $\gamma_0^0,\ldots,\gamma_0^{n_0},\ldots, \gamma_{k_0}^0,\ldots,\gamma_{k_0}^{n_{k_0}}$ generate a free $H^*(B\gT^d)$-module of dimension $n/4+1$.
 Moreover, $\gamma_i^{n_i+1}=0$.
 It is now easy to determine the matrix representing   the linear map that 
  sends this basis to $\hat\alpha^0,\ldots,\hat{\alpha}^{n/4} $. 
 It is obtained from the following non-square matrices 
 \begin{eqnarray*}
\left(\begin{array}{cccc}
 1 &0&\cdots&0\\
 u_i^2&1&\cdots&0\\
 \vdots&\vdots&\vdots&\\
 u_i^{\frac{n}{2}}&
{\footnotesize  \Bigl({\tiny \begin{matrix}
 n/4\\1
\end{matrix}}\Bigr) }u_{i}^{\frac{n-2}{2}}&
  \cdots & {\footnotesize \Bigl({\tiny \begin{matrix}
 n/4\\ n_i
\end{matrix}}\Bigr) } u_{i}^{\frac{n-2n_i}{2}}  \end{array}\right)\mbox{ for $i=0,\ldots,k_0$}
 \end{eqnarray*}
by putting their columns together into a square matrix. However, this is exactly the matrix described initially, and thus the claim follows.
\end{proof}

\subsection{\subsecCP}\label{subsec:CP} We want to prove Lemma~\ref{lem:CP}. 
Most of the proof is the same as in the previous subsection. We only describe the differences.  
As before, let $F_0,\ldots,F_{k_0}$ denote the fixed-point components. 
The graph of the one-skeleton is just the one-skeleton of a $k_0$-dimensional simplex with the twist that there is an additional edge from $i$ to itself with label $0$ whenever $F_i$ is not a point. The isotropy representation corresponding to each edge is, up to rational equivalence, determined by one homomorphism $r_{ij}\in \Lt^*$.\\[1ex]
{\bf Claim.} There are $u_0,\ldots, u_{k_0}\in \Lt^*$ such that 
$r_{ij}=u_i-u_j$ holds after rescaling. \\[1ex]
We choose $u_0=0$ and put initially $u_i=-r_{0i}$. 
Now we can rescale $u_i$ and $r_{1i}$ for all $i\ge 2$ such that we have 
$r_{1i}=u_1-u_i$.  This determines $u_i$ unless $u_i$ and $u_1$ are linearly dependent.
We claim that the equation then 
follows for all $ij$ after scaling $r_{ij}$ and possibly adjusting the scaling of those $u_i$ contained in $ \Q u_1$. 
To see this for $r_{23}$, we look at the simplex $(0,1,2,3)$.
We know that $r_{23}$ is in the span of $u_2$ and $u_3$, 
as well as in the span of $u_1-u_2$ and $u_1-u_3$. 
If $u_1,u_2$, and $u_3$ are linearly independent, we deduce $r_{23}=\lambda(u_2-u_3)$, as claimed.

If $u_1$, $u_2$, and $u_3$ are pairwise independent but not independent,
 the result follows similarly to the previous section by choosing $u_k\not\in \ml u_1,u_2\mr$ and then arguing in the simplex $(0,2,3,k)$ where we already know the equation for all but one of the edges. 

If $u_1$ and $u_2$ are linearly dependent, one first adjusts the scaling of $u_2$ to 
get the equation $r_{2k}=u_2-u_k$ for one $k$ with $u_k\not\in \Q u_0$.
To get it for all other $l$, one argues in the $(0,2,k,l)$ simplex.

This proves the above claim, and it is then not hard to see that 
the element $(u_0,\ldots,u_{k_0})\in H^2(\{p_0,\ldots,p_{k_0}\}\times B\gT^d)$ 
is in the cohomological image of the one-skeleton.
The rest of the proof is then the same as before.

\subsection{\sectfiveCb}\label{sec:t8sphere}
In even dimensions, Theorem \ref{thm:t8NoCurvature} under the assumptions of b) follows from
\begin{theorem}\label{thm:Bb}
Let $M^n$ be a connected, oriented, closed manifold with a smooth, effective  $\gT^d$-action with $d \geq 4$. Suppose that the odd Betti numbers of $M^n$ vanish.
 If every fixed-point component of every subtorus with codimension at most 
 two is a rational $\CP^h$, $\HP^h$, or $\s^{2h}$ for some $h \geq 0$, then $\chi(M^n)=\tfrac{n}{2m}+1$ for some $m\ge 1$. 

Moreover, if $m\in \{1,2,\tfrac n 2\}$ or $(m,n)=(4,16)$, then
$H^*(M;\Q) \cong H^*(\bar M; \Q)$ for some $\bar M \in \{\s^n, \CP^{\frac n 2}, \HP^{\frac n 4}, \CaP\}$.
\end{theorem}

By the Chang-Skjelbred Lemma, the one-skeleton of $M^n$ is connected. 
Since the Euler characteristic of $M^n$ coincides with Euler characteristic of its 
fixed-point set, the formula for the Euler characteristic follows from 
Lemma~\ref{lem:straight}. Now if $m\in \{1,2\}$, we can use Lemma~\ref{lem:CP} and 
Lemma~\ref{lem:HP} to recover the rational cohomology. 
If $m= \tfrac n 2$, the total Betti number of $M^n$ is $2$, which implies that it is a rational homology sphere. 
Finally, in the case of $(m,n)=(4,16)$, the total Betti number is three, 
and Poincar\'e  duality implies that the rational cohomology is that of $\CaP$.

\subsection{\sectfiveodd}\label{subsec:odd} 

\begin{lemma}\label{lem:gysin} Suppose 
$\gT^\ell$ acts on an odd-dimensional, orientable,
closed manifold  $M^n$ 
with $H^i(M)=0$ for all odd $i\le \tfrac{n-1}{2}$.
Then $H^*_{\gT^\ell}(M)\to H^*(M)$ is surjective in even degrees. Moreover, 
$H^{odd}_{\gT^\ell}(M)$ does not contain any $H^*(B\gT^\ell)$-torsion, that is, for  non-zero elements
 $x\in H^{odd}_{\gT^\ell}(M)$ and  $t\in H^*(B\gT^\ell)$
 we have $xt\neq 0$.
In particular, if  $\gS^1$ acts almost freely, then 
$H^*_{\gS^1}(M)$ has vanishing odd Betti numbers.
\end{lemma}


\begin{proof}[Proof of Lemma~\ref{lem:gysin}]
 Up to degree $n/2$, the Betti numbers of $H^*_{\gT^{\ell}}(M)$ are equal to the Betti numbers of the product $B\gT^\ell\times M$. Thus the map 
 $H^*_{\gT^\ell}(M)\to H^*(M)$ 
 is surjective in even degrees. 
 If $\ell=1$, we look at the Gysin sequence
$$
\cdots\to H^k_{\gS^1}(M)\stackrel{\cup e}{\rightarrow} H^{k+2}_{\gS^1}(M)\stackrel{j^*}{\rightarrow}H^{k+2}(M)\stackrel{Gy}{\longrightarrow} 
H^{k+1}_{\gS^1}(M)\to \cdots
$$
Since $j^*$ is surjective in even degrees, the Gysin homomorphism is zero in even degrees. Hence, cupping with the Euler class is injective in all odd degrees. 
In other words, all $H^*(BS^1)$-torsion is in even degrees. 
This proves the lemma if $\ell=1$. 
For larger $\ell$, we rule out torsion elements of odd degree by contradiction and induction on $\ell$. 
If  $x\in H^{k}_{\gT^\ell}(M)$ is a non-zero torsion element with minimal odd $k$, and if $t\in H^*(B\gT^\ell)\setminus\{0\}$ is such that $xt=0$, then we can find $\gT^{\ell-1}$ 
such that the pullback of $t$ to $H^*(B\gT^{\ell-1})$ is non-zero. We can use the Gysin sequence 
$$
H^{k-2}_{\gT^\ell}(M)\to H^k_{\gT^\ell}(M) \stackrel{p^*} \to  
H^k_{\gT^{\ell-1}}(M)
$$
and the minimality of our choice of $k$ to see that $p^*(x)\neq 0$ -- a contradiction to our induction assumption. 
\end{proof}

In view of Lemma~\ref{lem:gysin}, Theorem~\ref{thm:t8NoCurvature} follows in odd dimensions from the following:

\begin{theorem}\label{thm:oddt7}
Assume that a torus $\gT^d$ acts effectively on an odd-dimensional, simply connected, closed manifold $M^n$ such that any fixed-point component of any codimension three torus in $\gT^d$ is a rational sphere. If $H^*_{\gT^d}(M)\to H^*(M)$ is surjective in even degrees, then $M^n$ is a rational sphere.
\end{theorem}

\begin{proof} 
 We choose $\gS^1\subset \gT^2\subset \gT^3 \subset \gT^d$ 
 such that $\gS^1$ has the same fixed-point set as $\gT^d$ 
 and $\gT^{i+1}$ has the same $i$-skeleton as $\gT^d$ for $i=1,2$. 
Our first claim is already valid if fixed-point components of any codimension one (rather than three) torus are spheres.\\[1ex]
{\bf Step 1.} If $\gT^d$ has $q$ fixed-point components, 
then $H^{odd}_{\gT^2}(M)$ is a free $H^*(B\gT^2)$-module 
with a basis $b_1,\ldots,b_q\in H^{n}_{\gT^2}(M)$.\\[1ex]
We argue by contradiction.
We let $F$ denote the fixed-point set of $\gT^d$. 
We claim that $H^{j}_{\gT^2}(M\setminus F)\neq 0$ holds for infinitely many odd $j$. 
 
Consider the long exact sequence in equivariant cohomology of the pair 
 $(M,M\setminus F)$:
$$
H^k_{\gT^2}(M,M\setminus F)\to H^k_{\gT^2}(M)\to H^k_{\gT^2}(M\setminus F).
$$ 
Since $\gS^1$ acts almost freely on $M\setminus F$,  
the quotient $P=(M\setminus F)/\gS^1$ is an $(n-1)$-dimensional orbifold.
 As the natural map 
 $(M\setminus F)\times _{\gS^1}E\gS^1\to P$ 
 is a rational homotopy equivalence, we have 
 $H^*_{\gT^2}(M\setminus F)\cong H^*_{\gT^2/\gS^1}(P)$. 
 This implies that $H^{*}_{\gT^2}(M\setminus F)$ is generated as a $H^*(B\gT^2)$-module in degrees $<n$.
Using that the components of $F$ are rational spheres, it follows by excision and the Thom 
isomorphism that  
$H^{2k+1}_{\gT^2}(M,M\setminus F)=0$ for $2k+1<n$ 
and that $H^{n}_{\gT^2}(M,M\setminus F)\cong   \Q^q$. 
Notice that $q$ then corresponds to the total odd Betti number of $F$.
By Lemma~\ref{lem:injective}, the map 
$H^{odd}_{\gT^2}(M,M\setminus F)\to H^{odd}_{\gT^2}(M)$ 
is injective. Its image is a free module generated by $q$ elements of degree $n$. 
Since we argue by contradiction, the map is not surjective, and 
by the proof of Lemma~\ref{lem:gysin}, $H^{odd}_{\gT^2}(M)$ is torsion-free. 
Thuswe can find elements $b_1,\ldots,b_q\in H^{odd}_{\gT^2}(M)$
that are linearly independent over $H^*(B\gT^2)$
and satisfy $\deg(b_i)\le n$ and $\deg(b_1)\le n-2$.
Therefore, the image of $H^{odd}_{\gT^2 }(M,M\setminus F)$ has infinite codimension as a  vector subspace of $H^{odd}_{\gT^2}(M)$, and the claim follows.

 We deduce that 
 $H^*_{\gT^2/\gS^1}(P)$ is non-zero in arbitrarily large odd degrees. This implies that the circle action of $\gT^2/\gS^1$ on $P$ has a fixed-point component $L\subset P$ with a non-trivial odd Betti number.
Our choice of $\gT^2$ implies that $L$ is also 
a fixed-point component of the induced $(\gT^d/\gS^1)$-action on $P$. 
We choose a fixed-point component $\hat L\subset M$ of a $(d-1)$-dimensional subtorus in $\gT^d$ such that $L$ is the projection 
of $\hat L\setminus F$. By assumption, $\hat L$ is a sphere 
and $F\cap \hat L$ is either empty or a sphere, too. This implies 
that $\hat L\setminus F$ has the cohomology of a rational sphere as well. Thus $L=(\hat L\setminus F)/\gS^1$ has the rational cohomology of a complex projective space -- a contradiction, as $L$ has non-vanishing odd Betti numbers. \\[1ex]
{\bf Step 2.} Let $S\subset M$ denote the one-skeleton of
the action. Then all odd Betti numbers of 
$H^*_{\gT^2}(M\setminus S)$ vanish. \\
Notice that the action on $\gT^2$ on $M\setminus S$ is almost 
free, and the natural projection 
$(M\setminus S)\times_{\gT^2} E\gT^2\rightarrow Q:=(M\setminus S)/\gT^2$ induces an isomorphism on cohomology. Thus the claim implies that all odd Betti numbers of
the orbifold $Q$ vanish.\\[1ex] 
To prove Step 2, first observe that
  $H^*_{\gT^2}(M,M\setminus F)\to H^*_{\gT^2}(M)$ 
is injective in all degrees and, by Step 1 (and its proof), 
surjective in all odd degrees.
By the long exact sequence in equivariant cohomology of the pair $(M,M\setminus F)$, 
we deduce that $H_{\gT^2}^*(M\setminus F)$ 
has vanishing odd Betti numbers. 
Next note that $E=S\setminus F\subset M\setminus F$
is a disjoint union of submanifolds. 
Looking now at the long exact sequence for the pair 
$(M\setminus F, M\setminus S)$, Lemma~\ref{lem:injective} implies that
$H^*_{\gT^2}(M\setminus F, M\setminus S)$ injects into 
$H^*_{\gT^2}(M\setminus F)$. The statement follows.
 \\[1ex]
{\bf Step 3.} The two-skeleton of the action is connected. 
If the fixed-point set $F$ is not empty, then
it and the one-skeleton are connected as well.\\[1ex]
As explained after Step 2, the orbifold 
$Q=(M\setminus S)/\gT^2$ has vanishing odd Betti numbers. 
The fixed-point set of the $(\gT^d/\gT^2)$-action on $Q$ is the image of the two-skeleton of $M\setminus S$ in $Q$. 
By the Chang-Skjelbred Lemma, this image is contained in a single connected 
component of the one-skeleton of the $(\gT^d/\gT^2)$-action on $Q$. 
This in turn implies that all connected components of the two-skeleton 
of the $\gT^d$-action on $M$ are contained in a single connected 
component of the three-skeleton of $M$.

Since every connected component of the two-skeleton contains elements of the one-skeleton, the claim follows once we show the following: 
If $F_1$ and $F_2$ are two fixed-point components of 
two codimension one tori $\gT^{d-1}_1$ and $\gT^{d-1}_2$, 
which are contained in a fixed-point component of a codimension three torus, then they are also contained in a fixed-point component of a codimension two torus. 
But this is obvious since the fixed-point component of the codimension 
three torus is a rational homology sphere.

Suppose now $F_0$ is a fixed-point component of $\gT^d$. 
We claim that every fixed-point component  $E$ 
of a codimension  one torus contains $F_0$. 

Suppose $E$ does not contain $F_0$, and 
let $H\supseteq E$ be a fixed-point component of a codimension two 
torus. Since $H$ is a rational sphere, we would get a contradiction 
if $F\subset H$. Thus $F_0\not \subset H$, and in particular, all 
fixed-point components of codimension one tori contained in $H$ 
also have the property that they do not contain $F_0$.  
Since the two-skeleton is connected, it follows that no 
fixed-point component of a codimension one torus contains $F_0$, 
which is clearly nonsense.\\[1ex]
{\bf Step 4.} If  $F\neq \varnothing$, then $M$ is a rational sphere.\\[1ex]
By Step 3, $F$ is connected and hence a rational sphere. 
Moreover, the one-skeleton is connected, and thus 
any fixed-point component of a codimension one torus contains $F$.
We choose again $\gS^1\subset \gT^d$ such that 
the fixed-point set of $\gS^1$ coincides with the one of $\gT^d$. We claim that the orbifold 
$(M\setminus F)/\gS^1$ has no rational cohomology in odd degrees. Let $p\in F$ be a point. The claim follows from the relative cohomology sequence 
if we can show that $H^*_{\gS^1}(M\setminus p)$ has no cohomology in odd degrees. By Smith theory, we know that $H^k_{\gS^1}(M\setminus p)=0$ for all odd $k>n$. Now the claim follows completely analogously to the proof of Lemma~\ref{lem:gysin}.

We endow $P:=(M\setminus F)/\gS^1$ with the induced 
action of $\gT^{d-1}=\gT^d/\gS^1$. 
We let $E_1,\ldots,E_k\subset P$ denote the fixed-point components of $\gT^{d-1}$. 
We choose fixed-point components $F_1,\ldots,F_k\subset M$ 
of codimension one tori in $\gT^d$ with $E=(F_i\setminus F)/\gS^1$. 
As explained before, $F\subset F_i$ for all $i$.
 By assumption, both $F_i$ and $F$ are rational spheres. This implies that $F_i\setminus F$
 has the cohomology of $\Sph^{2\ell_i-1}$ with $2\ell_i=\dim(F_i)-\dim(F)$. 
 Therefore, $E_i$ has the rational cohomology of $\CP^{\ell_i-1}$. 
 Moreover, the submanifolds $F_i$ intersect pairwise perpendicularly at $F$. Thus 
$$\chi(P)=\sum_{i=1}^k\chi(E_i)=\sum_{i=1}^k\ell_i=\tfrac{n-\dim(F)}{2}.$$ 
 Consider a normal sphere $\Sph^{2k-1}$ of $F$ as a submanifold in $M\setminus F$. Then $2k=n-\dim(F)$. 
 The image $N$ of $\Sph^{2k-1}$ under the projection 
 $P$ is a weighted complex projective space with Euler characteristic 
 $\tfrac{n-\dim(F)}{2}$. It is clear that the inclusion 
 $j\colon N\to P$ induces a surjective map 
 on cohomology $H^*(P)\to H^*(N)$. 
 Since both spaces are simply connected, this map 
 is a rational homotopy equivalence. 
 This in turn implies that the inclusion map of the normal 
 sphere $\Sph^{2k-1}\rightarrow M\setminus F$ is a 
 rational homotopy equivalence. But then $M$ itself is a rational sphere.\\[-1ex]

In view of Step 4, it is clear that we can finish the proof by establishing \\
{\bf Step 5.} If $F=\varnothing$, then $M^n$ is a rational sphere.\\[1ex]
The circle chosen initially acts almost freely, and the orbifold $P=M/\gS^1$ has vanishing odd Betti numbers by the proof of Lemma~\ref{lem:gysin}. If we endow $P$ with the 
induced action by $\gT^{d-1}=\gT^d/\gS^1$, then all fixed-point components of codimension two tori in $\gT^{d-1}$ have the rational cohomology of complex projective spaces since they are images of rational spheres contained in $M$ under the natural projection.

It is easy to see that Lemma~\ref{lem:CP} and Corollary~\ref{cor:codim2CP} remain valid for closed orbifolds with the same proof. Thus $P$ has the rational cohomology of a complex projective space, and hence $M$ must be a rational homology sphere.
%
%
 %
\end{proof}

\begin{corollary} Suppose $M^n$ is a rationally elliptic, positively curved, simply connected, closed Riemannian manifold. If $\gT^8$ acts effectively by isometries on $M^n$ without fixed points, then $M^n$ is a rational sphere.
\end{corollary}

\begin{proof} 
If we choose an almost free subaction of $\gS^1\subset \gT^8$, then the orbifold $P=M/\gS^1$ is rationally elliptic. As in the proof of Step 5, all fixed-point components of codimension two tori in $\gT^7=\gT^8/\gS^1$ have the rational cohomology of complex projective spaces. In particular, $\chi(P)>0$ and thus all odd Betti number of $P$ vanish.
Therefore, we can argue as in proof of Step 5 to get the result.
\end{proof}

A major obstacle in the proof of the above theorem is that there is no a priori assumption
on the equivariant formality of the action. If we change this, then we can relax the remaining assumptions considerably.

\begin{theorem}\label{thm:odd}
Let $M^n$ be an connected, odd-dimensional, oriented, closed manifold with a smooth, effective  $\gT^d$-action 
with $d \geq 2$. Suppose that the action 
is equivariantly formal, that is, the map 
$H^*_{\gT^d}(M)\to H^*(M)$ is surjective. 
 If every fixed-point component of every subtorus with codimension 
 one is a rational homology sphere, then $M^n$ is a rational homology sphere. 
\end{theorem}

\begin{proof}
By the Chang-Skjelbred Lemma, the one-skeleton is connected. 
If the fixed set would not be connected, there would be a fixed-point component $N$ of a  codimension one torus such that the remaining circle action has two odd-dimensional fixed-point components. This is clearly impossible since $N $ is a rational cohomology sphere. Thus the fixed-point set is connected and, by assumption, given by a rational sphere. 
Since the action is equivariantly formal, the total rational Betti number of $M^n$ coincides with the total rational Betti number of the fixed-point set and equals two as well.
\end{proof}

\section{\sectsix}\label{sec:z2codim2}

The main aim of this section is to prove the following theorem
\begin{theorem}\label{thm:z2codim2}  Suppose a torus $\gT^d$ with $d\ge 4$ acts almost effectively on an even-dimensional, connected,  orientable, closed manifold $M^{n}$ with a non-trivial fixed-point set $F$ and the two following properties
\begin{enumerate}
\item[$\bullet$] If $N$ is a fixed-point component of a codimension two torus, 
then $N$ is contained in a fixed-point component $\hat N$ of some subgroup such that the kernel of induced action of $\gT^d$ on $\hat N$ has 
 codimension three and $\hat N$ has the rational cohomology of 
 $\CP^k$, $\HP^k$, or $\Sph^{2k}$ for some $k$.
\item[$\bullet$] If $M'$ is the fixed-point component of a finite 2-group $ \gE\subset \gT^d$ whose induced 
action of $\gT^d$ remains almost effective, then $M'$ is orientable and 
 the above property remains valid for 
the action of $\gT^d$ on $M'$.
\end{enumerate}
If $C$ is the connected component of the one-skeleton at 
any fixed point $p$, then
 $\chi(F\cap C)$ is equal to the Euler characteristic
 of an orientable rank one symmetric space of dimension $n$.
Moreover, any involution $\iota\in \gT^d$ has at most two fixed-point components intersecting $C$. If equality holds, their dimensions sum to $n-w$, where $w\in\{2,4,8,n\}$ denotes the degree of the generator of the cohomology of the model. 

\end{theorem}

\begin{proof}[{\bf Proof of Corollary~\ref{maincor:t7} using Theorem~\ref{thm:z2codim2}}]
Suppose $\gT^7$ acts effectively by isometries  on a positively curved, orientable manifold, which has vanishing odd Betti 
numbers.  
We know from the Addendum to Theorem~\ref{thm:t5} stated at the end of Section~\ref{sec:endgame} that the first property of Theorem~\ref{thm:z2codim2}
 holds for every subtorus of codimension two.
The fixed-point component $M'$ of 
any finite isotropy group  $\gF$ is orientable by Theorem~\ref{thm:FPSstructure}, 
 and thus the induced action on $M'$ is subject to the same conclusion.
 In addition, the Chang-Skjelbred Lemma (Theorem~\ref{lem:ChangSkjelbred}) implies that the one-skeleton is connected, so we recover the 
  Euler characteristic of $M^n$ from the above theorem.  By combining it with Theorem~\ref{thm:t8NoCurvature} under the assumptions of b), Corollary~\ref{maincor:t7} follows. 
\end{proof}

The assumptions of Theorem~\ref{thm:z2codim2} imply that every fixed-point component $L$ of  a codimension two subgroup $\gH\subset \gT^d$
has the rational cohomology of a rank one symmetric space other than the Cayley plane, as long as $\gH/\gH_0$ is a $2$-group. In fact, one can look at a maximal finite $2$-subgroup $\gF\subset \gH$ contained in a finite isotropy group of the isotropy representation of 
$\gH$ at $L$ and pass to the fixed-point component $M'$ of $\gF$ first. 
It is then easy to see that $\gF$ intersects every connected component of $\gH$ and hence that the induced action of $\gH/\gF$ on $M'$ is that of a torus.

\begin{definition}(Assigning $\Z_2$-weights to the graph of the one-skeleton)\label{def:z2graph}
Suppose we are in the situation of Definition~\ref{def:graph} and  $\chi(F\cap C)\ge 3$.
 Assume now in addition that any codimension two subgroup $\gH\subset \gT^d$ for which $\gH/\gH_0$ is a $2$-group has the property that the fixed point components has the rational cohomology of $\CP^k$, $\HP^k$, or $\Sph^{2k}$.
We then assign a second label to each edge of the graph, namely a homomorphism  
$\Z_2^d\rightarrow \Z_2$ where $\Z_2^d\subset \gT^d$. 
We look at the rational weight $r$ of an edge between $i$ and $j$ and let $\gR\subset \gT^d$ denote the subtorus 
whose Lie subalgebra is given by the kernel of $r$.
We  either assign as $\Z_2$-weight  $0$ or the homomorphism 
with kernel equal to $\Z_2^d\cap \gR$
 depending  on
whether or not $F_i$ and $F_j$ are in the same fixed-point component $N'\subset N$ of $\gR\cdot \Z_2^d$. 
If $N$ has $\HP$ type, then 
there are two edges with label $r$ between $i$ and $j$.  
If $N'$ is of $\CP$ type and still contains $F_i$ and $F_j$, 
one of those two  edges  gets the 
$\Z_2$-label $0$ while the other gets as label the homomorphism with kernel $\Z_2^d\cap \gR$.

For the edges from $i$ to itself, we proceed similarly. The ones with rational weight $0$ 
 get the $\Z_2$-label $0$ as well. An edge from $i$ to itself corresponding 
 to a codimension one torus  $\gR$  gets the $\Z_2$-label $0$ if the fixed-point component of $\gR\cdot \Z_2^d$ remains of $\HP$ type.  Otherwise, we assign the homomorphism with kernel equal to $\gR\cap \Z_2^d$. 
\end{definition}

We indirectly used the fact that, in the above situation, a fixed-point component of a codimension one torus cannot have the cohomology of $\Sph^{2k}$ 
 with $k\ge 3$, because otherwise we could use  $\chi(F\cap C)\ge 3$ 
 to get a contradiction for the fixed-point component of a suitably chosen codimension two 
 torus.

\begin{lemma}\label{lem:z2graph}
Suppose we are in the situation of Theorem~\ref{thm:z2codim2} and that $\chi(F\cap C)\ge 3$.
Then the graph has the following properties. 
\begin{enumerate}
\item[a)] 
There exist a subspace $U\subset (\Z_2^d)^*$ and elements $a_{ij}\in (\Z_2^d)^*$ such that $a_{ij}+U$ equals the subset of elements in $(\Z_2^d)^*$ that occur as labels of edges between vertices $i$ and $j$. 
 All labels occur with the same multiplicity $\multtwo$, which is a power of $2$. 
  Moreover, $a_{ij}+a_{ik}+a_{kj}=0\mod U$ 
 for pairwise distinct $i$, $j$, $k$. 
\item[b)] If the $i$-th fixed-point component has positive dimension, then the $\Z_2$-labels for edges from $i$ 
to itself are given by the elements in $U$. 
\end{enumerate}
\end{lemma}

Since $m=\multtwo\cdot 2^{\dim_{\Z_2}(U)}$, it follows that the number $m$ of edges between different vertices 
equals a power of $2$. If $\iota$ is an involution which is not contained in the kernel of all 
elements in $U$, then $\iota$ is contained in the kernel of precisely one half of the elements in $a_{ij}+U$ for all $i$ and $j$.
Therefore, $C\cap M^\iota$ is connected, and the corresponding connected component of $M^\iota$ has dimension $n/2$ by Lemma \ref{lem:straight}. 
In contrast,  $C\cap M^\iota$ has precisely two connected components provided that $\iota$ 
is the kernel of all elements in $U$ but not in the kernel of the action.
Moreover, if then $M^\iota_i$ is the component of $M^\iota$ containing the 
$i$-th component of $C\cap M^\iota$ for $i=1,2$, we have
$\dim(M^\iota_1)+\dim(M^\iota_{2})=\dim(M)-2m$ -- here we do not rule out $M^\iota_1=M^\iota_2$ 
as a possibility.

\begin{proof}
We will often pass to fixed-point components $N$ of subgroups $\gH$ whose 
component group $\gH/\gH_0$ is a $2$-group. Typically, it is clear that we mean by $N$ 
the fixed-point component of $\gH$ at some fixed-point component $F_j$ of $\gT^d$. 
We then consider the connected component  $C'$ of the one-skeleton $C\cap N$ at $F_j$.
We refer to the  graph of $C'$ as the {\it reduced graph} of $C$ with respect to $\gH$. 
If $\gH$ has codimension two and $N$ comes with a remaining almost effective $\gT^2$-action, then the reduced graph has either exactly one edge (resp. two edges) between any of the vertices if $N$ is of $\CP$ (resp. $\HP$ type). 
The possibility that $N$ is of $\Sph$ type is usually ruled out by the fact that $\chi(C'\cap F)\ge 3$.  \\[1ex]
 {\bf Claim 1.} Suppose the fixed-point component  $N$ of a codimension two group $\gH$ 
  has the rational cohomology of $\HP^k$ and that the induced $\gT^2$-action remains almost effective. 
  Then the rational labeling of the reduced graph has a linear model. That is,  
 if $F_{i_1},\ldots,F_{i_k}$ are the fixed-point components inside $N$, then we can find 
  homomorphisms $u_1,\ldots,u_k\colon \Lt\rightarrow \R$ such that the labels 
  of the two edges between $i_a$ and $i_b$ are, up to rescaling, given by $u_a \pm u_b$.\\[1ex]
  The claim is trivial if $N$ is four-dimensional, and 
 it immediately follows from the work of Hsiang and Su \cite{HsiangSu75} if one of the components $F_{i_j}$ 
  has positive dimension.  
   If they are all points, we make use of the fact that $N$ is the fixed-point component 
   of an $\gS^1$-action inside $\hat  N$ as described 
in Theorem~\ref{thm:z2codim2}. Now $\hat N$ has the rational cohomology 
   of a quaternionic projective space as well. But for $\gT^3$-actions on such 
   spaces, the existence of a linear model is known, again by the same work of Hsiang and Su.\\[1ex]
{\bf Claim 2.}
Suppose $a\in (\Z_2^d)^*$ occurs as a label between $i$ and $j$
and $b\in (\Z_2^d)^*$ as a label between $i$ and $k$ for pairwise distinct 
$i$, $j$, and $k$. 
 Then there is an edge between $i$ and $k$ with the label 
$a+b$. \\[1ex]
To prove this, we first consider 
the case that the rational labels  $r_1,r_2\colon \Lt\rightarrow \R$ 
of the two edges are linearly independent.
Our first step in the proof is to reduce the situation to the case that $a$ and $b$ 
are linearly independent over $\Z_2$.
We consider a maximal  codimension two group $\gH$ such that
the Lie algebra of $\gH$ is given $\Ker(r_1)\cap \Ker(r_2)$ and
the component group of $\gH$ consists of two-torsion such that  
both edges can still be found in the reduced graph of $\gH$. 

We pass to the fixed-point component $N$ at $F_j$ of $\gH$. The group $\gT^d/\gH$ still acts on $N$ almost effectively, and by construction, the three involutions  in $\gT^d/\gH$ 
are not in the kernel of the action. The reduced graph  $C'$  with its rational labels is part of the old graph. If $a$ and $b$ are linearly dependent over $\Z_2$, we get a refined $\Z_2$-labeling: 
Namely, we can attach an element of $(\Z_2^2)^*$ to each of the edges where 
$(\Z_2^2)^*$ is the dual space of $\Z_2^2\subset \gT^d/\gH$. 
If we can show that the refined labeling behaves additively  in the way we claim, then the same holds 
for the original labeling.
By our choice of the group $\gH$, it is clear that the refined labels of the two 
edges are linearly independent elements of  $(\Z_2^2)^*$.
By a slight abuse of notation, we still call them $a$ and $b$.

If $N$ is of $\CP$ type, there is only one edge between $j$ and $k$ in the reduced graph, 
and it is labeled by some element in $\{a,b,a+b,0\}$.  
The label cannot be $b$ or $0$, because then
we could  pass to the fixed-point component of the kernel of $b$ in $N$ at $F_j$. 
It would have edges from $j$ to $k$ and from $j$ to $i$ but none from $k$ to $i$. 
Similarly, if the labeling on the opposite edge were $a$, we could pass to the fixed-point component of the kernel of $a$ to see that the combinatorial structure 
of the reduced graph is inconsistent with the simplex structure it must have by Lemma~\ref{lem:graphofmodel}. 

It remains to 
 consider the possibility that $N$ is of  $\HP$ type. We have two edges between each of the three vertices for the reduced graph. We need to rule out that the label on the opposite edges are contained in
$\{ a,b,0\}$.  The labels on the opposite edges cannot be 
a subset of $ \{0,b\}$. In fact, otherwise we could pass to the fixed-point component of the kernel of $b$, and we would see that both edges from $i$ to $k$ are still in the component. At least one edge from $j$ to $k$ as well, but there is at most one edge from 
$j$ to $i$,  a contradiction to Lemma~\ref{lem:graphofmodel}. 
This only leaves the possibility that the opposite edges are labeled with $a$ and $b$. 
If we now pass to the fixed-point component of an involution contained in $\Ker(a)$ or $\Ker(b)$, we see that the component of the one-skeleton in the fixed-point component still contains all three components $F_i$, $F_j$, and $F_k$. Since the reduced one-skeleton must be of $\CP$ type, we deduce that 
for all distinct $u,v \in \{i,j,k\}$, the two edges between $u$ and $v$ of the reduced graph of $N$ are labeled $a$ and $b$.

If not all edges with $\Z_2$-label $a$ have the same rational label, we get a contradiction to 
a previous case by passing through the fixed-point set of the kernel  of $a$. 
Similarly, all edges with label $b$ must  have the same rational label.
But then the rational labeling of the graph does not correspond to a linear action,  contradicting Claim 1.
%

We have proved Claim 2 in the special case that the rational labels $r_1$ and $r_2$ are linearly 
independent. If they are not, then 
  $a$ and $b$ are linearly dependent over $\Z_2$ as well.
So we may assume $b=a$ or $b=0$.  If $a=b=0$, then the graph for the fixed-point component of $\gH\cdot \Z_2^d$ contains $F_i$, $F_j$, and $F_j$ which implies there is an edge with a zero label between vertices $j$ and $k$. If $a\neq 0=b$, then the remaining edge can only have a label $0$ or $a$, and the former case is ruled out after a permutation of $(i,j,k)$ as before.
Thus we only have to rule out the case where all three edges are labeled $a\neq 0$. We choose an edge emanating from $i$ 
whose rational label $r_3$ is not a multiple of $r_1$. 
 If possible, we choose the edge in such a way that it is an edge between vertices $i$ and $j$, after a possible permutation of $i$, $j$ and $k$. Similar to before, we consider a maximal  codimension two subgroup $\gH$ whose
Lie algebra is given by $\Ker(r_1)\cap \Ker(r_3)$ and whose
the component group consists of two-torsion such that 
all three edges can still be found in the reduced graph of $\gH$. 
We then assign again a refined labeling to the reduced graph. 
We continue to use $a$ to denote the refined $\Z_2$-label
 of the two edges with rational labels in $r_1$ and $r_2$,
 and we denote by $c$ the refined $\Z_2$-label of the third edge.  
 As before, $a$ and $c$ are linearly independent over $\Z_2$.

We first consider the subcase that the third edge connects $i$ and $j$.
In the reduced graph of the fixed-point component, we have exactly two edges between 
any two of the vertices $i$, $j$, and $k$. One of them is labeled $a$. 
Since one edge is labeled $c$, we can use the previous case to see that 
the remaining two edges are labeled $a+c$. 
By repeating the argument with $a+c$ in place of $c$, we see that there are also two edges labeled $c$ -- a contradiction.

Thus we may assume that the rational labels of all edges connecting $i$, $j$, and $k$ are multiples of $r_1$.  This implies $m\le 2$.
 Moreover, in the reduced graph of $\gH$, $c$ is the label between $i$ and fourth vertex $l$. 
We find  edges  between $j$ and $l$ and between $k$ and $l$ labeled $a+c$ from the previous case.
If we are in the $\CP$ case, 
 we see a triangle with labels $a$, $a+c$ and $a+c$, in contradiction to a previous case.
 Thus the component is of $\HP$ type.
It follows that there is a second edge from $j$ to $l$ with label 
$c=a+(a+c)$. Repeating this, we see that for all $h\in \{i,j,k\}$ the two edges 
from $l$ to $h$ are labeled $c$ and $a+c$. 
If we pass pass through the fixed-point component of the kernel of $c$,
we see in its reduced graph an edge 
between any two  of the four vertices $i$, $j$, $k$, and $l$.  
By the previous 
case, the label of any edge connecting vertices $\{i,j,k\}$ 
cannot be $c$ 
and hence must be $0$. But this proves our claim.

 We have proved Claim 2. 
 Notice that this  implies that the $\Z_2$-labels 
 of edges between $i$ and $j$ form an affine vector subspace of $(\Z_2^d)^*$.
 To see this consider a third vertex $k$.  
 Let $S_1\subset (\Z_2^d)^*$ denote the set of labels that 
 is assigned to edges  between $j$ and $k$. Similarly, denote by $S_2$ and  $S_3$ 
 the corresponding sets for labels of edges  between $i$ and $k$ and between $i$ and $j$, respectively. We know that $S_1+S_2=S_3$. Clearly all three sets must have the same number of elements. Moreover,  $S_3=b+S_2=a+S_2$ for all $a,b\in S_1$. This in turn implies $a+b+S_2=S_2$ for all $a,b\in S_1$, and hence  $U=S_1+S_1$ 
 has the same number of elements as $S_1$. Therefore, $U$ is a  
 vector space.\\[1ex]
 {\bf Claim 3.} Suppose $a\in (\Z_2^d)^*$ occurs $\multtwo(i,j,a)$ times as a label 
 of edges between $i$ and $j$. Then $\multtwo(i,j,a)$ only attains one non-zero value $\multtwo$,  
 and $\multtwo$ is a power of $2$.\\[1ex]
 After reordering we may assume that $\multtwo_0:=\multtwo(1,2,a)$ is the maximum of the function 
 $\multtwo(\cdot,\cdot,\cdot)$. 
 We let $b$ be a label of an edge from $1$ to $3$. It suffices to show that the label $a+b$ is assigned  $\multtwo_0$ times to edges between $2$ and $3$.

 
 For each codimension two torus whose reduced graph contains a label $a$ between 
 $1$ and $2$ and a label $b$ between $1$ and $3$, we see a label $a+b$ between $2$ and $3$, by Claim 2 (and its proof). 
 We need to make sure that if we see two edges labeled $a$ in the reduced graph, 
 then we see two  edges $a+b$ between $2$ and $3$.
 But if, for one choice of distinct $i, j \in \{1,2,3\}$,  
 the edges between $i$ and $j$ in the reduced graph are labeled differently, then it immediately follows 
 from Claim 2 that they are labeled differently for any choice of 
 $i\neq j\in \{1,2,3\}$. This proves that indeed $\multtwo(\cdot,\cdot,\cdot)$ 
 only attains one non-zero value.

 To show that $\multtwo$ is a power of $2$, we argue as follows. 
 We choose again a label $b$ between $1$ and $3$ and 
 consider the group $\gF\subset \Z_2^d$ of index at most four given 
 by the intersections of the kernels of $a$ and $b$. 
 Let $N$ denote the fixed-point component of $\gF$ at $F_1$. 
 If the kernel of the action of $\gT^d$ has codimension $\le 2$, then $N$ 
 is of  $\CP$ or $\HP$ type, and we see that $\multtwo \le 2$. 
 Otherwise, it follows that there is an involution $\iota\in \gF$ which is contained in a 
 finite isotropy group with respect to the isotropy representation of $\gT^d$ at $F_1$. 
 We consider the fixed-point component $L$ of $\iota$ at $F_1$. 
 The torus action only has a finite kernel $\gE$ and 
 we can replace $M$ by $L$ and $C$ by the connected component $C'$ of $C\cap L$ at $F_1$. We can use the group $\Z_2^d\subset \gT^d/\gE$ to refine our $\Z_2$-labeling. 
 
 The number of edges between different vertices for the graph of $C'$ is still a multiple of $\multtwo$. On the other hand, by induction on the dimension, the number is given by a power of $2$. Thus $\multtwo$ must be a power of $2$ as well.
\end{proof}

\begin{proof}[Proof of Theorem~\ref{thm:z2codim2}]
We argue by induction on the dimension $n$ and we may assume that 
the action of $\gT^d$ is effective rather than almost effective. The induction base $n=2$ is clear since no such action exists. In the induction step, we will make use of Lemma~\ref{lem:repsplitting} and its proof to identify situations where we can find finite isotropy groups 
with non-trivial $2$-torsion. If the $\Z_2$-weights of an isotropy 
representation at a fixed point contain all non-zero elements of a three-dimensional subspace $V\subset (\Z_2^d)^*$, then we can be sure there is an involution $\iota$ in a finite isotropy group which is not contained in the kernel of all elements in $V$. This is due to the fact that then the weights of the representation reduced modulo $2$ no longer have the codimension three property, which would be satisfied according 
to Lemma~\ref{lem:repsplitting} d) if no non-trivial finite isotropy groups
are present. In the proof, it was shown that a failure of the codimension three property actually leads to isotropy groups of even order.

We continue to denote by $m$ the number of different edges between two components. 
If $\chi(F\cap C)=2$, then everything is clear. Thus we may assume $\chi(F\cap C)\ge 3$. 
Then it is easy to see that there must be at least three 
connected components in $F\cap C$, since 
we can then find a fixed-point component  with an almost effective $3$-torus action 
which is acting on some rational $\CP^k$ or $\HP^k$.
There will be no induction on dimension $d$ of the torus, so the application of Theorem \ref{thm:GW} below is only required for the case $d = 4$. Recall that, in the main application of Theorem~\ref{thm:z2codim2}, we have $d=7$. 

Everything follows from Lemma~\ref{lem:z2graph} if we can establish $m\le 4$ with equality only if $(\chi(F\cap C),d)=(3,4)$. 
\\[1ex]
 {\bf Claim 1.} $m\le 4$.\\[-1.5ex]

First, we want to rule out the possibility $(\chi(F\cap C),d,m)=(3,4,8)$. Otherwise,
since $F\cap C$ has at least three connected components, it consists of three isolated points.
Moreover, any fixed-point component of a codimension two torus is either  a sphere or   a rational $\CP^2$.
 In fact, otherwise we would find a codimension two torus which fixes a submanifold that is a rational $\HP^2$, and it could not be contained in a strictly larger fixed-point component of $\CP$, $\HP$, or $\s$ type. Thus we can apply Theorem~\ref{thm:GW} from below 
 to rule out this case. 
 

Second, we consider the possibility that the group $U\subset (\Z_2^d)^*$ is at least three-dimensional. Then $m\ge 8$, and we can find involutions not contained in the kernel of $U$ which belong to a finite 
isotropy group.  Since there is only one fixed-point component intersecting $C$, 
we get a contradiction to our induction assumption unless $m=8$ and $(\chi(F\cap C),d)=(3,4)$, but that case was ruled out earlier.

Third, we plan to get a contradiction in the subcase that 
 the group $U\subset (\Z_2^d)^*$ is at most one-dimensional but we still have $m\ge 8$.  
 Recall that, in the graph, the $\Z_2$-labels on the edges between $i$ and $j$ are the elements of $a+U$ for some $a\in (\Z_2^d)^*$. 
 Thus there is a subgroup $\gH'\cong \Z_2^{d-2}$ of $\gT^d$  fixing all edges. 
  Let  $N$ be the fixed-point component of $\gH'$ at $F_i$ and
 $\gH$ the kernel of the action of $\gT^d$ on $N$. 
 If $\gH$ has dimension $d-2$, a contradiction easily arises since the fixed-point component $N$ must 
 be contained in a bigger one of $\CP$, $\HP$, or $\Sph$ type. 
 If $\gH$ does not have dimension $d-2$, then there is $2$-torsion in its component group, and thus there is an involution $\iota \in \gH$ contained in a finite isotropy group. 
 The fixed-point component $N_2$ of $\iota$ is strictly bigger than  $N$ and comes with 
 an almost effective action of $\gT^d$.  Since the reduced graph for $N$ still has $m\ge 8$ edges 
 between different vertices, this contradicts our induction assumption.

Finally, to prove Claim 1, we may assume that
$U\subset (\Z_2^d)^*$ is  two-dimensional and need to get a contradiction if $m\ge 8$.
  If $d=4$, we choose a group $\Z_2^2$ 
  such that $a(\Z_2^2)=\Z_2$ for all $a\in U\setminus\{0\}$. 
 The fixed-point component $N$ of $\Z_2^2$
contains $C\cap F$, and the reduced graph has $m/4\ge 2$ edges between any two vertices.
 If the isotropy group is two-dimensional, 
 we get a contradiction  from the fact that $N$ needs to be contained in a bigger 
 component which still has cohomology of $\CP$ or $\HP$ type. 
 Indeed, any bigger component must have at least $m/2\ge 4$ 
 edges between vertices of its corresponding graph.
 Thus the kernel of the action is not two-dimensional. Now there is an involution  in a finite isotropy group where the fixed point component $N$ contains $F\cap C$ and still has $m/2 \ge 4$ edges between components. 
 By induction, we have $m = 8$ and $\chi(C\cap F)=3$, but that case was already treated.

 Thus we may assume $d\ge 5$.
 If we look at the isotropy representation of $\Z_2^d$  at some fixed-point component, we can find $d-2$ elements $a_1, \ldots, a_{d-2}\in (\Z_2^d )^*$ projecting to linearly independent 
 elements of  $\Z_2^d/U$ such that all elements in $a_i+U$ occur as subrepresentations 
 of the isotropy representation of $\Z_2^d$. We pass to a three torus 
 whose Lie algebra is given by the kernel of the homomorphisms 
 of the edges whose $\Z_2$-labels are given by $a_2,\ldots,a_{d-2}$. 
 This three torus still has all $7$- non-trivial subrepresentation of $\Z_2^3\subset \gT^3$
 occurring in the isotropy representation. 
 Thus there is a finite isotropy group $\Z_2$. The fixed-point component of $\Z_2$ 
 still has an almost effective $\gT^d$-action, the one-skeleton has at least  three fixed-point components, and the graph has $m/2\ge 4$ edges between vertices. By the induction assumption, this is not possible 
  with $d\ge 5$.  \\[1ex]
{\bf Claim 2.} Assume $m=4$, and let $F_R$, $F_S$, and $F_T$ be three fixed-point components.  
There exist linearly independent elements $u_0,u_1,u_2,u_3\in \Lt^*$ such that, up to scaling, the four weights $r_1^\pm$, $r_2^\pm$ between $S$ and $T$ are given by $u_2 \pm u_3$ and $u_1 \pm u_0$, the weights $s_1^\pm, s_2^\pm$ between $T$ and $R$ by $u_1 \pm u_3$ and $u_2 \pm u_0$, and the weights $t_1^\pm, t_2^\pm$ between $R$ and $S$ by $u_1 \pm u_2$ and $u_3 \pm u_0$.\\[1ex]
%
 %
 %
 %
 %
 We can arrange for $Z:=\ml r_1^\pm,r_2^\pm,s_1^\pm,s_2^\pm \mr$ 
 to be at least four-dimensional.  In fact, for that it suffices to show that 
 $\ml r_1^\pm,r_2^\pm \mr$ is at least three-dimensional. But otherwise the weights are
 contained in a fixed-point component of  a codimension two torus, which in turn cannot be contained in any larger component of a codimension three group whose topology  is still of $\CP$, $\HP$, or $\Sph$ type. Once we prove the statement under the additional assumption $\dim(Z)\ge 4$, it follows that $ \ml r_1^\pm,r_2^\pm\mr $ is four-dimensional, and then the statement follows for all possible triangles.

 We choose a maximal subgroup $\gH$ whose fixed-point component $L$
 still has a one-skeleton with four edges between $S$ and $T$ and between 
 $R$ and $S$.  As before, we equip the  edges with the refined 
 $\Z_2$-labels contained in the dual space of  $\Z_2^h=\gT^d/\gH$
 where $h=\dim(\gT^d/\gH)\ge 4$. 
 By  Lemma~\ref{lem:z2graph}, there is subgroup 
 $U'\subset (\Z_2^h)^*$ and elements $a, b\in (\Z_2^h)^*$ such that the $\Z_2$ 
 labels are given by $a+U'$ and $b+U'$. 
 By construction, the intersection of the kernels of all elements in $a+U'$ and $b+ U'$ 
 is the trivial subgroup of 
 $\Z_2^h$. Since $\dim_{\Z_2}(U') \leq 2$, we find $h=4$, $U'$ is indeed two-dimensional, 
 and $a$ and $b$ project to a basis of $(\Z_2^4)/U'$. 
In particular,  any three of the eight weights $r_1^\pm ,r_2^\pm,t_1^\pm,t_2^\pm$ are linearly 
 independent since this is true for the $\Z_2$-weights.

 We let $\gF\cong \Z_2^2$ denote the intersection of the kernels of $a$ and $b$, 
 and choose an involution $\iota\in \gF$.  Reduction to the fixed-point component $L^\iota$ 
 eliminates exactly two of the edges between any two of the vertices $R$, $S$, and $T$.
 After renumbering, we can  assume that the $\Z_2$-representations 
 assigned to the  edges with rational weights $r_i^+$, $s_i^+$, $t_i^+$ become equivalent 
 when we restrict them to $\gF$ for $i=1,2$. We can also ensure that $s_1^-$, $r_1^-$, and  $t_1^-$ 
 are the weights that can be found in the fixed-point set of $\gF$ 
 and that $s_1^+$, $r_1^+$  and  $t_1^+$ are still tangential to $L^\iota$.

 The vector $t_1^-$ is then a both non-trivial linear combination of $s_1^-$ and $r_1^-$ 
 and of $s_1^+$ and $r_1^+$. This implies that we can find a non-zero vector $u_3\in \ml r_1^\pm \mr\cap \ml s_1^\pm \mr$. 
 After an adjustment of the scaling of these four weights, we find  $r_1^+-r_1^-=s_1^+-s_1^-=2u_3$. 
 We can then choose
 $u_1,u_2\in \Lt^*$ with $r_1^\pm =u_2 \pm  u_3$,  and $s_1^\pm=u_1\pm u_3$.
 The elements $u_1,u_2,u_3$ are linearly independent. 
 As in the proof of Lemma~\ref{lem:HP}, it now follows after rescaling
 that $t_1^\pm=u_1\pm u_2$. In summary, 
  \[
r_1^\pm =u_2 \pm  u_3, \quad s_1^\pm=u_1\pm u_3, \quad  t_1^\pm=u_1\pm u_2.
  \]
  
  This of course implies $s_1^\pm=r_1^\pm+t_1^-$ for $i=1,2$. 
Since 
 $s_2^\sigma$, $r_2^\sigma$ and $t_1^-$ are labels of the reduced graph of a fixed-point component of $\CP$-type for $\sigma\in \{\pm\}$, 
  we may also assume $s_2^\pm=r_2^\pm+t_1^-$
after
 rescaling of  $s_2^\pm$ and $r_2^\pm$.
   Next adjust the scaling of $t_{2}^\pm$   such that we can find $\lambda^\pm $ with 
 $t_{2}^\pm= r_{2} + \lambda^\pm  s_1^- = r_{2}^\pm + \lambda^\pm(u_1-u_3)$. 
 We can also choose $\mu^\pm$ and $\nu^\pm$ 
 with $t_{2}^\pm=\mu^\pm s_2^\pm+\nu^\pm(u_2-u_3)$.
 This gives 
 \begin{eqnarray*} 
0&=&  (1-\mu^\pm)r_{2}^\pm -\mu^\pm(u_1-u_2) - \nu^\pm (u_2-u_3)+\lambda^\pm(u_1-u_3)\\
  &=&  (1-\mu^\pm)r_2^\pm - (\mu^\pm - \lambda^\pm)(u_1-u_2) -(\nu^\pm - \lambda^\pm)(u_2-u_3)
 \end{eqnarray*}
 Since the three weights are linearly independent, $\mu^\pm=1=\lambda^\pm=\nu^\pm$.
 In summary,
 \[
s_2^\pm =r_2^\pm +t_1^-, \quad t_2^\pm= r_2^\pm + s_1^-. 
  \]
 Everything follows if we can show $r_2^++r_2^-= -2u_1$. Indeed, we can then define 
 $u_0$ such that $r_2^\pm=-u_1 \pm u_0$, and the above formulas give the remaining weights as claimed.

 We  know that there must be $\lambda^\pm$ and $\eta^\pm$ with 
 \begin{eqnarray*} 
 \lambda^+ r_{2}^+ +  \eta^-s_{2}^-=t_1^+= \lambda^- r_{2}^-+\eta^+s_{2}^+\quad \mbox{and hence}
 \end{eqnarray*}
 \begin{eqnarray*} 
 (\lambda^+-\eta^+) r_{2}^+ +(\eta^--\lambda^-)r_2^-  =(\eta^+-\eta^-)t_1^- 
 \end{eqnarray*}
 Since the three weights in this expression are linearly independent, it follows that
 $ \eta^\pm = \lambda^\pm =:\lambda$. 
 Similarly, we find  $\mu^\pm$ and $\nu^\pm$ with 
 \begin{eqnarray*} 
 \mu^+ r_{2}^+ +  \nu^-t_{2}^-=s_1^+=\mu^- r_{2}^- + \nu^+t_{2}^+
 \end{eqnarray*}
 which gives
 $ \mu := \mu^\pm = \nu^\pm$. We thus get
 \begin{eqnarray*}
 \lambda r_{2}^-+\lambda  r_{2}^+&=&u_1+u_2-\lambda(u_1-u_2)\\ 
 \mu r_{2}^- + \mu r_{2}^+&=&u_1+u_3 - \mu(u_1-u_3)
 \end{eqnarray*}
 As $u_1,u_2$, and $u_3$ are linearly independent, we deduce $\lambda=\mu=-1$.
 Therefore, $r_{2}^+   + r_{2}^-=-2u_1$ as claimed. 
 We finish the proof of the theorem by establishing \\[1ex]
 {\bf Claim 3.} If $m=4$, then $d=4$ and $n=16$.\\[1ex]
 %
%
From Claim 2, we see easily that the span of the weights of the four edges  between any two fixed-point components 
 is four-dimensional  and independent of the choice of the two components. It follows that $d=4$.
 
We can also deduce that the fixed-point set is discrete. 
 If there would be a fixed-point component $F_1$ of  positive 
 dimension, we could look at all edges between 1 and 2. There is a group of order $2$ that fixes all of them, 
 and using $d=4$ and Claim 2, we see that  the group that fixes all them is finite. 
 Thus we find a fixed-point component  $N$ of an involution which still has an almost effective $\gT^4$-action 
  with a fixed-point component $F_1$ of positive dimension. 
  Then the assumptions of the theorem imply that there is a third fixed-point component inside $N\cap C$, and thus
  we get a contradiction to  our induction assumption.

  Suppose on the contrary there are at least four isolated fixed points in $C$. 
  We call them $Q$, $R$, $S$, and $T$. 
  By Claim 2, the weights from $S$ to $T$ are given by 
  $r_1^\pm=(u_2\pm u_3)$ and $r_2^\pm=(u_1\pm u_0)$. 
  The weights from $S$ to $R$ are given by $t_1^\pm=(u_1\pm u_2)$ and $t_2^\pm=(u_3\pm u_0)$. 
  Suppose now the weights from $S$ to $Q$ are given by 
  $w_1^\pm$ and $w_2^\pm$. 
  We can assume that there is no involution in the kernel of the corresponding twelve $\Z_2$-weights. 
  Therefore, $\dim_{\Z_2}(U) \geq 1$, and we can assume 
  that the weights $(u_1\pm u_2)$, $(u_2\pm u_3)$, and $w_1^\pm$ are tangential to the fixed-point component of an involution.

  We plan to show $ \{ \ml w_1^\pm \mr,\ml w_2^\pm\mr \}=\{\ml u_0,u_2\mr,\ml u_1,u_3\mr\}$.  From Claim 2, 
  it follows for $i\in \{1,2\}$ that  $\ml w_i^\pm\mr$  has non trivial intersection 
  with each of the four vector spaces  $\ml r_j^\pm \mr$ and
   $\ml t_j ^\pm \mr$ for $j=1,2$.     
  If we choose $i\in \{1,2\}$ with 
  $\ml w_i^\pm\mr \subset \ml u_1,u_2,u_3\mr$,  then we use that
  $\ml w_i^\pm\mr$ 
  must have non-trivial intersection with both $\ml t_2^\pm\mr=\ml u_3\pm u_0\mr$ and $\ml r_2^\pm\mr=\ml u_1\pm u_0\mr$. 
  This implies that $\ml w_i^\pm \mr $ is given by $\ml u_1, u_3\mr$.
  If we choose $i\in \{1,2\}$ with 
  $\ml w_i^\pm\mr  \not\subset \ml u_1,u_2,u_3\mr$, then 
  we use that 
   $\ml w_i^\pm \mr$ must have non-trivial intersection with 
   $\ml u_2 \pm u_3\mr $ as well as with $\ml u_1\pm u_2 \mr$. 
   That clearly implies $u_2\in \ml w_i^\pm \mr$.
   Now $\ml w_i^\pm\mr = \ml u_0,u_1\mr$ follows since $\ml w_i^\pm\mr$
   has nontrivial intersection with both $\ml u_3\pm u_0\mr$ and $\ml  u_1\pm u_0 \mr$.

 A contradiction would arise easily if there would be a fifth fixed point $P$.  We could prove the same for the weights between $S$ and $P$, but then the weights between $S$ and $P$ and between $S$ and $Q$ violate Claim 2.
   
   Hence, there are only four fixed points in $C$. 
By Claim 2, the fixed-point component of any codimension two 
torus is either a rational sphere or of $\CP$ type. 
We can use  Theorem~\ref{thm:GW}  from below, to see 
that one of these fixed-point components must be a rational
   $\CP^3$ containing  the entire fixed-point set of $\gT^4$ in $C$. By assumption, there must be a 
   larger fixed-point component $N$ with the cohomology 
   of $\HP^3$. Since $N$ comes with a $\gT^3$-action, the corresponding weights 
   have a linear model.  
   We can assume that the weights $t_1^\pm$, $r_1^\pm$, and $w_1^\pm$ 
   are tangential to  $N$.  The existence of a linear model implies 
   %
 that the  three two-dimensional spaces $\ml t_1^\pm \mr$, $\ml r_1^\pm \mr$, and $\ml w_1^\pm \mr$ contain a common one-dimensional subspace. 
   Now our previous calculation implies $\ml w_1^\pm\mr=\ml u_0\pm u_2\mr$. But then the induced action of $\gT^4$ has a finite kernel, so we have a contradiction.
\end{proof}

\section{\sectseven }\label{sec:GKM}
This section is concerned with two generalizations 
of a result of Goertsches and Wiemeler. The first 
was already used to prove Theorem~\ref{thm:z2codim2} with $d=4$.
The second statement implies that, in even dimensions, the assumptions 
of a) in Theorem~\ref{thm:t8NoCurvature} imply those of b). 
Neither statement is needed for the other main results of the paper.
We will state and prove them in two separate subsections.

\subsection{Combinatorial usage of integral weights} We plan to show

\begin{theorem}\label{thm:GW} 
Suppose $\gT^d$ acts on a connected, even-dimensional, oriented, closed 
manifold $M^n$. If every  subtorus of codimension two only has fixed-point components with the rational cohomology of $\Sph^k$ or $\CP^2$,  then $\chi(F\cap C)$ equals  the Euler characteristic
 of an oriented, rank one symmetric space of dimension $n$.
 Here, as in Theorem~\ref{thm:z2codim2}, $C$ is a connected component 
 of the one-skeleton and $F$ the fixed-point set.
\end{theorem}

This partly generalizes the main result of \cite{GoertschesWiemeler15} and follows along the same lines 
 with a small twist. Some explanations are in order.  Although Goertsches and Wiemeler
 assume an invariant positively curved metric and vanishing odd Betti numbers, the 
curvature assumption is only used to recover the topology of four-dimensional fixed-point components, and the assumption 
on the odd Betti number is not used to determine the combinatorial structure of (each connected component of) the GKM-graph. However, there is something  that does need further explanation: the above theorem does allow for 
fixed-point components of codimension two tori to be rational  
spheres of dimension larger than four, while in \cite{GoertschesWiemeler15}, the action was required to be of type 
GKM$_3$, which means that every three weights of the isotropy representation at any fixed point are linearly independent. 
We explain how one can save the argument.

If $\chi(F\cap C)\ge 3$, then it is easy to see that the action must be of GKM type and 
 there are at least three fixed points $p_0$, $p_1$, and $p_2$ in $C\cap F$.
If a fixed-point component of a codimension two torus  has the rational cohomology of $\CP^2$, then it is equivariantly diffeomorphic to $\CP^2$ endowed with a linear action by the classification of four-manifolds with two torus actions.
We endow the GKM graph with the actual weights of the isotropy 
representation 
 which are determined up to a sign. 
 It follows that the  actual weights of the isotropy representation 
 in  a triangle of the graph corresponding to a $\CP^2$ add up to zero after a possible adjustment of the signs.

As in \cite{GoertschesWiemeler15}, we get, for an edge $i$ between  $p_1$ and $p_2$, 
a bijection $\sigma_i$ from the edges between $p_0$ and $p_1$ 
to the edges between $p_0$ and $p_2$. In fact, if $A$ is an edge 
between $p_0$ and $p_1$, we consider the codimension 
two torus determined by the linearly independent weights of the 
edge $i$ and the edge $A$. The fixed-point component must 
be a $\CP^2$ by our assumptions, and
we find exactly one edge $\sigma_i(A)$ between $p_0$ and $p_2$.
Accordingly, for any two edges $i$ and $j$ between 
$p_1$ and $p_2$, the map $\sigma_i^{-1}\circ \sigma_j$ 
is a permutation of the edges between $p_0$ and $p_1$.
We plan to show that it has order at most two. 
This is equivalent to proving that, for edges $A$ and $B$ between $p_0$ and $p_1$, 
\begin{eqnarray}\label{eq:GW}
\sigma_i(A)=\sigma_j(B) \quad \Rightarrow \quad \sigma_i(B)=\sigma_j(A).
\end{eqnarray}
Here now comes a small addition to the proof:
we choose a scalar product in the dual space $\Lt^*$ 
and first assume $\sigma_i(A)$ is the edge between $p_0$ and $p_2$ whose weight 
has the largest norm with respect to this scalar product.
We choose $B$ with $\sigma_i(A)=\sigma_j(B)$, and for this particular choice of $(A,B)$, 
we can prove the implication \eqref{eq:GW} as follows. 
 We argue by contradiction and assume
  $\sigma_j(A)\neq \sigma_i(B)$. We follow
 the proof of \cite[Sublemma 5.9]{GoertschesWiemeler15}
word for word up to the second-to-last sentence, which 
 establishes a relation among the weights $w_0$, $w_1$, and $w_2$ of the edges $\sigma_i(A) = \sigma_j(B)$, $\sigma_j(A)$, and $\sigma_i(B)$, respectively. 
 It implies $2w_0=\eps_1 w_1+\eps_2 w_2$ with
  $|\eps_i|=1$. 
   We cannot use linear independence of the three 
  weights to get a contradiction, but our choice of $w_0$ 
  implies $2\|w_0\|\ge \|w_1\|+\|w_2\|$ and a contradiction
  arises  since the linear independence of $\{w_1,w_2\}$ implies that the Cauchy inequality is strict.
 Thus we proved the implication \eqref{eq:GW} in the special case 
 that  $\sigma_i(A)$ is the edge whose 
  weight has the largest norm. 
   It follows that $\sigma_j(\{A,B\})=\sigma_i(\{A,B\})$, so we may move on to the complement of $\{A,B\}$.
  Choose
$A'$ and $B'$ with $\sigma_i(A')=\sigma_j(B')$ and the weight of the edge $\sigma_i(A')$ has maximal norm  among all   $A'\notin \{A,B\}$. One can then proceed as before. In the rest of the proof in \cite{GoertschesWiemeler15} the linear independence of three weights which all belong to edges with the same end points never enters and thus the theorem follows.  

\subsection{Proof  of Theorem~\ref{thm:t8NoCurvature} under the assumptions of a)}\label{subsec:GKM}

We have already proved Theorem~\ref{thm:t8NoCurvature} under the assumptions of b) (see Section~\ref{sec:t8sphere}) and have covered the odd-dimensional case in Section~\ref{subsec:odd}.
It remains to show that, in the even dimensions, the assumptions of a) imply those of b). 
By the Chang-Skjelbred Lemma, the one-skeleton is connected and, in view of Theorem~\ref{thm:Bb}, it suffices to show

\begin{proposition}\label{pro:GKM} Let $M^n$ be an connected, orientable, closed manifold with vanishing odd Betti 
numbers. Suppose $\gT^d$ acts effectively on $M^n$ such that every fixed-point component 
of every codimension three subtorus has the rational cohomology of $\CP$, $\HP$, or $\Sph$ type.
Then the number $m$ of edges between different vertices of the graph 
assigned to the one-skeleton satisfies $m \in \{1,2\}$ unless $(m,n)=(4,16)$ or $\chi(M^n)=2$.
\end{proposition}

If we can find a codimension one torus $\gH$ with a fixed-point component of dimension $\ge 4$,
then the result follows from  Lemma~\ref{lem:straight}. 
Thus we may assume that the action of $\gT^d$ on $M^n$ is of GKM type. 
That is, for the isotropy representation of $\gT^d$ at any fixed point, any two weights are 
linearly independent. As mentioned before, the action is called GKM$_3$ if the same holds for any three weights. 
 For GKM$_3$ actions, everything follows from work of Goertsches and Wiemeler \cite{GoertschesWiemeler15}, and the basic idea of the proof is to use the assumption on fixed-point components of codimension three tori 
to carry over their proof. We can assume  a weak form of GKM$_3$, namely, that the fixed-point set 
of any codimension two torus is either of $\CP$ type or a rational $\Sph^4$. In fact, if it was of $\HP^k$  
type with $k>1$ we could easily rule out $m\ge 3$ by passing to the fixed-point set of a suitable codimension three  
torus. If it was $\Sph^{2\ell}$ type with $\ell \ge 3$, we could similarly show $\chi(M^n)=2$, since otherwise there is a codimension three torus with a fixed-point component  not of $\CP$, $\HP$, or $\Sph$ type. \\[-1ex]

In the following we assume \(\chi(M)\geq 3\)
Let \(p_0,p_1,p_2\) be vertices of the GKM graph of \(M\).
We already know that between any two of these vertices there are precisely \(m\) edges.
Denote by \(e^{ij}_{k}\) for \(1\leq k \leq m\) the edges between \(p_i\) and \(p_j\).
\\[1ex]
{\bf Claim 1.} \(e^{01}_k\) induces a natural bijection \(\sigma_k:\{1,\dots,m\}\rightarrow \{1,\dots,m\}\) such that \(e_k^{01}, e_i^{12}, e^{02}_{\sigma_k(i)}\) are edges in the GKM-graph of the fixed set of some codimension two torus. After reordering the \(e^{02}_j\), we can assume that \(\sigma_1\) is the identity. \\[1ex]
Assume the weight of \(e^{01}_k\) is $r_k$. If the weight of \(e_i^{12}\) is $t_i$, we can look at the codimension two torus 
determined by $r_k$ and $t_i$. Its fixed-point component contains $p_0$, $p_1$, and $p_2$, and as explained above, it is of $\CP$ type. Thus, its GKM graph contains a unique edge $e^{02}_{\sigma_k(i)}$ 
from $p_0$ to $p_2$. \\[1ex]
{\bf Claim 2.} For distinct $i,j \in \{1,\ldots,m\}$, the permutation $\sigma_i^{-1}\circ \sigma_j:\{1,\dots,m\}\rightarrow \{1,\dots,m\}$  
is fixed-point free and has order two.  \\[1ex]
Given the weight $t_k$ of \(e^{12}_k\), the codimension three torus determined 
by $r_i$, $r_j$, and $t_k$ must have a fixed-point component $N$ of $\HP$ type.
Here \(r_i\) and \(t_j\) are the weights of \(e^{01}_i\) and \(e^{01}_j\), respectively.
The reduced  graph has exactly two edges from $p_1$ to $p_2$, and it is clear that $\sigma_i^{-1}\circ \sigma_j$
 switches these two edges. 

 In particular, the $\sigma_i=\sigma_1^{-1}\circ \sigma_i$ have order at most two and \([\sigma_i,\sigma_j]=(\sigma_i\circ \sigma_j)^2=1\).
  So we see that $\{\sigma_{i}\mid i=1,\ldots,m\}$ generates a commutative group of permutations of order $\le 2$ that acts transitively on the edges between $p_1$ and $p_2$.
  Indeed, for all edges \(j,j'\) between \(p_1\) and \(p_2\) there is an edge \(i\) with \(\sigma_i(j)=j'\).
  This can be seen by looking at the fixed-point components determined by the codimension three torus determined by the weights of \(j,j'\) and the weight of the edge \(1\) from \(p_2\) to \(p_3\).  This already implies that $m$ must be a power of $2$.

\begin{lemma}\label{lem:gkm_edges}
If $\chi(M) \geq 3$, then the number of edges $m$ connecting any two vertices in the graph satisfies $m \in \{1,2,4\}$.
\end{lemma}

This is the analogue of \cite[Lemma 5.8]{GoertschesWiemeler15}, and the following proof is very similar.

\begin{proof}
By the remarks before Claim 1, we can assume that the weights of the edges \(e^{01}_k,e^{12}_i,e^{12}_j\) for \(i\neq j\) are linearly independent. The proof of \cite[Lemma 5.8]{GoertschesWiemeler15} now goes through with only minor changes, which we summarize here.

We first set up the notation. Let $p_0$, $p_1$, and $p_2$ be vertices in the graph. Denote by $r_i$, $s_i$, and $t_i$ for $1 \leq i \leq m$ the weights of the edges \(e^{01}_i\), $e^{02}_i$, and $e^{12}_i$, respectively. After rescaling the weights, we can assume that
\begin{equation*}
  t_i=r_1-s_i
\end{equation*}
for \(i=1,\dots,k\). Moreover, for each \(j\in\{2,\dots,k\}\), there are \(\epsilon_{ij},\eta_{ij}\in \mathbb{Q}-\{0\}\)  with
\begin{equation}
  \label{eq:1}
  t_i=\epsilon_{ij}r_j+\eta_{ij}s_{\sigma_j(i)},
\end{equation}
where the \(\sigma_j:\{1,\dots,m\}\rightarrow \{1,\dots,m\}\) are the injections from before.
This equation holds for \(j=1\) as well if we define $\epsilon_{i1}= 1$ and $\eta_{i1} = -1$.
As shown before, the \(\sigma_i\) generate an abelian group such that all elements have order at most two. We are done when we can show that this group is generated by at most two elements.
To do so we have to analyze the coefficients in equation (\ref{eq:1}).\\[1ex]
\smallskip
\textbf{Claim 3.}
   For \(j>1\) we have $\eta_{ij}=1$ and $\epsilon_{\sigma_j(i)j}=\epsilon_{ij}$, i.e., 
\begin{align*}
  t_i&=\epsilon_{ij} r_j + s_{\sigma_j(i)}& t_{\sigma_j(i)}&=\epsilon_{ij} r_j + s_{i}.
\end{align*}

\smallskip
  \hspace*{0em}From the above relation (\ref{eq:1}), it follows that
\begin{align*}
  r_1&=t_i + s_i=\epsilon_{ij} r_j +\eta_{ij} s_{\sigma_j(i)} + s_i\\
  r_1&=t_{\sigma_j(i)} + s_{\sigma_j(i)}=\epsilon_{\sigma_j(i)j} r_j +\eta_{\sigma_j(i)j} s_{i} + s_{\sigma_j(i)}.
\end{align*}
By the weak GKM$_3$-condition, the weights $r_j$, $s_{i}$, and $s_{\sigma_j(i)}$ are linearly independent and  the claim follows.\\[1ex]
\smallskip
\textbf{Claim 4.} For $j>2$, we have
$\epsilon_{\sigma_j(i)2}=-\epsilon_{i2}$.
\smallskip

\hspace*{0em}From Claim 3, it follows that
\begin{align*}
  \epsilon_{i2}r_2&=t_{\sigma_2(i)} - s_i=\epsilon_{\sigma_2(i)j} r_j + s_{\sigma_j\circ\sigma_2(i)} - s_i\\
  \epsilon_{\sigma_j(i)2} r_2&=t_{\sigma_j(i)} - s_{\sigma_2\circ\sigma_j(i)}=\epsilon_{ij} r_j + s_{i} - s_{\sigma_j\circ\sigma_2(i)}.
\end{align*}
Now the weak GKM$_3$-condition implies that \(\epsilon_{i2}=-\epsilon_{\sigma_j(i)2}\). This shows the claim.
\smallskip

Because we know that $m$ is a power of two, if $m>2$, then $m\geq 4$, so we may choose \(j>j'>2\). 
Then, by applying Claim 4, we have
\[
\epsilon_{\sigma_j(\sigma_{j'}(i))2}=-\epsilon_{\sigma_{j'}(i)2}=\epsilon_{i2}.
\]
\hspace*{0em}From the proof of the transitivity of the action of the \(\sigma_i\), it follows that there is a \(j''\) such that \(\sigma_{j''}(i)=\sigma_j\circ\sigma_{j'}(i)\).
Since $\sigma_j\circ\sigma_{j'}$ does not have fixed points, it follows, again by Claim 4, that $j''=2$.  Because this holds for each \(i\), it follows that \(\sigma_{2}=\sigma_j\circ\sigma_{j'}\).
 This shows that the group of \(\sigma_i\)'s is generated by \(\sigma_2\) and \(\sigma_3\). Hence, we have \(m\leq 4\).
\end{proof}

To complete the computation of the unlabeled graph, it remains to show that $m= 4$ occurs only if $\chi(M) \le 3$. This is proved in the two following lemmas.

\begin{lemma}
\label{sec:proof-lemma-refl}
Assume \(\chi(M)\geq 3\) and \(m=4\). Let $p_0$, $p_1$, and $p_2$ be distinct vertices in \(\Gamma\), and let $r_1,\ldots,r_4$ denote the weights of the edges between $p_1$ and $p_2$. For $h \in \{1,2\}$, let $s_{hi}$ for $1 \leq i \leq 4$ denote the weights of the edges between $p_0$ and $p_h$. After possibly reordering and rescaling, we have
\[s_{1i} = \frac 1 2 \sum (-1)^{\delta_{ij}} r_j~~~\mathrm{and}~~~s_{2i} = \frac 1 2 \sum (-1)^{\delta_{1j} + \delta_{ij}} r_j,\]
where $\delta_{ij} = 1$ if $i = j$ and $\delta_{ij} = 0$ otherwise. In particular, the $r_i$ are linearly independent.
%
\end{lemma}

\begin{proof}
The formulas for the $s_{hi}$ are shown in the proof of \cite[Lemma 5.14]{GoertschesWiemeler15} relying on the Sublemmas 5.9--5.11 therein. As pointed out before, the proofs go through in our situation with only minor changes:
  By Claims 1, 3, and 4 from above, we may assume that the following relations hold:
 \begin{align}
   \label{eq:2}s_{11}&=s_{21}-r_1=s_{22}+r_2\\
   \label{eq:3}s_{12}&=s_{21}-r_2=s_{22}+r_1\\
   \label{eq:4}s_{13}&=s_{21}-r_3=-s_{22}+r_4\\
   \label{eq:5}s_{14}&=s_{21}-r_4=-s_{22}+r_3
 \end{align}
(The translation from the notation there to the notation here is \(t_i\mapsto s_{1i},r_i\mapsto s_{2i}, s_i\mapsto r_i\))
By adding equations \eqref{eq:3} and \eqref{eq:4}, we get that
\begin{equation*}
  s_{21}=\frac{1}{2}(r_1+r_2+r_3+r_4).
\end{equation*}
Hence, the \(s_{1i}\) are of the form \(\frac{1}{2}(-r_i+\sum_{j\neq i} r_j)\). By Claim 3 in the proof of Lemma~\ref{lem:gkm_edges}, we may assume that
\begin{equation*}
  s_{11}=s_{21}-r_1=s_{22}+r_2=s_{23}+r_3=s_{24}+r_4.
\end{equation*}
Thus, for \(i>1\),
\begin{equation*}
  s_{2i}=\frac{1}{2}(-r_1-r_i+\sum_{j\neq i, 1} r_j).
\end{equation*}
This proves the formulas for the \(s_{ij}\).

For the last claim, first note that $r_1$, $r_2$, $r_3$, and $s_{11}$ are linearly independent. By the formula for $s_{11}$, it follows that the $r_i$ are linearly independent. 
%
\end{proof}

The proposition now follows from
\begin{lemma}\label{lem:GKM-OP2}
If $\chi(M) \geq 3$ and $m = 4$, then $d=4$ and $\chi(M) = 3$.
\end{lemma}

If we put
\begin{align*}
  u_1&=\frac{1}{2}(r_1-r_2)&u_2&=\frac{1}{2}(r_1+r_2)&u_3&=\frac{1}{2}(r_3-r_4)&u_0&=\frac{1}{2}(r_3+r_4)
\end{align*}
in the situation of the previous lemma, then we see that we have achieved the same situation as in Claim 2 of the proof of Theorem~\ref{thm:z2codim2}. Now $d=4$ follows as in Claim 3 of that proof. 
Moreover, Lemma~\ref{lem:GKM-OP2} follows from Theorem~\ref{thm:GW} unless a fixed-point component of a codimension two torus is given by a rational $\CP^k$ with $k\ge 3$. But then a fixed-point component of a codimension three torus has the rational cohomology of $\HP^l$ with $l\ge 3$. The contradiction now arises as in Claim 3 of the proof of Theorem~\ref{thm:z2codim2}.

\end{document}